\documentclass{article}

\usepackage[english]{babel}
\usepackage{mysty}
\usepackage[margin=1in]{geometry}

\usepackage{amsmath}
\usepackage{graphicx}

\title{Efficient Optimization on the Euclidean Sphere:\\Riemannian Gradient Alignment as a Unifying Principle}
\author{Puqian Wang\\
UW-Madison\\
\texttt{pwang333@wisc.edu}
\and Nikos Zarifis\\
MIT\\
\texttt{zarifis@mit.edu}
\and Jelena Diakonikolas\\
UW-Madison\\
\texttt{jelena@cs.wisc.edu}
}
\date{}

\begin{document}
\maketitle

\begin{abstract}
Euclidean sphere-constrained optimization problems represent perhaps the most basic class of nonconvex-constrained problems for which there exist efficient optimization algorithms for specific problem instances, such as the eigenvalue problems and learning problems arising in the recent literature. However, when and why we can efficiently optimize on the sphere using simple first-order methods is poorly understood. We introduce Riemannian Gradient Alignment (RGA) as a unifying structural condition enabling fast convergence. RGA is a
parameterized directional error bound requiring a tangent vector to 
sufficiently align with a target solution vector. We prove that  Riemannian
Gradient Descent with diminishing step sizes converges linearly when RGA holds with order $r=1$, and at rate
$O(k^{-1/(2r-2)})$ for $r>1$. The result applies even to tangent vector fields
that are not gradients of any objective.

Our main technical contribution is a family of derivative-based sufficient conditions 
for RGA. We prove these conditions for a range of learning problems, including 
multi-instance learning/max pooling, Gaussian halfspaces, single-index and
generalized linear models, phase retrieval, complete orthogonal dictionary
learning, and (sparse) PCA. The provided examples span distributional families that include standard Gaussian, conditionally Gaussian, and more general structured distributions. The provided guarantees cover objectives that are generally not 
(geodesically) convex, may be nonsmooth, and may arise from discrete distributions. These results establish tractability of problem classes where no prior error bound conditions were known, as well as 
 unify, strengthen, and generalize several prior results. 
\end{abstract}

\newpage
\tableofcontents
\newpage

\section{Introduction}
Optimization problems on the Euclidean sphere $\spd := \{\vx \in \sR^d: \norm{\vx}_2 = 1\}$,
\begin{equation}\label{problem:opt-on-sphere-general}
    \min_{\vx\in\spd} f(\vx), \tag{P}
\end{equation}
represent perhaps the simplest class of problems with nonconvex constraints for which there are efficiently solvable instances. Although nonconvex, Euclidean sphere is a smooth manifold and among the simplest examples of Riemannian optimization problems where generalizations of Euclidean concepts of distance and gradient updates transfer almost seamlessly. As such, optimization problems on the sphere feature prominently in the literature, arising in applications such as  
eigenvalue/eigenvector computation for real symmetric matrices~\cite{zhang2016riemannian,ADVA21,alimisis2022geodesic} {and their generalizations to sparse Principal Component Analysis (PCA)~\cite{jolliffe2003modified,zou2006sparse,zou2018selective,chen2020proximal} and robust PCA~\cite{zhang2018robustPCA}}, dictionary learning~\cite{sun2015complete,bai2018subgradient,gilboa2019efficient,Li2019WeaklyCO,qu2019analysis}, 
robust subspace recovery~\cite{zhu2018DPCP,Li2019WeaklyCO,Zhu2019LinearConvergenceGrassman,ding2021DPCP},
learning single-index models and generalized linear models~\cite{BenArous2021OnlineSGD,damian2023smoothing,ZWDD2024,WZDD24b}, and (as a special case of learning generalized linear models) learning halfspaces (a.k.a.\ binary linear classification)~\cite{SB14,DKTZ20c,DDKWZ2023}.

Despite substantial progress in Riemannian optimization and in optimizing different objective functions on the sphere, there still remains a significant gap in our understanding of what kind of objective functions are amenable to being efficiently minimized on the sphere. Additionally, the authors' own recent work \cite{WZDD2023,WZDD24b,ZWDD2024,zarifis2025robustly,wang2026robustly} has revealed several examples of learning problems for which direct generalization of (projected) gradient descent to sphere-constrained problems, with suitably selected step sizes, exhibits provable linear convergence despite the apparent nonconvexity of the objective (in the traditional or Riemannian sense) and often even nonsmoothness. 

These recent results motivated the present investigation of the underlying problem geometry that enables such results. Our guiding principle is the development of theory akin to classical Euclidean optimization, focusing on simple algorithms (mainly, Riemannian Gradient Descent) and familiar concepts from classical optimization. As our conceptual contribution, we introduce the Riemannian Gradient Alignment (RGA) property parameterized by $r \geq 1, \mu > 0$, as a unifying concept that defines classes of optimization problems on the sphere for which we prove that diminishing-step Riemannian Gradient Descent (RGD; see \Cref{alg:RGD}) is guaranteed to converge to a solution at either linear rate (for $r = 1$) or fast sublinear rate ($(1/k)^{1/(2r -2)}$ for $r > 1$), depending on the parameters defining the problem class. 

Our main technical contributions are two-fold: (1) as a parallel to classical optimization theory developed for Euclidean (and closely related) spaces, we introduce derivative-based sufficient conditions for optimization problems on the sphere to satisfy RGA; and (2) we prove the RGA condition, providing fully specified RGA parameters, for a wide range of learning problems, by verifying the introduced sufficient conditions. Together with the established convergence guarantees for diminishing-step RGD under RGA, these results imply fast (linear or sublinear, see the paragraph above) convergence of diminishing-step RGD on the considered learning problems. Importantly, many of the considered problems are neither convex (in Euclidean or geodesic sense) nor continuously differentiable, demonstrating the power of RGA as a unifying property.

\subsection{Contributions and Technical Overview}

The central goal of this work is to identify general structural conditions under which a simple first-order
method can efficiently optimize an objective over the Euclidean sphere, even when the objective is nonconvex,
nonsmooth, or not explicitly available. Toward this goal, we introduce a parameterized directional error bound (RGA), establish convergence guarantees for a
diminishing-step RGD method under this condition, and develop derivative-based
conditions that allow RGA to be proved for broad classes of learning problems. 

\paragraph{Riemannian Gradient Alignment}
We define RGA as follows. 
\begin{definition}[Riemannian Gradient Alignment--RGA]\label{def:sharp}
    Fix a target unit vector $\vx^*\in \spd$. For $\vx\in\spd$, we say that the vector field $\vg(\vx)\in\mathrm{T}_{\vx}\spd$, $\mathrm{T}_{\vx}\spd = \{\vy \in \R^d: \vy \text{ is orthogonal to }\vx\},$ satisfies the RGA property with parameters $(\mu,r)$ at  $(\vx,\vx^*)$ if $\vg(\vx) \neq \vzero$ whenever $\vx \neq \vx^*$ and \begin{equation}\label{eq:sharp}
        \vg(\vx)\cdot\vx^*\leq -\mu  \sin^r(\theta(\vx,\vx^*)) \|\vg(\vx)\|_2 ,
    \end{equation}
    where $\theta(\vx, \vx^*) = \arccos(\vx \cdot \vx^*).$ 
    In this case, we say that the vector field $\vg(\vx)$ is $(\mu,r)$-RGA at $(\vx,\vx^*)$, and in particular, when $r=1$, we simply say that $\vg(\vx)$ is $\mu$-RGA at $(\vx,\vx^*)$.

    If, {at $(\vx,\vx^*)$, every measurable selection $\vg(\vx)$ generated from the subdifferential of a minimization objective $f(\vx)$ is $(\mu,r)$-RGA, then we say that $f$ is $(\mu,r)$-RGA (or simply $\mu$-RGA when $r = 1$) at $(\vx,\vx^*)$.}
    
    If there exists a fixed $\vx^*$ such that for any vector $\vx$ in the subset $\cX\subseteq\spd$, $f$ is $(\mu,r)$-RGA at $(\vx,\vx^*)$, then we say $f$ is $(\mu,r)$-RGA on $(\cX,\vx^*)$. 
\end{definition}

RGA resembles and is related to inequality-based conditions bounding the correlation between the gradient of an objective function and a target solution, studied under names ``restricted secant inequality'' \cite{ZY2013}, ``Riemannian Regularity Condition'' \cite{Zhu2019LinearConvergenceGrassman}, and  ``(alignment) sharpness'' \cite{WZDD2023,WZDD24b,ZWDD2024,zarifis2025robustly,wang2026robustly}. Compared to these past works, RGA can be seen as both a specialization (e.g., from Euclidean spaces \cite{ZY2013,WZDD2023} to the sphere) and generalization (e.g., using a generally smaller Riemannian gradient in place of the Euclidean one \cite{ZWDD2024} and normalizing it to a unit-norm vector, different from all the aforementioned past work). 

Observe here that the RGA property can be defined for an arbitrary vector field $\vg(\vx)$, not necessarily associated with a minimization problem. In that sense, RGA is a general property that can plausibly be used to study other related problems on the sphere, such as root-finding problems or fixed-point equations. Further, even when the underlying problem we want to solve is a minimization problem expressible as \eqref{problem:opt-on-sphere-general}, $\vg(\vx)$ need not be the (Riemannian) (sub)gradient of the objective. Although it may seem unintuitive that considering such vector fields would help in any way for minimization problems, this higher level of generality turns out crucial for some problems. In particular, we point to \cite{zarifis2025robustly}, where it was argued that for the (generalized linear model) learning problem considered in that work, the Riemannian gradient of the optimization objective may, in fact, negatively correlate with the target directions (meaning that it is not a reliable ``signal,'' or, in other words, that it would guide gradient-based methods \emph{away} from target solutions).

Another property worth pointing out right away is that RGA refers to $\vx^*$ as some fixed \emph{target} vector rather than a \emph{solution} to any given problem. Of course, to give a meaning to an algorithm converging to some $\vx^*,$ we would often want to provide further specification. For instance, the `target vectors' can be local minima, local stationary points, or any vectors with small error (not necessarily local minima or stationary points); the appropriate definition is problem-specific. 
Relatedly, the convergence result we establish for diminishing-step RGD (\Cref{lem:descent-rsi}) does not require the target parameter $\vx^*$ to have any specific property such as stationarity. Rather, $\vx^*$ can be any fixed vector for which the vector field $\vg(\vx)$ satisfies RGA along the trajectory. 
For example, in \cite{WZDD2023,ZWDD2024,WZDD24b,zarifis2025robustly,wang2026robustly}, there is a spherical cap in which every vector has the loss (objective) value smaller than $C\,\opt$ for a constant $C > 1$, where $\opt$ denotes the minimum loss value. The RGA property established in these prior works holds with respect to any vector $\vx^*$ in this region, for $\vx$ outside the spherical cap containing $C\, \opt$ solutions.

Finally, in comparison to standard notions such as geodesic (strong) convexity (see \Cref{sec:prelims} for relevant definitions), which can hold \emph{locally} on the sphere, we argue that RGA is a much weaker property. RGA is implied by geodesic strong convexity (see \Cref{claim:str-cvx-implies-RGA}) and it is possible for RGA to hold on a spherical cap of constant radius around a solution $\vx^*$, whereas geodesic convexity fails on a much smaller spherical cap region, of radius $O(1/\sqrt{d})$ (see \Cref{claim:lattice-relu-glm-not-geo-convex}).

\paragraph{Convergence of Diminishing-Step Riemannian Gradient Descent}  To establish the sufficiency of RGA for efficient optimization on the sphere,  in \Cref{sec:sharpness-RSI}, we analyze an RGD-style update, of the form
\[
    \vx_{k+1}
    =
    \frac{\vx_k-\eta_k \vg(\vx_k)}
         {\|\vx_k-\eta_k \vg(x_k)\|_2}, \; k \geq 0,
\]
where $\vx_0$ is initialized on the sphere and $\eta_k$ is a suitably chosen sequence of diminishing step sizes. In particular, the analyzed schedule of step sizes $\eta_k$ can be seen as an estimating sequence of the iterate distance to the target vector $\vx^*.$ 

Under the assumption that RGA holds on the algorithm trajectory $\vx_0, \vx_1, \vx_2, \dots,$ for a fixed target $\vx^* \in \spd$ and for fixed RGA parameters $\mu > 0, r \geq 1,$ we establish the following convergence guarantees. Suppose that $\theta(\vx_0, \vx^*) \leq \Bar{\theta} < \pi/2$ for some $\Bar{\theta} > 0.$ When $r=1$, the iterates converge linearly:
\[
    \frac{1}{2}\norm{\vx_k - \vx^*}_2 = \sin\!\left({\theta(\vx_k,\vx^\ast)}/{2}\right)
    \leq
    \sin\!\left({\bar\theta}/{2}\right)
    \exp\bigl(-\Omega(\mu^2 k)\bigr).
\]
For $r>1$, the iterates converge sublinearly as
\[
    \frac{1}{2}\norm{\vx_k - \vx^*}_2 = \sin\!\left({\theta(\vx_k, \vx^\ast)}/{2}\right)
    =
    O\!\left(
        \left(
            \frac{1}{\mu^2(r-1)(k+1)}
        \right)^{\!1/(2r-2)}
    \right).
\]
In particular, for $r = 2,$ the convergence rate is of the order $1/\sqrt{k+1}.$ 

Moreover, we argue that if RGA holds on a spherical cap centered at $\vx^*$ and $\vx_0$ is initialized in that spherical cap, then all RGD iterates described above are guaranteed to remain in the same spherical cap on which RGA holds. This makes the RGA algorithm trajectory requirement trivial (under suitable initialization) for all our examples provided in \Cref{sec:examples}, as in all these examples we prove the RGA holds on a spherical cap centered at a solution $\vx^*.$ Moreover, in most examples, RGA is proven for all $\vx$ such that $\theta(\vx, \vx^*) \in [0, \pi/2)$, in which case random initialization suffices by the standard concentration of measure on the Euclidean sphere: with probability one, $\vx_0 \in \spd$ chosen uniformly at random from the sphere satisfies $|\theta(\vx_0, \vx^*)| \in [0, \pi/2)$. Thus, it suffices to run the algorithm initialized at $\vx_0$ and $-\vx_0.$  

We note here that while similar convergence guarantees have been established for diminishing-step RGD when $r = 1$ under comparable assumptions to RGA (see \Cref{sec:related work} for a detailed comparison to prior work), the result for $r > 1$ (which is more challenging to establish, see the proof of \Cref{lem:descent-rsi}), is new, to our knowledge. We remark here that our sparse PCA example in \Cref{sec:spiked-sparse-pca} provably satisfies RGA with $r =2$ but not with $r = 1$ (see \Cref{lem:spiked-sparse-pca-cannot-have-lower-order} for details).  

\paragraph{Derivative-based Sufficient Conditions for RGA} Our main technical contribution is a collection of conditions that make RGA verifiable without directly
estimating the angle between a vector field $\vg$ and an unknown target $\vx^*$. Our sufficient conditions are provided for two general settings: (1) encompassing stochastic optimization problems most naturally mapped to regression, and (2) corresponding to general vector fields $\vg,$ not necessarily associated with an optimization problem. 

For our first set of sufficient conditions, we consider stochastic optimization problems on the sphere, whose objective $f(\vx)$ in \eqref{problem:opt-on-sphere-general} is expressible as
\[
    f(\vx; \vx^*)
    =
    \E_{\mZ\sim \D}
    \big[
        \ell(\mZ\vx,\mZ\vx^*;\vx,\vx^*)
    \big],
\]
where $\mZ \in \R^{m \times d}$ is a random matrix drawn from a probability distribution $\D$ and $\ell$ is a penalty (or loss) function. We only consider the case where $m > 1$ for the standard Gaussian distribution $\D = \cN(\vzero, \mI)$; the sufficient conditions we provide for more general distributions are for $m = 1.$  
Our sufficient conditions require that $\ell$ has a nontrivial dependence on the first argument only, while it may be independent of the remaining arguments; in fact, this will be true for the Complete Orthogonal Dictionary (COD) learning example in \Cref{sec:ex:general Gaussian}. The dependence on the remaining arguments is included to cover a broader set of examples in \Cref{sec:examples}. 

Our sufficient conditions in this case establish RGA for a notion of Riemannian correlation vector field that we define as 
\begin{equation*}
     \vg_\ell(\vx, \vx^\ast)
    :=
    \E_{\mZ\sim\D}
    \big[
        P_{\perp \vx}\mZ^\top
        \partial_1\ell(\mZ\vx,\mZ\vx^*;\vx,\vx^*)
    \big],
\end{equation*}
where $P_{\perp \vx}$ denotes the projection onto the subspace of $\R^d$ orthogonal to $\vx$ and $\partial_1$ denotes the gradient of $\ell$ with respect to its first argument (or a measurable selection from $\ell$'s Clarke subdifferential when the function is possibly nonsmooth; see \Cref{sec:prelims} for relevant definitions). The vector field $\vg_\ell$ may or may not coincide with the objective's Riemannian (sub)gradient, depending on the problem setting.

Our sufficient conditions can be seen as ``curvature'' conditions, based on a subset of ``second derivatives'' of the expected loss $\ell.$ Here, we allow for functions that are not strictly twice (or sometimes even once) differentiable, through the notions of Clarke subdifferentials and weak distributional derivatives (see \Cref{sec:prelims} for definitions). In particular, our sufficient conditions are based on what can be interpreted as the second derivative of $\ell$ with respect to (w.r.t.) the first argument ($\partial^2_1 \ell$) and the sequential derivative w.r.t.\ the first two arguments ($\partial_2 \partial_1 \ell$, where differentiation is carried out w.r.t.\ the first then w.r.t.\ the second argument; henceforth, this mixed derivative is referred to as the ``joint curvature''). We highlight here that either of these high-dimensional derivatives may determine the RGA property. 
In particular, {for the COD learning problem, spiked-sparse covariance PCA problem, and the Rayleigh quotient problem discussed in \Cref{sec:ex:general Gaussian}}, $\ell$ only depends on the first argument, so the joint curvature is zero and thus RGA is determined by a ``second-order'' condition for the first argument. In all remaining examples, it is the joint curvature w.r.t.\ the first and the second argument that determines RGA; the ``second-order'' derivative with respect to the first argument is either zero or dominated by the contribution of the joint curvature in these cases. 

When $m \geq 1$ and $\mZ$ is a matrix whose rows are i.i.d.\ standard Gaussian, the sufficient conditions simplify considerably. In this case, it suffices that the expectation of the ``joint curvature'' of $\ell$ (interpreted in the weak sense discussed in the previous paragraph) has strictly negative \emph{trace}. In other words, if we think about this ``joint curvature'' as the Hessian of the function, in parallel to Euclidean optimization, this corresponds to a much weaker condition than strict convexity (which would require the minimum eigenvalue of the Hessian, if it exists, to be positive). In this setting, this mild sufficient condition suffices for $\vg_\ell$ (and the associated regression problem) to satisfy RGA with $\mu = r = 1.$  

Our second set of sufficient conditions applies to arbitrary vector fields $\vg(\vx, \vy) \in \mathrm{T}_\vx\spd$ such that $\vx, \vy \in \spd$. Roughly, if $\vg(\vx,\vx)=\vzero$ and every Clarke Jacobian of the map $\vy\mapsto \vg(\vx,\vy)$ has uniformly negative
curvature in the target tangent direction, while its operator norm is uniformly bounded, we have that 
$\vg(\vx,\vx^*)$ is (locally) RGA. This result does not require a stochastic representation or even an underlying
objective, and it allows us to treat nonsmooth loss functions corresponding to generalized linear models under discrete covariate distributions (see \Cref{sec:ex:general distributions}).  

\paragraph{Learning and Estimation Examples} As our final set of contributions, we prove sufficient conditions (and thus RGA) for a range of learning and estimation problems, obtaining the following consequences.

\vspace{5pt}\noindent
 \emph{Standard Gaussian designs.}
We show that the population loss for Gaussian Multi-Instance Learning with max pooling (see \Cref{prob:mil}) is $1$-RGA on the entire target hemisphere for a broad class of locally Lipschitz penalties, including standard hinge-type and
logistic penalties. The RGA parameter is independent of the number of instances in each ``bag.'' Learning a
Gaussian halfspace is recovered as the ``single-instance per bag'' special case. We also establish $1$-RGA for several
 Gaussian Single-Index Model (SIM) objectives (\Cref{prob:glm+sim}, Part 2): the squared loss for arbitrary nonconstant
Gaussian-square-integrable target links; the nonsmooth $L_1$ loss for strictly increasing, locally Lipschitz
links; and an $L_1$ phase-retrieval objective. The proofs characterize the corresponding best-fitting links and
show that their joint curvature is strictly negative.

\vspace{5pt}\noindent \emph{General and conditionally Gaussian designs.}
We extend the Gaussian sufficient condition to conditionally Gaussian distributions with random
positive-semidefinite covariance matrices. This formulation allows us to analyze the Gaussian-Bernoulli model
used in COD learning (\Cref{prob:dictionary}) and recover an explicit local RGA guarantee from~\cite{bai2018subgradient,Zhu2019LinearConvergenceGrassman} by decomposing the
population field into its ``signal'' and ``noise'' components. In this example, $r = 1.$ We further argue that deterministic problems, such as PCA (or Rayleigh quotient optimization, see \Cref{prob:eigenvector-problem}) and spiked sparse PCA (\Cref{prob:spiked-sparse-pca}) can be cast as regression problems with general-covariance Gaussian covariates. As a consequence, we prove the RGA property for both problems by verifying our sufficient conditions. In particular, for minimum\footnote{For our results, the matrix is allowed to be indefinite, so the maximum eigenvalue/eigenvector problem corresponding to PCA is recovered by symmetry, applying the result to $-\mA.$} eigenvalue/eigenvector problems defined w.r.t.\ a real symmetric matrix $\mA \in \R^{d \times d}$ with eigenvalues $\lambda_1 \geq \lambda_2 \geq \dots\geq \lambda_r = \dots = \lambda_d,$  we prove that the Rayleigh quotient $\vx^\top\mA \vx$ is RGA with $r = 1$ and $\mu$ scaling with $\sqrt{\frac{\lambda_{r-1} - \lambda_d}{\lambda_1 - \lambda_d}}$, {recovering the results from~\cite{zhang2016riemannian,ADVA21,alimisis2022geodesic}}. For lasso-regularized Rayleigh
quotient in a spiked sparse-covariance model, we prove that the objective is locally RGA with $r = 2$, near the planted sparse
vector. Such a local error bound result was not previously known, to our knowledge. We further exhibit a geodesic path approaching the target along which no uniform RGA inequality with
$r<2$ can hold. Thus, the higher-order ($r = 2$) regime is intrinsic to this example, rather than an artifact of the
analysis, and the $r>1$ convergence result is necessary to cover its local geometry.

\vspace{5pt}\noindent \emph{Well-behaved distributions.}
For a class of ``well-behaved'' distributions (namely, isotropic distributions whose two-dimensional marginals have a density bounded below on a bounded square; see \Cref{def:well-behaved}), 
we study SIMs (\Cref{prob:glm+sim}, Part 2) over a class of structured ($(a,b)$-unbounded, see \Cref{def:ab-unbounded-link})  activations, which contains activations like rectified linear unit (ReLU), its ``leaky'' variants (LeakyReLU), and exponential linear unit (ELU). For every smooth strongly convex penalty, we construct a Riemannian correlation vector field that is RGA with $r = 1$ and  constant $\mu > 0$ (dependent on the problem parameters), 
throughout the target hemisphere. This extends the squared-loss Euclidean gradient alignment analysis from prior work \cite{ZWDD2024} 
to the spherical setting, allows for non-quadratic penalties, and provides the error bound parameter $\mu$ that is substantially larger than in \cite{ZWDD2024}. This example also illustrates the usefulness of RGA defined for vector fields that
are not necessarily equal to the Riemannian gradient of the corresponding problem objective.

\vspace{5pt}\noindent \emph{Near-isotropic and discrete distributions.}
For generalized linear models (GLMs, \Cref{prob:glm+sim}, Part 1) with $(a,b)$-unbounded activations, we prove local RGA guarantees under mild distributional assumptions that cover even discrete distributions like the uniform distribution on $\{-1, 0, 1\}^d$. The result applies to more general loss functions than prior work on Euclidean gradient alignment \cite{WZDD2023} (see \Cref{thm:general-distribution-ab-link-GLM-are-RGA}), and it holds uniformly for every Riemannian Clarke subgradient. This
substantially extends the scope of the framework to nonsmooth objectives and discrete covariate distributions. 
For example, under the uniform distribution on $\{-1,0,1\}^d$ with a ReLU activation, we show that the squared loss
is RGA on a constant-radius cap even though it fails to be geodesically convex on a much smaller cap around a target solution, of radius
order-$(1/\sqrt d)$ (\Cref{claim:lattice-relu-glm-not-geo-convex}).

\subsection{Related Work}\label{sec:related work}

Our work is related to multiple lines of research belonging to Riemannian optimization, local error bounds, and computational learning theory. We discuss only those works that are most closely related to ours. 

\paragraph{Riemannian Optimization}Riemannian optimization methods are powerful tools for  optimization problems with geometric and usually nonconvex constraints that can be formulated as smooth Riemannian manifolds~\cite{Absil2008OptimizationAlgorithmsOnMatrixManifolds,boumal2023introduction}. 
Hence, over the past decades, Riemannian optimization has garnered wide attention and witnessed great advancements.
Many applications in image and signal processing, machine learning, and statistics can be cast as optimization problems on Riemannian manifolds.  
For example, Riemannian optimization methods are used in {PCA/eigenvector problems~\cite{zhang2016riemannian,ADVA21,alimisis2022geodesic}}, matrix completion and recovery~\cite{boumal2011lowRankMatrix,vandereycken2013lowRankMatrix,2021LowRankMatrixRecovery}, robust PCA~\cite{zhang2018robustPCA}, robust subspace recovery~\cite{zhu2018DPCP,Li2019WeaklyCO,Zhu2019LinearConvergenceGrassman,ding2021DPCP}, tensor-on-tensor regression~\cite{luo2024tensor}, dictionary learning~\cite{sun2015complete,bai2018subgradient,gilboa2019efficient,Li2019WeaklyCO,qu2019analysis}, {sparse PCA~\cite{jolliffe2003modified,zou2018selective,chen2020proximal}}, SIM and GLM learning~\cite{ChenM20,BenArous2021OnlineSGD,damian2023smoothing,ZWDD2024,WZDD24b}, and in optimizing unit-weight neural networks~\cite{lee2025neural}. 
Most of the aforementioned works on Riemannian optimization consider RGD and its variants,  
whose convergence rates for geodesically convex and strongly convex objectives have been studied in detail (see the  textbook \cite{boumal2023introduction}). The objective functions considered in our work are not geodesically convex in general (see  \Cref{claim:lattice-relu-glm-not-geo-convex} for one such example).

\paragraph{Riemannian Local Error Bounds}While first-order optimization methods are often used in practice, it is rarely the case that the considered optimization objectives are geodesically (strongly) convex. Further, the associated objectives are often nonsmooth, invalidating standard arguments for RGD convergence analysis that rely on some form of the objective's smoothness. Thus, a large body of work in this area has studied convergence of RGD under local error bounds (a.k.a.\ ``growth conditions'').  
Local error bounds have a long history in optimization theory, starting with the seminal work of Hoffman \cite{hoffman1952approximate}. For concreteness, we focus on reviewing results most relevant to the present work, addressing problems on the unit sphere. 

\vspace{5pt}
\noindent\emph{Riemannian Regularity Condition.}
\cite{bai2018subgradient} showed that for the COD learning problem, the population loss and its empirical counterpart enjoy a type of local error bound, later termed Riemannian Regularity Condition (RRC) and generalized to the Grassmannian manifold \cite{Zhu2019LinearConvergenceGrassman}. The same work further proved that the Dual Principal
Component Pursuit (DPCP) loss for the robust Grassmannian subspace recovery problem also satisfies RRC (\cite{ding2021DPCP} later extended this result to noisy DPCP loss). 
Contrary to dictionary learning, DPCP is not formulated as a regression problem with underlying distributional assumptions on the data. 

While RRC applies more generally to other types of Riemannian manifolds, it is subsumed by RGA with $r = 1$ on the sphere: for non-zero Riemannian subgradient $\vg(\vx)$ of $ f(\vx)$ and parameters $L_0,\mu_{\rm RRC}>0$, RRC condition is defined by
\begin{equation}\tag{RRC}
    \begin{aligned}
        \|\vg(\vx)\|_2 &\leq L_0,\\
        \vg(\vx)\cdot\vx^* &\leq -\mu_{\rm RRC}\sin\theta(\vx, \vx^*).
    \end{aligned} \end{equation}
Comparing to RGA (\Cref{def:sharp}), it is immediate that RRC implies RGA with $\mu = \frac{\mu_{\rm RRC}}{L_0}$ and $r = 1.$ 
We note here that \cite{Zhu2019LinearConvergenceGrassman} also analyzed diminishing step RGD under RRC, proving a linear convergence rate.

\vspace{5pt}
\noindent\emph{Riemannian P\L{}/Gradient Dominance, Weak-Quasi-Convexity, Quadratic Growth.}
Another thread of work~\cite{zhang2016riemannian,balashov2020gradient} studied the Riemannian Polyak-\L{}ojasiewicz(P\L{})/Gradient Dominance conditions:
\begin{equation}\tag{P\L}
    f(\vx) - f(\vx^*)\leq \mu_{\mathrm{PL}}\|\nablar f(\vx)\|_2^2, \; \forall \vx^* \in \X^*, \mu_{\mathrm{PL}} >0,
\end{equation}
where $\X^*$ is the set of global solutions to the considered minimization problem \eqref{problem:opt-on-sphere-general} and $\nabla^R f(\vx) = P_{\perp \vx} \nabla f(\vx)$ is the Riemannian gradient of $f$.  
 They showed that when $f$ is geodesically smooth, there is a diminishing step-size schedule for RGD that leads to a linear convergence rate.
However, the only example of objectives in  \cite{zhang2016riemannian,balashov2020gradient} that satisfy the PL condition is the Rayleigh quotient $f(\vx) = \vx^\top\mA\vx$ where $\mA$ is a real symmetric matrix. 

 Similarly, focusing specifically on the Rayleigh quotient {of real symmetric positive semi-definite matrices,}
\cite{ADVA21} showed that the negative Rayleigh quotient function $h(\vx) = -f(\vx)= -\vx^\top\mA\vx$ is geodescially smooth and possesses weak-quasi-convexity and quadratic growth properties: for a global minimizer $\vx^*$ of $g$ and $\vx \in \X\eqdef \{\vy\in\spd: \theta(\vy,\vx^*)\leq \pi/2\}$, \begin{align}
    a(h(\vx) - h(\vx^*)) &\leq \langle \nablar h(\vx), -\Log_\vx(\vx^*)\rangle_\vx,\;  a>0, \tag{WQC}\\
    h(\vx) - h(\vx^*)&\geq \frac{\mu_{\mathrm{QG}}}{2}\mathrm{dist}^2(\vx, \vx^*),\;  \mu_{\mathrm{QG}}>0, \tag{QG}
 \end{align}
where $\Log_\vx(\vx^*)$ is the logarithmic map: $\Log_\vx(\vx^*) = ({\theta}/{\sin\theta})(\vx^*)_{\perp\vx}$, $\theta\eqdef\theta(\vx,\vx^*)$.
These conditions were later generalized to Grassmannian manifolds for the problem of finding the top-$k$ eigenvector space of a positive semi-definite matrix $\mA$, an extension of the PCA/Rayleigh quotient optimization problem~\cite{alimisis2022geodesic}.
They also proposed a diminishing-step RGD algorithm for objectives that are geodesically smooth and satisfy WQC and QG.
Note that in a Euclidean space, PL implies QG when $f$ is smooth~\cite{karimi2016linear} or weakly convex~\cite{Liao2024ErrorBoundsForWeaklyCOnvexFunctions}; 
On a Riemannian manifold, the implication from PL to local QG remains valid~\cite{rebjock2025fast}, furthermore, when $f$ is twice continuously differentiable, then PL, QG are essentially locally equivalent~\cite{rebjock2025fast}. 
However, these objective-based conditions do not, by themselves, imply RGA.
What can be established is that $L$-geodesic smoothness + WQC + QG implies RGA: let $\theta = \theta(\vx,\vx^*)$; since $\nablar f(\vx^*) = 0$ and $\|\nablar f(\vx) - \nablar f(\vx^*)\|_2\leq L\theta$, we have 
\begin{equation*}
    \nablar f(\vx)\cdot\vx^* = \langle\nablar f(\vx), (\vx^*)_{\perp\vx}\rangle_\vx \leq -\frac{a \mu_{\mathrm{QG}}}{2}\theta\sin\theta \leq-\frac{a \mu_{\mathrm{QG}}}{2 L}\|\nablar f(\vx)\|_2\sin\theta.
\end{equation*}

\vspace{5pt}
\noindent\emph{Riemannian Weak Sharpness}
\cite{Li2019WeaklyCO} studied objectives $f(\mX)$ that are weakly convex in the Euclidean space and satisfy a Riemannian weak-sharpness (WS) condition~\cite{karkhaneei2019nonconvex,Li2011WeakSharpMinimaOnRiemannianManifolds} on the rank-$r$ Stiefel manifold $\mX\in{\rm ST}(r,d)$:
\begin{equation}\tag{WS-Stiefel}
    f(\mX) - f(\mX^*)\geq \mu\, \dist(\mX,\X^*), \,\forall\mX^*\in\X^*,
\end{equation}
where $\X^*$ is the set of minima on the Stiefel manifold, and $\dist$ is the Frobenius distance.
For $f$ that are weakly convex and satisfy WS-Stiefel, \cite{Li2019WeaklyCO} gave a diminishing-step RGD algorithm that converges at a linear rate.
\cite{Li2019WeaklyCO} further showed that DPCP also satisfies WS-Stiefel and is weakly convex, and, in addition, empirically observed that orthogonal dictionary learning satisfies WS-Stiefel on ${\rm ST}(r,d)$.  

This type of a weak-sharpness condition has also appeared in the thread of research on matrix completion and recovery. For example, \cite{2021LowRankMatrixRecovery} studied composite losses of the form $f(\mX) = h(F(\mX))$ for a range of matrix recovery/sensing problems, where $\mX$ are rank-$r$ orthogonal matrices.  
When the matrix recovery/sensing problems satisfy a certain type of regularity condition (i.e., the Restricted Isometry Property (RIP)), \cite{2021LowRankMatrixRecovery} showed that the loss $f(\mX)$ satisfies WS-Stiefel and, in addition, it holds that 
\begin{equation}\label{eq:additional-conditions-for-matrix-recovery}
\begin{aligned}
    f(\mY) - h(F(\mX) + \la\nabla F(\mX), \mY - \mX\ra) &\leq \frac{\rho}{2}\|\mY - \mX\|_F^2,\; \forall \mX, \mY \in {\rm ST}(r,d)\\
    \|\vg(\mX)\|_F &\leq L,\;\forall\vg(\mX)\in\partial f(\mX).
\end{aligned}
\end{equation}
Although \cite{2021LowRankMatrixRecovery}  has also studied the WS-Stiefel property for the losses corresponding to matrix recovery problems, the discussion primarily concerns the Euclidean space, while their algorithms are not of RGD type. 

The unit sphere manifold considered in our work is a rank-$1$ Stiefel manifold, which is a special case of the manifolds studied in~\cite{Li2019WeaklyCO,2021LowRankMatrixRecovery}, but it is unclear whether the conditions from~\cite{Li2019WeaklyCO,2021LowRankMatrixRecovery} can be induced from RGA. 
However, it is worth noting that our RGA condition does not require additional assumptions such as weak-convexity or conditions in \Cref{eq:additional-conditions-for-matrix-recovery}, and for many examples discussed in this work, weak-convexity is not satisfied (e.g., many of the SIM problems in \Cref{sec:examples}).

\paragraph{Diminishing-step RGD} In the line of work discussed  above~\cite{zhang2016riemannian,zhu2018DPCP,Zhu2019LinearConvergenceGrassman,balashov2020gradient,ADVA21,alimisis2022geodesic,Li2019WeaklyCO,2021LowRankMatrixRecovery}, the analyzed algorithm is RGD with geometrically decreasing step sizes, shown to converge at a linear rate when the associated objective satisfies the respective local error bounds and additional conditions discussed above. These results are closely related to our analysis of RGD (\Cref{alg:RGD}) in the cases where RGA is satisfied with order $r= 1$. Additionally, our own prior work~\cite{WZDD2023,WZDD24b,ZWDD2024,zarifis2025robustly,wang2026robustly} established similar results (linear convergence of RGD with geometrically decaying step sizes) under related (but not identical) conditions to RGA. However, no prior work has established convergence guarantees that are related to RGA with $r > 1$. 

\paragraph{Context of Our Contributions} Our RGA condition is a direct generalization of the line of work on GLM and SIM learning \cite{mei2018landscape,diakonikolas2020approximation,yehudai2020learning,diakonikolas2022learning,WZDD2023,WZDD24b,ZWDD2024,zarifis2025robustly,wang2026robustly}. The power of RGA condition is in its ability to cover both smooth and nonsmooth objectives and to even extend to vector fields that are not (Riemannian sub)gradient vector fields. RGA also generalizes to degrading landscapes when $r>1$, covering examples such as sparse PCA (\Cref{prob:spiked-sparse-pca}) in \Cref{sec:spiked-sparse-pca}. 
Compared to the other types of local error bounds summarized above, although many of them can deal with more general manifolds, when constrained to the sphere, we either subsume these prior conditions (like RRC and smoothness + WQC + QG), or they are not directly comparable (WS + weakly convex in $\R^d$ or \Cref{eq:additional-conditions-for-matrix-recovery}), but we can cover examples that do not satisfy those conditions (e.g.,  many examples in \Cref{sec:examples} are not weakly convex).
More importantly, the ultimate goal of our work is not merely to define a new local error bound, but to investigate what structural properties of problems (such as function curvature properties) make local error bounds like RGA possible on the sphere. 
We focus on regression problems that span a wide range of learning and estimation applications, and propose a series of {derivative-based} sufficient conditions for RGA that are unified by {the curvature of the loss.}
The sufficient conditions are verified in all provided examples with explicit RGA parameters, and, based on these results, we not only recover known prior results {(\cite{bai2018subgradient,Zhu2019LinearConvergenceGrassman} for dictionary learning,~\cite{zhang2016riemannian,ADVA21} for optimizing Rayleigh quotient)}, but also streamline the relevant arguments within a unified framework~\cite{zhang2016riemannian,bai2018subgradient,Zhu2019LinearConvergenceGrassman,ADVA21}, extend the problem classes that can be handled through both generalizations of the loss functions and the associated distributions~\cite{mei2018landscape,diakonikolas2020approximation,yehudai2020learning,DKTZ20c,diakonikolas2022learning,WZDD2023,DDKWZ2023,WZDD24b,ZWDD2024,zarifis2025robustly,wang2026robustly}, 
and even prove local error bounds (namely, RGA) that are new in the literature {(e.g., for multi-instance learning problems---see \Cref{prob:mil} and \Cref{sec:mil}---and spiked-sparse covariance PCA---see \Cref{prob:spiked-sparse-pca} and \Cref{sec:spiked-sparse-pca})}.

\section{Preliminaries}\label{sec:prelims}

In this section, we introduce the necessary notation, definitions, and auxiliary facts used in our analysis. 

\paragraph{Notation}
Given two vectors $\vu, \vv$ in $\R^d,$ we use $\vu \cdot \vv = \vu^\top \vv$ to denote their (standard) inner product and $\theta(\vu,\vv) = \arccos({\vu \cdot \vv}/{\norm{\vu}_2 \norm{\vv}_2}) \in [0, \pi]$ to denote the angle between $\vu, \vv$, where $\norm{\cdot}_2$ denotes the $\ell_2$ norm. For a matrix $\mA$, $\|\mA\|_2$ denotes its operator norm.
{We denote $[d] :=\{1,\dots,d\}$.}
For any vector $\vv\in\R^d$, we use $\vv = (v_1,\dots,v_d)$ to denote its entries. $\ve_i$, $i\in[d]$ will be the canonical basis of the Euclidean space $\R^d$.
We write $\vu \perp \vv$ to indicate that $\vu, \vv$ are orthogonal to each other ($\theta(\vu,\vv) = \pi/2$). 
For vectors $\vu,\vv$ such that $\|\vu\|_2 = 1$, we use $\vv_{\perp\vu}\eqdef (\mI - \vu\vu^\top)\vv$ to denote the component of $\vv$ that is orthogonal to $\vu$. 
{The matrix $\mP_{\perp\vx}\eqdef\mI - \vx\vx^\top$ denotes the projection orthogonal to a unit vector $\vx$.}
{For a vector $\vb\in\R^m$, $\diag(\vb)$ denotes the diagonal matrix in $\R^{m\times m}$ with diagonal $\vb$, and $\va\odot \vb$ is the entry-wise product between vectors $\va$ and $\vb$. }
For a set $\V\subset\R^d$, we denote $\proj_\V(\vx) \eqdef \argmin_{\vv\in\V}\|\vx - \vv\|_2$. 
We use $\B^m(\vc,r)$ to denote the Euclidean ball centered at $\vc$ with radius $r$ in $\R^m$, and  $\Sp^{m-1}(\vc,r)$ to denote the sphere centered at $\vc$ with radius $r$ in $\R^m$. When $m = d$, we simply use $\B$ to denote the centered ($\vc = \vzero$) unit ball in $\R^d$ and $\Sp^{d-1}$ the centered unit sphere. 
We denote $[t]_+ = \max\{0,t\}$, and use $\1\{\cE\} = \1_\cE$ as an indicator for event $\cE$. 

{
For a random variable $X$, let $\cL(X)$ denote its law, and let $p_X$ denote its density when one exists. We use $L_p(\nu)$ for the usual $L_p$ space under the measure $\nu$, and occasionally write $L_p(X)$ as shorthand for $L_p(\cL(X))$. 
A random vector $\vz$ is isotropic if $\E[\vz] = 0$ and $\E[\vz\vz^\top] = \mI$.
We use $\calN_m$ to denote the $m$-dimensional standard Gaussian distribution and write $\calN \eqdef \calN_1$.
We denote the density of $\calN(0,\mI_m)$ by $p_{\calN_m}$, and write $p_\calN\eqdef p_{\calN_1}$.
We use $\Phi$ for the standard Gaussian cumulative distribution function.
}

On the manifold $\cM$, we denote the geodesic connecting $\vx,\vy\in\cM$ (i.e., the shortest path from $\vx$ to $\vy$) by $\calG_{\vx,\vy}(t)$, $t\in[0,1]$. 
{On the sphere, when $\theta(\vx,\vy)\in(0,\pi)$, the geodesic between $\vx$ and $\vy$ is the great circle: $\calG_{\vx,\vy}(t) = \cos(t\theta)\vx + \sin(t\theta)\vy_{\perp\vx}/\|\vy_{\perp\vx}\|_2, t\in[0,1]$.
Therefore, the geodesic distance on the sphere is $\dist_{\spd}(\vx,\vy) = \theta(\vx,\vy)$.}
{A set $\cX\subset\spd$ is geodesically convex if it contains the minimizing geodesic between every pair of vectors contained in  it. A function $f:\cX\to\R$ is geodesically $\mu$-strongly convex if for any geodesic $\cG_{\vx,\vy}(t)$ in $\cX$, it satisfies:
\begin{align}\label{eq:geo-str-cvx-basic}
    f(\cG_{\vx,\vy}(t)) + \frac{\mu}{2}t(1 - t)\dist_{\spd}(\vx,\vy)^2 \leq (1 - t)f(\vx) + t f(\vy),\;\forall t\in[0,1].
\end{align}
If $f:\cX\to\R$ satisfies \eqref{eq:geo-str-cvx-basic} with $\mu = 0$, they we say $f$ is geodesically convex.
}
{The tangent space of a point $\vx$ on the manifold $\cM$ is denoted by $\mathrm{T}_{\vx}\cM$. When $\cM$ is the unit sphere $\spd$, the tangent space is simply ${\rm T}_\vx\spd = \{\vy\in\R^d:\vy\perp\vx\}$.}

\paragraph{Derivatives of Nonsmooth Functions: Clarke Subdifferentials}
We consider optimization problems of the form \eqref{problem:opt-on-sphere-general}, 
where $f:\spd\to\R$ is a locally Lipschitz function. By Rademacher's theorem, $f$ is almost everywhere (a.e.) differentiable. 
For these locally Lipschitz functions $f$, we use $\nablar f(\vx)$ to denote the Riemannian gradient of $f$ at $\vx$; 
{when $f$ has a locally Lipschitz extension $\Tilde{f}$ to an open neighborhood of $\spd$, $\nablar f(\vx) = (\nabla \tilde{f}(\vx))_{\perp\vx}$ when it is differentiable at $\vx$. 
This definition is independent of the chosen extension.} With a slight abuse of notation, for $\vx \in \spd$, we use $\nabla f(\vx)$ to denote both the gradient (at points $\vx \in \spd$ where $f$ is differentiable) and any element in the Clarke subdifferential set:
\begin{align*}
    \nabla f(\vx)\in\partial^{C} f(\vx)\eqdef \mathrm{Conv}\bigg\{\lim_{k\to\infty} \nabla f(\vy_k): \text{ any } \vy_k\to\vx, f \text{ differentiable on $\vy_k$}\bigg\}.
\end{align*}
The same convention is applied to partial derivatives, i.e., ${\partial}_\vu\ell(\vu,\vv)\in\partial^{C}_\vu\ell(\vu,\vv)$ and Riemannian gradients $\nablar f(\vx)\in\partial^{C,R} f(\vx)$. 
When $f(\vx)$ is a population loss, meaning that it is expressible as $f(\vx) = \Evz[\ell(\vx;\vz)]$, define the Aumann expectation by:
\begin{align*}
    \Evz[\partial^C_\vx \ell(\vx;\vz)]\eqdef \bigg\{\Evz[\partial_\vx\ell(\vx;\vz)]\,:\,\partial_\vx\ell(\vx;\vz) \text{ is a measurable and integrable selection from $\partial_\vx^C \ell(\vx;\vz)$}\bigg\}.
\end{align*}
Then we have the following properties:
\begin{fact}\label{fact:Aumann-expectation}
Let $\cX$ be an open set in $\R^d$.
\begin{enumerate}
    \item \cite{BurkeChenSun2020} {Let $\vz$ be a random variable on $\Omega$. Let $\ell:\cX\times\Omega\to\R$ be measurable in $\vz$ and locally Lipschitz in $\vx$ almost surely. Suppose that, for every compact $\cK\subset \cX$, there is an integrable random variable $M_\cK(\vz)$ such that $|\ell(\vx,\vz) -\ell(\vy,\vz)|\leq M_\cK(\vz)\|\vx - \vy\|_2$, $\vx,\vy\in \cK$. Then $f(\vx) = \Evz[\ell(\vx;\vz)]$ is locally Lipschitz with respect to $\vx$ and $\partial^C f(\vx)\subseteq \Evz[\partial^C_\vx \ell(\vx;\vz)]$.}
    \item \cite{clarke1990optimization}  If $F:\cX\to\R^m$ is continuously differentiable and $g:\R^m\to\R$ is locally Lipschitz near $F(\vx)$, then $\partial^C (g(F(\vx)))\subseteq J_F(\vx)^\top\partial^C(g(F(\vx)))$, where $J_F$ is the Jacobian of $F$.
    In particular, if $h:\R\to\R$ is continuously differentiable and $\sigma:\R\to\R$ is locally Lipschitz, then $\partial^C (h(\sigma(u)))\subseteq h'(\sigma(u))\partial^C\sigma(u)$.
\end{enumerate}
\end{fact}
We note that the same result applies to the Riemannian subdifferential by projecting to the space orthogonal to $\vx$, by definition.

\paragraph{Derivatives of Nonsmooth Functions: Distributional derivatives under Gaussian Integrals}
In several proofs, we differentiate functions with jump discontinuities inside expectations. Such derivatives are understood in the weak sense, or as distributional derivatives (see, e.g., \cite[Chapter 9]{folland1999real}). Let $Df(z)$ be the distributional derivative of a function $f$. Note that when a function $f$ is differentiable, its distributional derivative is equal to its classical derivative: $Df = f'$. In this paper, we only require the following properties:
\begin{fact}[Distributional Derivatives]\label{fact:distributional-derivatives}
    It holds that $D|z| = \sign(z)$, $D\,\sign(z) = 2 \delta(z)$, where $\delta(z)$ is the Dirac measure at the origin, and $D\,\sign(q(z)) = 2\delta(q(z))q'(z)$ for any continuously differentiable function $q$.

    The notation $\E[Df(X)]$, where $X$ is a random variable with continuous density $p_X(x)$, means the `distribution-pairing' between $Df$ and $p_X(x)$, and in particular, for any function $g\in L_2(X)$, it holds:
\begin{align*}
    \E[g(X)\delta(X)] = \int_\R g(x)\delta(x)p_X(x)\diff{x}  = g(0)p_X(0). \end{align*}
\end{fact}
Further computational facts regarding integration of the Dirac delta function are gathered in \Cref{fact:dirac-inetgral}.

 We make extensive use of Stein's Lemma, stated below for completeness. 
\begin{fact}[Stein's Lemma~\cite{Stein1981}]\label{fact:stein}
    Let $\D = \calN(\vec 0,\vec I_m)$ and let $g:\sR^m \to \sR$ be an integrable function such that $\Evz[g(\vz)\vz]$ exists. Suppose $g(\vz)$ has a distributional derivative $D g(\vz)$.
    Then, $\Evz[g(\vz)\vz] = \Evz[D g(\vz)]$.
\end{fact}
\begin{proof}
    Since $\nabla p_\calN(\vz) = -\vz p_\calN(\vz)$, by the property of distributional derivatives (see, e.g., \cite{folland1999real}), it holds: 
    \begin{align*}
        \E_{\vz\sim\calN (\vec 0,\vec I_m)}[Dg(\vz)] = \int_{\R^m} p_\calN(\vz)\diff{(D g)}(\vz) = -\int_{\R^m} g(\vz)\nabla p_\calN(\vz)\diff{\vz} = \int_{\R^m} \vz g(\vz)p_\calN(\vz)\diff{\vz} = \E_{\vz\sim\calN (\vec 0,\vec I_m)}[g(\vz)\vz].
    \end{align*}
\end{proof}

\section{Riemannian Gradient Alignment and Diminishing-step  RGD}\label{sec:sharpness-RSI}

In this section, we establish sufficiency of the RGA property for efficient optimization on the sphere. We do so by proving that \Cref{alg:RGD} converges to target $\vx^*$ at a linear rate (for $r = 1$) or at rate $1/k^{1/(2r-2)}$ (for $r>1$). The considered algorithm is known as Riemannian Gradient Descent (RGD) when the input vector field $\vg(\vx)$ is the Riemannian gradient of the objective associated with a minimization problem on the sphere \eqref{problem:opt-on-sphere-general}. Although our work does not restrict $\vg$ to Riemannian gradient fields, we still refer to the considered algorithm as RGD, for consistency and comparison to related literature.

\begin{algorithm}[ht]
   \caption{Riemannian Gradient Descent (RGD)}
   \label{alg:RGD}
\begin{algorithmic}[1]
\STATE {\bfseries Input:} oracle access to vector field $\vg(\vx) \in \mathrm{T}_{\vx}\spd$, initial point $\vx^0$, step size schedule $\eta_k$, number of iterations $K$ \FOR{$k = 0,\dots,K-1 $}
\STATE $\vx^{k+1} = (\vx^k - \eta_k \vg(\vx^k))/\|\vx^k - \eta_k \vg(\vx^k)\|_2$
\ENDFOR
\STATE {\bfseries Return:} $\vx^{K}$
\end{algorithmic}
\end{algorithm}
\begin{restatable}[Convergence of RGD under RGA]{proposition}{proprgd}\label{lem:descent-rsi}
Let $\vx^*\in\spd$ be a fixed target vector.  
Given an initial vector $\vx^0 \in \spd$ such that $\theta(\vx^0,\vx^*)\leq \bar{\theta} < \pi/2$, let $\{\vx^k\}_{k=1}^K$ be the sequence of vectors generated by \Cref{alg:RGD}, and denote $\theta_k = \theta(\vx^k,\vx^*)$. Suppose $\vg(\vx)$ is $(\mu,r)$-RGA at $(\vx^k,\vx^*)$ for all $\{\vx^k\}_{k=0}^K$, where $0 < \mu \leq 1$. 

\begin{enumerate}
    \item If $r = 1$ and $\eta_k = \sin(\bar{\theta}/2)\mu(1 - \mu^2/8)^k/\|\vg(\vx^k)\|_2$, then
\begin{equation*}
        \frac{1}{2}\|\vx^k - \vx^*\|_2 = \sin(\theta_k/2)\leq \sin(\bar{\theta}/2)\exp\bigg(-\frac{\mu^2 k}{8}\bigg);
    \end{equation*}
    \item If $r>1$, let $\bar{k} \eqdef \lfloor 4/((r-1)\mu^2(\sin(\bar{\theta}/2)/2)^{2(r-1)})\rfloor - 1$ and 
    \begin{align*}
        \eta_k = \begin{cases}
            \frac{\mu\sin^r(\bar{\theta}/2)}{\|\vg(\vx^k)\|_2} & \text{when } k\leq \bar{k}-1\;,\\
            \big(\frac{4}{\mu^2(r-1)(k+1)}\big)^{r/(2r-2)}\frac{\mu}{\|\vg(\vx^k)\|_2} & \text{when } k\geq \bar{k}\;;
        \end{cases}
    \end{align*}
 then
    \begin{equation*}
        \frac{1}{2}\|\vx^k - \vx^*\|_2 =\sin(\theta_k/2)\leq 2\bigg(\frac{4}{\mu^2(r-1)}\cdot\frac{1}{k+1}\bigg)^{\frac{1}{2(r-1)}}\land \sin(\bar{\theta}/2).
    \end{equation*}
\end{enumerate}
\end{restatable}

\begin{proof}
    For simplicity of notation, in the following proof we denote $\vg(\vx^k)$ by $\vg^k$ and assume $\vg^k \neq \mathbf{0},$ as otherwise the algorithm would have converged to $\vx^*$. Observe that since $\vg^k$ is orthogonal to $\vx^k$, we have $\|\vx^k - \eta_k\vg^k\|_2\geq \|\vx^k\|_2 = 1$,
hence  $\vx^{k+1} = \proj_{\B}(\vx^k - \eta_k\vg^k).$ As a consequence, the distance between $\vx^{k+1}$ and $\vx^*$ can be bounded above in the following way:
    \begin{align}
        \|\vx^{k+1} - \vx^*\|_2^2 &= \|\proj_{\B}(\vx^k - \eta_k\vg^k) - \vx^*\|_2^2 \nonumber\\
        &\leq \|\vx^k - \vx^* - \eta_k\vg^k\|_2^2\nonumber\\
        &= \|\vx^k - \vx^*\|_2^2 + \eta_k^2\|\vg^k\|_2^2 - 2\eta_k\vg^k\cdot(\vx^k - \vx^*)\nonumber\\
        &= \|\vx^k - \vx^*\|_2^2 + \eta_k^2\|\vg^k\|_2^2 + 2\eta_k\vg^k\cdot \vx^*,\label{eq:expanding ||xk+1 - x*||2}
    \end{align}
    where in the first inequality we used the fact that projection onto $\B$ is a nonexpansive operator and in the last equality we used that $\vg^k\perp \vx^k$, which holds by the definition of $\vg^k$. 
    
    Since $\vx^k,\vx^{k+1},\vx^*$ are vectors on the unit sphere, the Euclidean distance between these vectors can be alternatively expressed by the angle. In particular, we have $\|\vx^{k+1} - \vx^*\|_2 = 2\sin(\theta_{k+1}/2)$ and similarly, $\|\vx^k - \vx^*\|_2 = 2\sin(\theta_k/2)$. Thus, plugging these equalities back into \Cref{eq:expanding ||xk+1 - x*||2} and applying the condition $\vg^k\cdot\vx^*/\|\vg^k\|_2\leq -\mu(\sin\theta_k)^r$, we further get
    \begin{equation}\label{eq:base-equation}
        4\sin^2(\theta_{k+1}/2)\leq 4\sin^2(\theta_k/2) + \eta_k^2\|\vg^k\|_2^2 - 2\eta_k\mu\|\vg^k\|_2\sin^r(\theta_k).
    \end{equation}

    \paragraph{Case I: $r=1$.} Let $\phi_k = \sin(\bar{\theta}/2)(1 - \mu^2/8)^k$ and choose step size $\eta_k = \sin(\bar{\theta}/2)\mu(1 - \mu^2/8)^k/\|\vg^k\|_2 = \mu\phi_k/\|\vg^k\|_2$. Plugging the definition of $\eta_k$ into \Cref{eq:base-equation} yields
    \begin{equation}\label{eq:decrease-r=1}
        4\sin^2(\theta_{k+1}/2)\leq 4\sin^2(\theta_k/2) + \mu^2\phi_k(\phi_k - 2\sin(\theta_k)).
    \end{equation} 
    We show that $\sin(\theta_k/2)\leq \phi_k$ by induction on $k$. 
    Observe first that since $\theta_0\in[0,\bar{\theta}]$, the base case $\phi_0 = \sin(\bar{\theta}/2) \geq \sin(\theta_0/2)$ holds trivially. 
    
    Now suppose that the inductive hypothesis is true for $0,\dots,k$. In particular, $\sin(\theta_k/2) \leq \phi_k$ for some $k \geq 0.$  Since $\phi_k \leq \phi_0 < \frac{1}{\sqrt{2}}$, we have that $\theta_k < \pi/2,$ which implies $\sin(\theta_k) \geq \sqrt{2}\sin(\theta_k/2) \geq \sin(\theta_k/2)$. Substituting into \Cref{eq:decrease-r=1}, this leads to
\begin{equation}\label{eq:decrease-r=1-2}
        4\sin^2(\theta_{k+1}/2)\leq 4\sin^2(\theta_k/2) + \mu^2\phi_k(\phi_k - 2\sin(\theta_k/2)).
    \end{equation} 
Divide both sides of \Cref{eq:decrease-r=1-2} by $\phi_k^2$ and denote $y_k := \sin(\theta_k/2)/\phi_k$ to get:
\begin{equation}\label{eq:r=1-ratio-bnd}
        \frac{4\sin^2(\theta_{k+1}/2)}{\phi_k^2} \leq h(y_k) := 4 y_k^2 + \mu^2 - 2\mu^2 y_k. 
   \end{equation}
Observe that $h$ is convex in $y_k.$ Since $y_k \in [0, 1]$, $h$ is maximized at its argument equal to either 0 or 1. Further, for $\mu \in (0, 1],$ we have $h(1) - h(0) = 4 - 2\mu^2 > 0,$ thus we can conclude that $h(y_k) \leq h(1) = 4 - \mu^2.$ Thus, we can conclude from \Cref{eq:r=1-ratio-bnd} that
\begin{equation}\notag
        \sin^2(\theta_{k+1}/2) \leq \phi_k^2 \big(1 - \mu^2/4\big) \leq \phi_k^2 \big(1 - \mu^2/8\big)^2 = \phi_{k+1}^2, 
    \end{equation}
completing the inductive proof. 

      \paragraph{Case II: $r > 1$.} Denote $p = 1/(r - 1)\in(0,+\infty)$, so that  $r = 1 + 1/p$. Let $\phi_k$ be a decreasing sequence of order $(k+1)^{-1/(2r-2)}$, defined as follows 
    \begin{equation*}
        \phi_k = \frac{C_\phi}{(k+1)^{\frac{p}{2}}} = 2\bigg(\frac{1}{\mu}\sqrt{\frac{4p}{k+1}}\bigg)^p = 2\bigg(\frac{1}{\mu^2(r-1)}\cdot\frac{4}{k+1}\bigg)^{\frac{1}{2(r-1)}}, \;\text{where } C_\phi=  2\bigg(\frac{\sqrt{4p}}{\mu}\bigg)^p .
    \end{equation*}
    Our goal is show that $\sin(\theta_k/2)\leq \phi_k\land \sin(\bar{\theta}/2)$ for $k\geq 1$ using induction on $k$. 

Observe that when $k = 0$,  we have $\phi_0 = 2(2/\mu)^p p^{p/2}\geq 1$ since $\mu\leq 1$ and $p^{p/2}\geq 1/2$ for all $p>0$. Therefore, it always holds that $\phi_k\geq \sin(\bar{\theta}/2)$ in the beginning phase where $k$ is small. Thus, to begin with, consider the case $0\leq k\leq \bar{k} \eqdef \lfloor 4p/(\mu^2(\sin(\bar{\theta}/2)/2)^{2/p})\rfloor - 1$ where it holds $\phi_k\geq \sin(\bar{\theta}/2)$. 
Then $\eta_k = \mu\sin^r(\bar{\theta}/2)/\|\vg^k\|_2$ for all $k\leq\bar{k} - 1$ by our definition of step size. 
We show by induction that $\sin(\theta_k/2)\leq \sin(\bar{\theta}/2)$ when $0\leq k\leq \bar{k}$. 
The inequality naturally holds when $k =0$ since $\bar{\theta}\geq \theta_0$ by the definition of $\bar{\theta}$. Now using induction hypothesis that $\sin(\theta_k/2)\leq \sin(\bar{\theta}/2)$ at $k\leq \bar{k}-1$, consider the updated angle $\theta_{k+1}$ after one algorithm iteration. 
Plugging the definition of $\eta_k$ into \Cref{eq:base-equation} and recalling that $\sin(\theta_k)\geq \sqrt{2}\sin(\theta_k/2)$ when $\theta\in(0,\pi/2)$, we get \begin{align}
        4\sin^2(\theta_{k+1}/2)&\leq 4\sin^2(\theta_k/2) + \mu^2\sin^{2r}(\bar{\theta}/2) - 2\mu^2\sin^r(\bar{\theta}/2)\sin^r(\theta_k) \nonumber\\
        &\leq 4\sin^2(\theta_k/2) + \mu^2\sin^r(\bar{\theta}/2)(\sin^r(\bar{\theta}/2) - 2^{1+r/2}\sin^r(\theta_k/2)).\label{eq:decrease-r>1-0}
    \end{align}
Suppose first $\sin(\bar{\theta}/2)\geq \sin(\theta_k/2)\geq \sin(\bar{\theta}/2)/\sqrt{2}$. Then it holds $4\sin^2(\theta_{k+1}/2)\leq 4\sin^2(\bar{\theta}/2) - \mu^2\sin^{2r}(\bar{\theta}/2)\leq 4\sin^2(\bar{\theta}/2)$, indicating that $\sin(\theta_{k+1}/2)\leq \sin(\bar{\theta}/2)$. On the other hand, suppose $\sin(\theta_k/2)\leq \sin(\bar{\theta}/2)/\sqrt{2}$, then using \Cref{eq:decrease-r>1-0} again we see that $\sin(\theta_{k+1}/2)$ cannot be much larger than $\sin(\theta_k/2)$: 
\begin{align*}
    4\sin^2(\theta_{k+1}/2)&\leq 4\sin^2(\theta_k/2) + \mu^2\sin^{2r}(\bar{\theta}/2)\leq 2\sin^2(\bar{\theta}/2) + \mu^2\sin^{2r}(\bar{\theta}/2)\\
    &\leq 4\sin^2(\bar{\theta}/2)\bigg(\frac{1}{2} + \frac{\mu^2}{4}\sin^{2(r-1)}(\bar{\theta}/2)\bigg)\leq 4\sin^2(\bar{\theta}/2),
\end{align*}  
where in the last inequality we used the fact that $\mu\leq 1$ and $r>1$, hence $\mu^2\sin^{2(r-1)}(\bar{\theta}/2)\leq 1$. The inequality above implies again that $\sin(\theta_{k+1}/2)\leq \sin(\bar{\theta}/2)$. This completes the induction argument: we have proved that $\sin(\theta_k/2)\leq \sin(\bar{\theta}/2)$ for all $k\leq \bar{k}$.

Next, consider the case $k\geq \bar{k}$. We continue using induction to show that $\sin(\theta_k/2)\leq \phi_k$, for all $k\geq \bar{k}$. So far, we have proved that $\sin(\theta_{\bar{k}}/2)\leq \sin(\bar{\theta}/2)\leq \phi_{\bar{k}}$, hence the base case for the induction is valid. Now suppose that the inductive hypothesis holds at step $k\geq \bar{k}$. Plugging in the choice of step size $\eta_k = \mu(\phi_k/2)^r/\|\vg^k\|_2$ as well as the RGA inequality $\vg^k\cdot\vx^*/\|\vg^k\|_2\leq -\mu\sin^r(\theta_k)$ into \Cref{eq:base-equation}, we obtain
\begin{align}
    4\sin^2(\theta_{k+1}/2)&\leq 4\sin^2(\theta_k/2) + \mu^2(\phi_k/2)^{2r} - 2\mu^2(\phi_k/2)^r\sin^r\theta_k\nonumber\\
    &\leq 4\sin^2(\theta_k/2) + \mu^2(\phi_k/2)^{r}((\phi_k/2)^r - 2(\sqrt{2}\sin(\theta_k/2))^r),\label{eq:decrease-r>1}
\end{align}
where in the last inequality we used the fact that $\sin\theta\geq \sqrt{2}\sin(\theta/2)$ whenever $\theta\in(0,\pi/2)$. 

Similar to our previous arguments, we discuss the cases $\phi_k\geq \sin(\theta_k/2)\geq \phi_k/(2\sqrt{2})$ and $\sin(\theta_k/2)\leq \phi_k/(2\sqrt{2})$. In the first case, by \Cref{eq:decrease-r>1},
dividing $4\phi_k^2$ on both sides, we have:
\begin{align*}
    \frac{\sin^2(\theta_{k+1}/2)}{\phi_k^2}&\leq\frac{\sin^2(\theta_k/2)}{\phi_k^2} + \frac{\mu^2}{16}\bigg(\frac{\phi_k}{2}\bigg)^{2r-2} - \frac{\mu^2}{2}\frac{\phi_k^{r-2}}{2^r}(\sqrt{2})^r\sin^2(\theta_k/2)\\
    &=\frac{\sin^2(\theta_k/2)}{\phi_k^2} + \frac{\mu^2}{16}\bigg(\frac{\phi_k}{2}\bigg)^{2r-2} - 2^{3r/2 - 3}\bigg(\frac{\sin(\theta_k/2)}{\phi_k}\bigg)^r \mu^2\bigg(\frac{\phi_k}{2}\bigg)^{2r-2}.
\end{align*}
Recall that by our definition of $\phi_k$, it holds that $\mu^2(\phi_k/2)^{2/p} = 4p/(k+1)$. Then, denoting $y_k = \sin(\theta_k/2)/\phi_k$, we have:
\begin{align*}
     \frac{\sin^2(\theta_{k+1}/2)}{\phi_k^2}& = y_k^2 + \frac{p}{4(k+1)} - \frac{2^{(3/2)r} y_k^r p}{2(k+1)} = y_k^2 + \frac{p}{k+1}\bigg(\frac{1}{4} - \frac{1}{2}(2\sqrt{2}y_k)^r\bigg).
\end{align*}
Since we have assumed $2\sqrt{2}y_k = 2\sqrt{2}\sin(\theta_k/2)/\phi_k\geq 1$, and that $r>1$, it then holds that:
\begin{align*}
    \frac{\sin^2(\theta_{k+1}/2)}{\phi_k^2}\leq y_k^2 + \frac{p}{k+1}\bigg(\frac{1}{4} - \frac{1}{2}(2\sqrt{2}y_k)\bigg) = y_k^2+ \frac{p}{k+1}\bigg(\frac{1}{4} - \sqrt{2}y_k\bigg):=h(y_k).
\end{align*}
Since $h$ is a convex quadratic function on $[0,1]$, its maximum value is obtained at either $h(0)$ or $h(1)$. For $h(0)$, since $k+1\geq \bar{k} + 1\geq \lfloor 4p\rfloor\geq 2p$ when $p>1/2$ and $k+1\geq 1\geq 2p$ when $p\leq 1/2$, it holds that $h(0) = {p}/{(4(k+1))}\leq 1 - p/(k+1)$. On the other hand, for $h(1)$, we have
\begin{align*}
    h(1) = 1 + \frac{p}{k+1}(1/4 - \sqrt{2})\leq 1 - \frac{p}{k+1},
\end{align*}
since $\sqrt{2}\geq 5/4$. Therefore, we obtain:
\begin{align*}
    \sin^2(\theta_{k+1}/2)\leq \phi_k^2\bigg(1 - \frac{p}{k+1}\bigg) = \frac{C_\phi^2}{(k+1)^p}\bigg(1 - \frac{p}{k+1}\bigg).
\end{align*}
Hence, to prove that $\sin^2(\theta_{k+1}/2)\leq \phi_{k+1}^2$ it suffices to prove $1/(k+1)^p(1- p/(k+1))\leq 1/(k+2)^p$.
Multiplying $(k+2)^p$ on both sides of this inequality and canceling the parameter $C_\phi^2$, the inequality we want to prove turns out to be 
    \begin{equation}\label{eq:decrease-r>1-sufficient-1}
        \bigg(1 + \frac{1}{k+1}\bigg)^p\bigg(1 - \frac{p}{k+1}\bigg)\leq 1.
    \end{equation}

    Recall that we have the basic inequality $(1 + x)^a\leq 1/(1 - ax)$ whenever $x\in(-1,1/a]$, $a>0$.
Since we have argued that $k+1\geq 2p$, it holds that $x = 1/(k+1)\in(-1, 1/p]$.
Therefore, the inequality implies 
    \begin{equation*}
        \bigg(1 + \frac{1}{k+1}\bigg)^p\bigg(1 - \frac{p}{k+1}\bigg)\leq \frac{1}{1 - p/(k+1)}\frac{k +1 - p}{k+1} = 1,
    \end{equation*}
    thus, \Cref{eq:decrease-r>1-sufficient-1} is a valid inequality.
    This concludes the induction argument for the first scenario of $\phi_k/(2\sqrt{2})\leq \sin(\theta_k/2)\leq \phi_k$.

    We proceed to the second case where $\sin(\theta_k/2)\leq \phi_k/(2\sqrt{2})$, $k\geq \bar{k}$, and show that $\sin(\theta_{k+1})\leq \phi_{k+1}$. We use the inequality for $\sin(\theta_{k+1}/2)$ and $\sin(\theta_k/2)$:
    \begin{align*}
        2(\sin(\theta_{k+1}/2) - \sin(\theta_k/2)) &= \|\vx^{k+1} - \vx^*\|_2 - \|\vx^k - \vx^*\|_2 \leq \|\vx^{k+1} - \vx^k\|_2 \nonumber\\
        &\leq \|\vx^k - \eta_k\vg^k - \vx^k\|_2 = \eta_k\|\vg^k\|_2\leq \mu(\phi_k/2)^r.
    \end{align*}
Therefore, recalling that we have assumed $\sin(\theta_k/2)\leq \phi_k/(2\sqrt{2})$, $\phi_{k+1} - \sin(\theta_{k+1}/2)$ can be bounded below in the following way
    \begin{align*}
        \phi_{k+1} - \sin(\theta_{k+1}/2)&\geq \phi_{k+1} - \sin(\theta_k/2) - \frac{\mu}{2}\frac{\phi_k^r}{2^r} \geq \frac{C_\phi}{(k+2)^{p/2}} - \frac{C_\phi}{
        2\sqrt{2}{(k+1)}^{p/2}} - \frac{\mu}{2}\frac{C_\phi^r}{2^r(k+1)^{pr/2}}\\
        & = \frac{C_\phi}{(k+2)^{p/2}}\bigg(1 - \frac{1}{2}\bigg(\frac{k+2}{k+1}\bigg)^{p/2}\bigg(\frac{1}{\sqrt{2}} + \frac{\mu}{2}\frac{C_\phi^{r-1}}{2^{r-1}(k+1)^{p(r-1)/2}}\bigg)\bigg).
    \end{align*}
    Plugging in the definition of $C_\phi$ and recalling that $r-1 = 1/p$, we have $C_\phi^{r-1} = 2^{r-1}(2/\mu)\sqrt{p}$. Furthermore, $(k+1)^{p(r-1)/2} = \sqrt{k+1}$. Therefore, the inequality above can be further simplified as
    \begin{align*}
        \phi_{k+1} - \sin(\theta_{k+1}/2)&\geq \frac{C_\phi}{(k+2)^{p/2}}\bigg(1 - \frac{1}{2}\bigg(1 + \frac{1}{k+1}\bigg)^{p/2}\bigg(\frac{1}{\sqrt{2}} + \sqrt{\frac{p}{k+1}}\bigg)\bigg)\\
        &\geq \frac{C_\phi}{(k+2)^{p/2}}\bigg(1 - \frac{\sqrt{2}}{2}\bigg(1 + \frac{1}{2p}\bigg)^{p/2}\bigg)\geq \frac{C_\phi}{(k+2)^{p/2}}(1 - \sqrt{2}e^{1/4}/2)\geq 0,
    \end{align*}
    where in the second inequality we used $k + 1\geq\bar{k} + 1\geq 2p$ and in the last inequality we used the basic inequalities $(1+1/(2p))^{2p/4}\leq e^{1/4}$ and $\sqrt{2}e^{1/4}/2 < 1$. This concludes that when $\sin(\theta_k/2)\leq \phi_k/(2\sqrt{2})$, we still have $\sin(\theta_{k+1}/2)\leq \phi_{k+1}$.

    Thus, in summary, by induction on $k$, we have $\sin(\theta_k/2)\leq \phi_k\land\sin(\bar{\theta}/2)$ for all $k\geq 1$, completing the proof of the proposition.
\end{proof}

A few remarks are in order here. 
First, the geometrically
diminishing schedule in item~(i), $\eta_k \;=\; \mu\,(1-\mu^2/8)^k \,/\,\|\vg^k\|_2,$
should be viewed as an ``angle-guessing'' policy. In the $r=1$ case, the one-step angle
descent analysis shows that the decrease in $\sin\theta_{k}$ is maximized when
$\eta_k$ is of the order of $\mu\,\sin\theta_k/\|\vg^k\|_2$, which is generally not
known to the algorithm, as $\theta_k$ is the angle between $\vx_k$ and the unknown target vector $\vx^*$. The geometric rule above underestimates this unknown target
after a few iterations; once $\eta_k$ falls into the safe range, the proof of
\Cref{lem:descent-rsi} yields the linear rate
$\sin(\theta_k)\le 2\exp(-\mu^2 k/8)$. If an early $\eta_k$ is too large (i.e.,
overestimates the unknown angle scale), the projection step can cause a temporary
increase of the angle, but the analysis shows this increase is uniformly controlled;
because $\eta_k$ shrinks by the fixed factor $(1-\mu^2/8)$, the iterates quickly
return to the safe regime and the linear decay then persists. 

Second, observe that the current algorithm requires the knowledge of $\mu$ and $r$ to set the step size. For abstract, blackbox problems, these parameters would generally not be known. However, as we later show in \Cref{sec:examples}, these parameters can be fully specified for a wide range of regression-based problems, based on verifying the derivative-based sufficient conditions for RGA that we introduce in \Cref{sec:sufficient conditions}. Thus, for these example problems, parameters $\mu, r$ are explicitly known (and in all our examples $r \in \{1, 2\}$). 

More generally, outside the provided example problems,   
we expect most problems of interest that satisfy the RGA condition to satisfy it with $r = 1$ or $r = 2.$ Thus, running the algorithm for a constant number of guesses for the value of $r$ should suffice. Moreover, any reasonable lower bound on $\mu$ suffices, since any problem that satisfies the RGA property for some $\mu > 0$  trivially satisfies the same property for any parameter $\hat{\mu} \in (0, \mu).$ For this reason, the knowledge of $\mu$ can be relaxed by running the algorithm for guesses of $\mu$ of the form $(1/2)^k$ for $k = 0, 1, 2, \dots$, using the predicted number of steps from \Cref{lem:descent-rsi} until a solution with the target error is reached. This comes at, at most, a logarithmic cost in $1/\mu$, since $(1/2)^k \leq \mu$ for $k \geq \log_2(1/\mu).$ It is an interesting open question whether the same convergence bounds as in \Cref{lem:descent-rsi} can be attained using a step size schedule that is independent of the values of $\mu$ and $r$. 

Finally, the statement of \Cref{lem:descent-rsi} assumed the RGA condition on the algorithm trajectory, which itself may be challenging to guarantee if the region of $\spd$ on which RGA holds is small or irregular. However, in all examples supplied in \Cref{sec:examples}, we prove that the RGA condition holds on a spherical cap centered at $\vx^*.$ In this case, it suffices that the algorithm is initialized within this spherical cap; all the iterates can then be guaranteed to remain in the same spherical cap, as argued below. 
    
\begin{claim}[Iterates Stay in the Spherical Cap Region]
    Fix a target vector $\vx^*\in\spd$ 
and suppose that at any $\vx\in\mathcal{S}_\lambda\eqdef \{\vx\in\spd: \|\vx - \vx^*\|_2\leq \lambda\}$, $\lambda\leq \sqrt{2}$, the vector field $\vg(\vx)$ is $(\mu,r)$-RGA. 
Set the step size $\eta_k$ as in \Cref{lem:descent-rsi} with parameter $\bar{\theta} = 2\arcsin(\lambda/2)$. Then, starting from any vector $\vx^0\in\mathcal{S}_\lambda$, the iterates of \Cref{alg:RGD} stay in $\mathcal{S}_\lambda$ and RGA continues to hold on the entire algorithm trajectory.
\end{claim}
\begin{proof}
    We only need to notice that by \Cref{lem:descent-rsi}, for any vector $\vx^k$ generated by \Cref{alg:RGD}, it holds $\|\vx^k - \vx^*\|_2\leq 2\sin(\bar{\theta}/2) = \lambda$, indicating that the algorithm iterates stay in the RGA region $\mathcal{S}_\lambda$. 
\end{proof}

\section{Derivative-Based Sufficient Conditions for RGA}\label{sec:sufficient conditions}

In this section, we introduce derivative-based sufficient conditions for problems to satisfy the RGA condition. Our results involve two sets of conditions: (1) specific to stochastic optimization problems interpretable as regression/supervised learning tasks (\Cref{subsec:regression}), and (2) more general problems that do not necessarily involve stochasticity or even have an interpretation as minimization problems (\Cref{subsec:sufficient-condition-general-vector-field}). Thus, provided with a vector field for which any of such sufficient conditions can be verified, \Cref{lem:descent-rsi} implies that the provided diminishing-step RGD (\Cref{alg:RGD}) converges to a target solution $\vx^*$ at either linear (if $r = 1$) or sublinear (if $r > 1$) rate, depending on the parameter $r$ in the RGA condition. The subsequent section (\Cref{sec:examples}) provides a range of examples for which those sufficient conditions are verified, with fully specified RGA parameters $r$ and $\mu.$ 

As mentioned earlier, the considered vector field $\vg$ for which the RGA condition is derived may or may not correspond to the Riemannian gradient of an objective function, or even be associated with a minimization problem. To indicate when the vector field is associated with a loss function, we use subscript $\ell$; i.e., we denote it by $\vg_\ell$.

Most of our sufficient conditions (in fact, all but those in \Cref{subsec:Gaussian}) will require some regularity assumption about the considered vector field $\vg$; in particular, that it is either bounded or restricted Lipschitz-continuous (between an arbitrary $\vx$ and a fixed target $\vx^*$), as summarized below. This appears a mild condition, in line with the definitions of classical problem classes that have been studied in terms of oracle complexity \cite{nemirovsky1983wiley}, particularly considering that the problem is stated for vectors with unit norm. 
\begin{assumption}\label{assum:smoothness-of-regression-loss}
     For a fixed $\vx^*$, $\vg(\vx;\vx^*)$ satisfies either of the following conditions:
    \begin{enumerate}
        \item[(i)] $\|\vg(\vx;\vx^*)\|_2 \leq L_1$ for some fixed $L_1 > 0$ and all $\vx\in\spd$;
        \item[(ii)] $\|\vg(\vx;\vx^*)\|_2 \leq L_2 \|\vx - \vx^*\|_2$ for some fixed $L_2 > 0$ and all $\vx\in\spd$. \end{enumerate}
\end{assumption}
Observe that \Cref{assum:smoothness-of-regression-loss}(i) is the weaker of the two: it is implied by \Cref{assum:smoothness-of-regression-loss}(ii) with $L_1 = 2 L_2$, since it must be $\norm{\vx - \vx^*}_2 \leq 2$, as both $\vx, \vx^*$ are on the unit sphere. 

\subsection{Sufficient Conditions For Riemannian Correlation Vector Fields}\label{subsec:regression}

We begin our discussion of sufficient conditions by considering stochastic minimization problems on the sphere that are generically expressible as  
\begin{equation}\label{problem:general-regression-formulation}
    \min_{\vx\in\spd}\big\{f(\vx) \eqdef \E_{\mZ\sim\D}[\ell(\mZ\vx,\mZ\vx^*;\vx,\vx^*)]\big\}, \tag{Regression}
\end{equation}
where $\mZ \in\R^{m\times d}$, $\D$ is a distribution measure on $\R^{m\times d}$, $\ell:(\R^m\times\R^m)\times(\R^d\times\R^d)\to\R$ is a loss function, and $\vx^*\in\spd$ is a target solution (typically, a function minimizer). 
Such a problem formulation captures broad classes of fundamental learning tasks such as multi-instance learning/max-pooling, binary classification, dictionary learning, learning of generalized linear models (GLMs), learning of single-index models (SIMs), and (sparse spiked) PCA, all formally defined and further discussed in \Cref{sec:defs-for-examples}. The loss function includes the current candidate vector $\vx$ and the target parameter vector $\vx^*$ as explicit arguments to further emphasize that loss $\ell$ might implicitly depend on $\vx$ and $\vx^*$. 
An important example, which will be discussed in detail in \Cref{sec:defs-for-examples}, is the SIM regression problem. 
\begin{example}\label{examp:SIM-varying-loss}
    For SIM regression, the goal is to find both the target parameter $\vx^*$ and the target unknown link function (activation) $\sigma^*$ from a function class $\cF$, by minimizing (for example) the $L_2^2$ loss $f_\Ltwo(\vx)$:
\begin{gather*}
f_{\Ltwo}(\vx)\eqdef \min_{\sigma\in\cF}\Ez[(\sigma(\vz\cdot\vx) - \sigma^*(\vz\cdot\vx^*))^2] \eqdef \Ez[\ell_{\Ltwo}(\vz\cdot\vx,\vz\cdot\vx^*;\vx,\vx^*)], \\
\text{where }\ell_{\Ltwo}(u,v;\vx,\vx^*)\eqdef (\sigma_\Ltwo(u;\vx,\vx^*) - \sigma^*(v))^2, \sigma_\Ltwo(u;\vx,\vx^*)\in \argmin_{\sigma\in\cF}\Ez[(\sigma(\vz\cdot\vx) - \sigma^*(\vz\cdot\vx^*))^2]. 
\end{gather*}
Here, the link function $\sigma_\Ltwo(u;\vx,\vx^*)$ is defined as the best-fitting activation given $\vx$ under the $L_2^2$ loss.
Notably, $\sigma_\Ltwo(u;\vx,\vx^*)$ is an activation that depends on both $\vx$ and $\vx^*$. \end{example}

Our sufficient conditions are based on the properties of derivatives of the loss function $\ell(\vu,\vv;\vx,\vx^*), \ell:(\R^m\times\R^m)\times(\R^d\times\R^d)\to\R$, for which we introduce the following notation. \begin{gather*}
    \partial_1\ell(\vu,\vv;\vx,\vx^*)\in\R^m,\; (\partial_1\ell(\vu,\vv;\vx,\vx^*))_i = \frac{\partial}{\partial u_i} \ell(\vu,\vv;\vx,\vx^*),\\
    \partial_2\partial_1\ell(\vu,\vv;\vx,\vx^*)\in\R^{m\times m},\; (\partial_2\partial_1\ell(\vu,\vv;\vx,\vx^*))_{i,j} = \frac{\partial}{\partial v_j}(\partial_1\ell(\vu,\vv;\vx,\vx^*))_i = \frac{\partial^2}{\partial v_j\partial u_i}\ell(\vu,\vv;\vx,\vx^*),
\end{gather*}
where, as before, the gradient and partial derivatives at points where the function is not necessarily differentiable are interpreted using the Clarke subdifferential when the objective is locally Lipschitz or as distributional derivatives, as discussed in \Cref{sec:prelims}. Derivatives 
$\partial_3\ell$ and $\partial_4\ell$ are defined similarly.

Our sufficient conditions are provided for a `gradient-like' vector field, which does not necessarily correspond to the Riemannian gradient of the objective $f(\vx)$. The precise definition is provided below.
\begin{definition}[Riemannian Correlation Vector Field]\label{def:grad-field}
Let $\ell(\vu,\vv;\vx,\vx^*):\R^m\times\R^m\times\R^d\times\R^d\to\R$ be a loss that is locally Lipschitz with respect to $\vu$. We say that a vector field $\vg\in\mathrm{T}_{\vx}\spd$ is a Riemannian correlation vector field, if it is expressible as 
\begin{align*}
    \vg_\ell(\vx,\vx^*)\eqdef \E_{\mZ\sim\D}[\mP_{\perp\vx}\mZ^\top\partial_1 \ell(\mZ\vx,\mZ\vx^*;\vx,\vx^*)],
\end{align*}
where $\mZ\in\R^{m\times d}$ and $\mP_{\perp\vx} = \mI - \vx\vx^\top$. If $\ell$ is not everywhere differentiable, by convention, {we fix a measurable and integrable single-value selection $\partial_1\ell(\vu,\vv;\vx,\vx^*) \in \partial_1^C \ell(\vu,\vv;\vx,\vx^*)$ and define $\vg_\ell$ using $\partial_1\ell(\vu,\vv;\vx,\vx^*).$} 
All subsequent derivatives of $\partial_1 \ell$ refer to this fixed selection. 
\end{definition}
Whenever differentiation under the expectation is valid and the chain rule holds with equality (for example, when the loss is differentiable almost surely and its derivative admits an integrable local envelope), we have that $\nablar f(\vx) = \Evz[\mP_{\perp\vx}\mZ^\top\partial_1\ell(\mZ\vx,\mZ\vx^*;\vx,\vx^*)] + \Evz[\mP_{\perp\vx}\partial_3\ell(\mZ\vx,\mZ\vx^*;\vx,\vx^*)]$. As a consequence, if $\E_{\mZ\sim\D}[\partial_3\ell(\mZ\vx,\mZ\vx^*;\vx,\vx^*)] = \vzero$ (see, e.g., the examples in \Cref{subsec:glm-sim-Gaussian}), then the Riemannian correlation vector field coincides with the Riemannian gradient of the objective, as the chain rule implies:
\begin{align*}
    \nablar f(\vx) &= \Evz[\mP_{\perp\vx}\mZ^\top\partial_1\ell(\mZ\vx,\mZ\vx^*;\vx,\vx^*)] + \Evz[\mP_{\perp\vx}\partial_3\ell(\mZ\vx,\mZ\vx^*;\vx,\vx^*)] = \vg_\ell(\vx,\vx^*).
\end{align*}

In the rest of this subsection, we present sufficient conditions for Riemannian correlation vector fields (\Cref{def:grad-field}) associated with problems expressible as \eqref{problem:general-regression-formulation} to satisfy RGA. We begin the discussion by providing a minimal condition (\Cref{assum:ell-logconcave}), which, combined with the mild vector field regularity (\Cref{assum:smoothness-of-regression-loss}), implies RGA. This condition does not require any specific distributional assumptions and it is provided in \Cref{subsubsec:general-suff-cond}. We then specialize this condition to centered Gaussian distributions, considering distributions with  general positive semidefinite covariance matrices (\Cref{subsec:Gaussian-sigma}) and the standard Gaussian distribution (\Cref{subsec:Gaussian}). As the distribution becomes more specialized, the conditions become simpler, and in the case of standard Gaussians, \Cref{assum:smoothness-of-regression-loss} is no longer needed. 

Throughout this section, all displayed expectations and conditional expectations are assumed to exist as finite expectations, and all functions and distributional derivatives to which \Cref{fact:stein} is applied satisfy its integrability hypotheses.

\subsubsection{Sufficient Conditions for General Distributional Families}\label{subsubsec:general-suff-cond}

Our derivative-based condition for general distributional families is provided below. Throughout \Cref{subsubsec:general-suff-cond} and \Cref{subsec:Gaussian-sigma}, we specialize to $m = 1.$

\begin{assumption}\label{assum:ell-logconcave}
For a given \eqref{problem:general-regression-formulation} with its associated target $\vx^*\in\spd$, and given $\vx\in\spd$ such that $\theta(\vx,\vx^*)\in(0,\pi/2)$, denote $\vv = \vx^*_{\perp\vx}/\|\vx^*_{\perp\vx}\|_2$. Given a distribution $\D$ on $\R^d$, define the conditional kernel:
\begin{align*}
    c_{\vv\cdot\vz|\vx\cdot\vz}(t)\eqdef \Evz\bigg[\bigg(\vv\cdot\vz - \E[\vv\cdot\vz | \vx\cdot\vz]\bigg)\1\{t\leq \vv\cdot\vz\} \bigg| \vx\cdot\vz\bigg].
\end{align*}
There exist constants $\mu^{(1)}, \mu^{(2)}\geq 0$, such that $\mu^{(1)} +\mu^{(2)} > 0$, and the following conditions both hold:
\begin{enumerate}
    \item[(i)] $\Evz\big[\E[\partial_1\ell(\vx\cdot\vz,\vx^*\cdot\vz;\vx,\vx^*) | \vx\cdot\vz] \, \E[\vv\cdot\vz | \vx\cdot\vz]\,\big] \leq -\mu^{(1)}$;
    \item[(ii)] $\partial_1\ell(u_1,u_2;\vx,\vx^*)$ is locally Lipschitz in $u_2$, $\E_{\vx \cdot\vz}\big[\int_\R |\partial_2\partial_1\ell(\vx\cdot\vz,\vx\cdot\vz\cos\theta + t\sin\theta;\vx,\vx^*)c_{\vv\cdot\vz|\vx\cdot\vz}(t)|\diff{t}\big]<\infty$, and
    \begin{align*}
        \E_{\vx\cdot\vz}\bigg[\int_\R \partial_2\partial_1\ell(\vx\cdot\vz,\vx\cdot\vz\cos\theta + t\sin\theta;\vx,\vx^*)c_{\vv\cdot\vz|\vx\cdot\vz}(t)\diff{t}\bigg]\leq -\mu^{(2)}.
    \end{align*}
\end{enumerate}
\end{assumption}

\Cref{assum:ell-logconcave}, when combined with the regularity of the vector field (\Cref{assum:smoothness-of-regression-loss}), implies RGA.

\begin{theorem}\label{lem:regression-loss-satisfy-sharpness-logconcave}
Given \eqref{problem:general-regression-formulation} with target $\vx^*$ and an associated Riemannian correlation vector field $\vg_\ell(\vx,\vx^*)$, for any $\vx$ such that \Cref{assum:ell-logconcave} holds, where $\theta = \theta(\vx,\vx^*)\in (0,\pi/2)$:  
    \begin{enumerate}
        \item If \Cref{assum:smoothness-of-regression-loss}(i) is satisfied, then $\vg_\ell(\vx,\vx^*)\cdot\vx^*\leq -\frac{\mu^{(1)} + \mu^{(2)}\sin\theta}{L_1}\sin\theta\|\vg_\ell(\vx,\vx^*)\|_2$.
\item If \Cref{assum:smoothness-of-regression-loss}(ii) is satisfied, then $\vg_\ell(\vx,\vx^*)\cdot\vx^*\leq -\frac{\mu^{(1)} + \mu^{(2)}\sin\theta}{2L_2}\|\vg_\ell(\vx,\vx^*)\|_2$.
\end{enumerate}
\end{theorem}

Before proving \Cref{lem:regression-loss-satisfy-sharpness-logconcave}, we provide a few remarks clarifying the intuition behind \Cref{assum:ell-logconcave} and its correspondence with Stein's lemma (\Cref{fact:stein}). 

\begin{remark}[\textbf{\Cref{assum:ell-logconcave}~(i) is a loss curvature condition}]
    The insight behind Part (i) of \Cref{assum:ell-logconcave} is that it is a loss curvature condition. This is clearly seen when the distribution $\D$ is Gaussian $\calN(\vec 0,\mSigma)$ with a positive-definite covariance $\mSigma$, under which it holds: $\E[\vv\cdot\vz|\vx\cdot\vz] = (\vv^\top\mSigma\vx/\vx^\top\mSigma\vx)\vx\cdot\vz$. Then, denoting $\kappa = \vv^\top\mSigma\vx/\vx^\top\mSigma\vx$, Part (i) of \Cref{assum:ell-logconcave} becomes:
    \begin{align*}
        \Evz\big[\partial_1\ell(\vx\cdot\vz,\vx^*\cdot\vz;\vx,\vx^*) \, \E[\vv\cdot\vz | \vx\cdot\vz]\,\big] = \kappa\Evz[\partial_1\ell(\vx\cdot\vz,\vx^*\cdot\vz;\vx,\vx^*)\vx\cdot\vz].
    \end{align*}
    Now let $\vp = \mSigma^{-1/2}\vz$ (observe that $\vp\sim\calN(0,\mI)$), and assume $\partial_1 \ell$ is differentiable in its first two arguments, with the relevant derivatives being integrable. Applying Stein's lemma (\Cref{fact:stein}) leads to:
    \begin{align*}
        &\quad \Evz\big[\partial_1\ell(\vx\cdot\vz,\vx^*\cdot\vz;\vx,\vx^*) \, \E[\vv\cdot\vz | \vx\cdot\vz]\,\big]\\
        &=\kappa\E_{\vp\sim\calN_d}[\partial_1\ell((\mSigma^{1/2}\vx)\cdot\vp,(\mSigma^{1/2}\vx^*)\cdot\vp;\vx,\vx^*)(\mSigma^{1/2}\vx)\cdot\vp]\\
        &=\kappa\E_{\vp\sim\calN_d}[\partial_1^2 \ell((\mSigma^{1/2}\vx)\cdot\vp,(\mSigma^{1/2}\vx^*)\cdot\vp;\vx,\vx^*)](\vx^\top\mSigma\vx) + \kappa\E_{\vp\sim\calN_d}[\partial_2\partial_1 \ell((\mSigma^{1/2}\vx)\cdot\vp,(\mSigma^{1/2}\vx^*)\cdot\vp;\vx,\vx^*)](\vx^\top\mSigma\vx^*)\\
        &=(\vv^\top\mSigma\vx)\Evz[\partial_1^2
        \ell(\vx\cdot\vz,\vx^*\cdot\vz;\vx,\vx^*)] + (\cos\theta + \kappa\sin\theta)(\vv^\top\mSigma\vx)\Evz[\partial_2\partial_1
        \ell(\vx\cdot\vz,\vx^*\cdot\vz;\vx,\vx^*)],
    \end{align*}
    where we have used $\kappa \vx^\top\mSigma\vx = \vv^\top\mSigma\vx$ and $\vx^* = \vx\cos \theta + \vv \sin \theta.$ 
    Therefore, Part (i) of \Cref{assum:ell-logconcave} is a second-order (curvature) condition for the loss. \end{remark}

\begin{remark}[\textbf{Often, $\mu^{(1)} = 0$, making $\partial_2\partial_1\ell$ the primary contributor}]
    We should remark that in most of our examples (all except for the examples in \Cref{sec:ex:general Gaussian}), the contribution of the $\partial_1^2\ell$ term that enters via Part (i) of \Cref{assum:ell-logconcave} is often either small or zero, and therefore, the joint curvature $\partial_2\partial_1\ell$ is what truly determines the RGA property. In fact, in the examples in \Cref{subsubsec:logconcave-distribution-ab-link-SLM-are-RGA}, we have 
    \begin{align*}
        (i'):\; \E[\partial_1\ell(\vx\cdot\vz,\vx^*\cdot\vz;\vx,\vx^*) | \vx\cdot\vz] = 0.
    \end{align*}
     In addition, when $\vv\cdot\vz$ and $\vx\cdot\vz$ are independent and zero-mean (for instance, when $\D$ is the standard Gaussian), $\E[\vv\cdot\vz|\vx\cdot\vz] = 0$, hence $\mu^{(1)} = 0$ and, thus, only the joint curvature in Part (ii) matters.
\end{remark}

\begin{remark}[\textbf{The Conditional Kernel and Stein's Lemma}]
Let us denote $V = \vv\cdot\vz$ and $X = \vx\cdot\vz$. Intuitively, the conditional kernel $c_{V|X}(t)$ can be viewed as a probability density function  (after normalization) with respect to which Stein's Lemma (i.e., $\E_{Z\sim\calN}[Z g(Z)] = \int_\R g'(z)p_\calN(z)\diff{z}$, see \Cref{fact:stein}) extends from standard Gaussian to more general distributions. In particular, it is possible to argue (see the proof of \Cref{claim:nablarf-x*-bounded-by-sin2theta-log-concave}) that, for locally absolutely continuous function $g$ satisfying $\int_\R |g'(t)|c_{V|X}(t)\diff{t} < \infty,$ we have 
$$\Cov(g(V), V|X) = \int_\R g'(t)c_{V|X}(t)\diff{t}.$$ 
Therefore, Part (ii) of \Cref{assum:ell-logconcave} can also be viewed as an application of this generalized Stein's lemma.

We briefly justify the claim that the kernel $c_{V|X}(t)$ can indeed be regarded as a probability density. First, $c_{V|X}(t)$ is always non-negative for any $t\in\R$, because by definition, $0 = \E[(V - \E[V|X])|X] = \E[(V - \E[V|X])\1\{V\geq t\}|X] + E[(V - \E[V|X])\1\{V < t\}|X]$, which implies: 
\begin{align}\label{eq:kernel-positive}
    c_{V|X}(t) &= \E[(V - \E[V|X])\1\{V\geq t\}|X] = \E[(\E[V|X] - V)\1\{V < t\}|X]. 
\end{align}
If $\E[V|X] < t$, we have $c_{V|X}(t) = \E[(V - \E[V|X])\1\{V\geq t\}|X] \geq \E[(V - t)\1\{V\geq t\}|X] \geq 0.$ Similarly, if $\E[V|X]\geq t$, then $c_{V|X}(t) = \E[(\E[V|X] - V)\1\{V < t\}|X] \geq \E[(t - V)\1\{V < t\}|X] \geq 0;$ thus, 
it always holds that $c_{V|X}(t)\geq 0$.
Second, the generalized Stein's lemma above also implies that $\int_\R c_{V|X}(t)\diff{t} = \Var(V|X)$ by choosing $g(t) = t$. Hence, $c_{V|X}(t)/\Var(V|X)$ is a probability density function (so long as $\Var(V|X) \in (0, \infty)$, so we avoid degenerate/corner cases).
\end{remark}

\begin{remark}[\textbf{Sufficient Conditions for Standard Gaussian}]
When $\D = \calN(0,\mI)$, the sufficient conditions are much simpler. We have already argued in previous remarks that under the standard Gaussian distribution, $\E[\vv\cdot\vz|\vx\cdot\vz] = 0$ since $\vv\perp\vx$ and hence \Cref{assum:ell-logconcave}(i) is always satisfied with $\mu^{(1)} = 0$. Thus, the sufficient conditions reduce to checking Part (ii). Using the fact that $\vv\cdot\vz$ and $\vx\cdot\vz$ are independent standard Gaussians, the conditional kernel becomes $c_{\vv\cdot\vz|\vx\cdot\vz}(t) = \E[\vv\cdot\vz\1\{t\leq \vv\cdot\vz\}] = p_{\vv\cdot\vz|\vx\cdot\vz}(t) = p_{\calN}(t)$. Thus, \Cref{assum:ell-logconcave}(ii) is simply:
    \begin{align*}
    \E_{\vx\cdot\vz}\bigg[\int_\R \partial_2\partial_1\ell(\vx\cdot\vz,\vx\cdot\vz\cos\theta + t\sin\theta;\vx,\vx^*)p_\calN(t)\diff{t}\bigg] &=\E_{\vx\cdot\vz}[\E_{\vv\cdot\vz}[\partial_2\partial_1\ell(\vx\cdot\vz,\vx\cdot\vz\cos\theta + \sin\theta \vv\cdot\vz;\vx,\vx^*)]]\\
        &=\Evz[\partial_2\partial_1\ell(\vx\cdot\vz,\vx^*\cdot\vz;\vx,\vx^*)]\leq -\mu.
    \end{align*}
In \Cref{lem:regression-loss-satisfy-sharpness-Gaussian}, 
    we show that, in fact,  
    $\Evz[\partial_2\partial_1\ell(\vx\cdot\vz,\vx^*\cdot\vz;\vx,\vx^*)]<0$ alone 
    is sufficient for RGA, and we can even extend to the setting where the data is a standard Gaussian matrix 
    $\mZ\in\R^{m\times d}$.
\end{remark}

To prove \Cref{lem:regression-loss-satisfy-sharpness-logconcave}, we first prove two auxiliary lemmas (\Cref{claim:covariance-decomposition} and \Cref{claim:nablarf-x*-bounded-by-sin2theta-log-concave}). \Cref{claim:covariance-decomposition} implies that the correlation between the vector field $\vg_\ell(\vx,\vx^*)$ and the target direction $\vx^*$ can be decomposed into ``mean'' and ``covariance'' terms. \Cref{claim:nablarf-x*-bounded-by-sin2theta-log-concave} then uses \Cref{assum:ell-logconcave} to bound the two terms from the decomposition.

\begin{lemma}\label{claim:covariance-decomposition}
Given \eqref{problem:general-regression-formulation} and its corresponding  Riemannian correlation vector field $\vg_\ell$, fix $\vx \in \spd$ such that $\theta(\vx, \vx^*) \in (0, \pi/2),$ and let $\vv := \vx^*_{\perp\vx}/\|\vx^*_{\perp\vx}\|_2$. Then:
$$\vg_\ell(\vx,\vx^*)\cdot\vv = \E_{\vx\cdot\vz}\bigg[\Cov\bigg(\partial_1\ell(\vx\cdot\vz, \vx^*\cdot\vz;\vx,\vx^*), \vv\cdot\vz \bigg| \vx\cdot\vz\bigg)\bigg] + \Evz\bigg[\E\bigg[\partial_1\ell(\vx\cdot\vz, \vx^*\cdot\vz;\vx,\vx^*) \bigg| \vx\cdot\vz\bigg] \vv\cdot\vz\bigg].$$
    \end{lemma}
    \begin{proof}
    Denote, for simplicity, $X = \vx \cdot \vz,$ $V = \vv \cdot \vz$, and $A = \partial_1 \ell(\vx \cdot \vz, \vx^* \cdot \vz; \vx, \vx^*).$ The lemma claim follows from the conditional covariance decomposition: $\E[AV] = \E[\Cov(A, V| X)] + \E[\E[A|X]\E[V|X]]$ (a consequence of the definition of the conditional covariance), combined with the tower property of expectations, by which $\E[\E[A|X]V] = \E[\E[\E[A|X]V|X]] = \E[\E[A|X]\E[V|X]].$  
     \end{proof}

\begin{lemma}\label{claim:nablarf-x*-bounded-by-sin2theta-log-concave}
    Given \eqref{problem:general-regression-formulation} and the associated Riemannian correlation vector field $\vg_\ell,$ suppose \Cref{assum:ell-logconcave} is satisfied for some $\vx\in\spd$ such that $\theta(\vx,\vx^*)\in(0,\pi/2)$. Then:  
    $$\vg_\ell(\vx,\vx^*)\cdot \vx^* \leq -\mu^{(1)}\sin\theta(\vx,\vx^*) - \mu^{(2)}\sin^2\theta(\vx,\vx^*).$$
\end{lemma}
\begin{proof}
    Let $\vv = \vx^*_{\perp\vx}/\|\vx^*_{\perp\vx}\|_2$ and consider the decomposition of $\vg_\ell(\vx,\vx^*)\cdot \vv$ from 
    \Cref{claim:covariance-decomposition}.  
    By Part (i) of \Cref{assum:ell-logconcave}, we have 
$\Evz\big[\partial_1\ell(\vx\cdot\vz,\vx^*\cdot\vz;\vx,\vx^*) \, \E[\vv\cdot\vz | \vx\cdot\vz]\,\big] \leq -\mu^{(1)}$;
    therefore, it only remains to bound the covariance part.
Denote for simplicity $ z_\vu = \vu\cdot\vz$ for any unit vector $\vu$.
    We observe that the covariance term can be expressed as an integral with respect to the joint curvature of the loss $\ell$. Given any $ z_\vx$, let us denote
    \begin{align*}
       p_{ z_\vx}(t) \eqdef \partial_1\ell( z_\vx, z_\vx \cos\theta  + t \sin\theta ;\vx,\vx^*)\;\Rightarrow \; p_{ z_\vx}( z_\vv) = \partial_1\ell( z_\vx, z_\vx \cos\theta  + z_\vv\sin\theta; \vx,\vx^*). 
    \end{align*}
Observe that:  
\begin{equation}\label{eq:Cov-via-p_zx}
     \Cov(\partial_1\ell( z_\vx, z_{\vx^*};\vx,\vx^*),  z_\vv |  z_\vx)  = \Cov(p_{z_\vx}(z_\vv),  z_\vv |  z_\vx).
     \end{equation}
        Let $z_\vv'$ be an independent copy of $z_\vv$ conditional on $z_\vx.$ Then:
        \begin{align}
            \Cov(p_{z_\vx}(z_\vv),  z_\vv |  z_\vx) &= \frac{1}{2}\E\big[(p_{z_\vx}(z_\vv) - p_{z_\vx}(z_\vv'))(z_\vv - z_\vv')|z_\vx\big]\notag \\
            &= \frac{1}{2}\E\Big[\Big(\int_\R p_{z_\vx}'(t)(\1\{t\leq z_\vv\} - \1\{t\leq z_\vv' \})\diff{}t\Big)(z_\vv - z_\vv')|z_\vx\Big], \label{eq:Cov-of-p_zx-via-FTC}
        \end{align}
        where we have used that $p_{z_\vx}$ is locally Lipschitz, and thus absolutely continuous on every compact interval, so the fundamental theorem of calculus applies. This holds because $\partial_1\ell(u_1,u_2; \vx, \vx^*)$ is locally-Lipschitz in $u_2$, thus differentiable a.e., and, therefore, $p_{ z_\vx}(t)$ is also differentiable a.e., with $p'_{ z_\vx}(t) = \sin\theta\, \partial_2\partial_1\ell( z_\vx, z_\vx \cos\theta   + t \sin\theta ; \vx,\vx^*)$.
        
        Now observe that $(z_\vv - z_\vv')(\1\{t\leq z_\vv\} - \1\{t\leq z_\vv' \}) \geq 0$ and $\frac{1}{2}\E[(z_\vv - z_\vv')(\1\{t\leq z_\vv\} - \1\{t\leq z_\vv' \})|z_\vx] = \E[(z_\vv - \E[z_\vv|z_\vx])\1\{t\leq z_\vv\} |z_\vx] = c_{z_\vv|z_\vx}(t)$. Therefore,
\begin{align*}
    \frac{1}{2}\E\Big[\int_\R |p_{z_\vx}'(t)||\1\{t\leq z_\vv\} - \1\{t\leq z_\vv' \}||z_\vv - z_\vv'|\diff{}t\Big|z_\vx\Big] = \int_\R |p_{z_\vx}'(t)|c_{z_\vv|z_\vx}(t) \diff{} t,
\end{align*}
which is finite by \Cref{assum:ell-logconcave}(ii), as $p'_{ z_\vx}(t) = \sin\theta\, \partial_2\partial_1\ell( z_\vx, z_\vx \cos\theta   + t \sin\theta ; \vx,\vx^*)$. Thus, we can use Fubini's theorem to exchange the order of integration in \eqref{eq:Cov-of-p_zx-via-FTC}, so that, combined with \eqref{eq:Cov-via-p_zx}, we get
\begin{align*}
        \Cov(\partial_1\ell( z_\vx, z_{\vx^*};\vx,\vx^*),  z_\vv |  z_\vx) = \int_\R p'_{ z_\vx}(t) c_{ z_\vv| z_\vx}(t)\diff{t} =\sin\theta\int_\R\partial_2\partial_1\ell( z_\vx, z_\vx \cos\theta + t\sin\theta;\vx,\vx^*)c_{ z_\vv| z_\vx}(t)\diff{t}.
    \end{align*}
Using Part (ii) of \Cref{assum:ell-logconcave}, 
    we obtain: $\E_{z_{\vx}}[\Cov(\partial_1\ell( z_\vx, z_{\vx^*};\vx,\vx^*),  z_\vv |  z_\vx)]\leq -\mu^{(2)}\sin\theta$ and, thus, combined with the bound on the mean from Part (i), we get
    \begin{align*}
        \vg_\ell(\vx,\vx^*)\cdot\vx^* = \sin\theta\vg_\ell(\vx,\vx^*)\cdot\vv\leq  -\mu^{(1)}\sin\theta -\mu^{(2)}\sin^2\theta,  
    \end{align*}
    which completes the proof.
\end{proof}

\begin{proof}[Proof of \Cref{lem:regression-loss-satisfy-sharpness-logconcave}]
    The proof is a direct consequence of \Cref{claim:nablarf-x*-bounded-by-sin2theta-log-concave}, combined with \Cref{assum:smoothness-of-regression-loss} to bound $\norm{\vg_\ell(\vx, \vx^*)}_2$.  In particular, if \Cref{assum:smoothness-of-regression-loss}(i) holds, then \Cref{claim:nablarf-x*-bounded-by-sin2theta-log-concave} implies $\vg_\ell(\vx,\vx^*)\cdot\vx^*\leq -(\mu^{(1)}/L_1)\sin\theta\|\vg_\ell(\vx,\vx^*)\|_2 -(\mu^{(2)}/L_1)\sin^2\theta\|\vg_\ell(\vx,\vx^*)\|_2$. 
On the other hand, if \Cref{assum:smoothness-of-regression-loss}(ii) holds, observing that $\|\vx - \vx^*\|_2 = 2\sin(\theta/2)\leq 2\sin\theta$ and 
    applying \Cref{claim:nablarf-x*-bounded-by-sin2theta-log-concave}, we get $\vg_\ell(\vx,\vx^*)\cdot\vx^*\leq -((\mu^{(1)} + \sin\theta\mu^{(2)})/(2L_2))\|\vg_\ell(\vx, \vx^*)\|_2$.
\end{proof}

\subsubsection{Sufficient Conditions under General-Covariance Gaussian Distributions}\label{subsec:Gaussian-sigma}

When the distribution $\D$ is specialized to a centered Gaussian, $\D = \calN(\vzero, \mSigma)$, simpler sufficient conditions can be established for \eqref{problem:general-regression-formulation} to satisfy RGA. Throughout this section, the covariance matrix $\mSigma$ may be singular. Let $\vp \sim \cN(\vzero, \mI_d)$ so that $\vz = \mSigma^{1/2}\vp \sim \cN(\vzero, \mSigma).$ When the loss derivatives appearing in this setting are distributional, their expectations are interpreted through this standard-Gaussian representation. Specifically, we assume that $\partial_1 \ell((\mSigma^{1/2}\vx)^\top \vp, (\mSigma^{1/2}\vx^*)^\top \vp; \vx, \vx^*)$ satisfies the integrability hypotheses of \Cref{fact:stein} and that the distributional chain rule
\begin{align*}
    D_{\vp}\partial_1\ell((\mSigma^{1/2}\vx)^\top \vp, (\mSigma^{1/2}\vx^*)^\top \vp; \vx, \vx^*) =\;& D_1 \partial_1\ell((\mSigma^{1/2}\vx)^\top \vp, (\mSigma^{1/2}\vx^*)^\top \vp; \vx, \vx^*) \mSigma^{1/2}\vx\\
    &+ D_2 \partial_1\ell((\mSigma^{1/2}\vx)^\top \vp, (\mSigma^{1/2}\vx^*)^\top \vp; \vx, \vx^*) \mSigma^{1/2}\vx^*
\end{align*}
holds, with all resulting Gaussian pairings being finite. 

In particular, consider the condition summarized below.  
\begin{assumption}\label{assum:ell-Gaussian-sigma}
    Let $\D = \calN(\vzero,\mSigma)$ for some positive-semidefinite matrix $\mSigma$ and $\vx^*\in\spd$ be associated with a given \eqref{problem:general-regression-formulation} problem. Given $\vx \in \spd$ such that $\theta\eqdef\theta(\vx,\vx^*)\in(0,\pi/2)$,  denote $\vv = \vx^*_{\perp\vx}/\|\vx^*_{\perp\vx}\|_2$. There exist non-negative reals $\mu_\theta^{(1)}, \mu_\theta^{(2)}\geq 0$, $\mu_\theta^{(1)} + \mu_\theta^{(2)} >0$ (possibly dependent on $\theta$) such that:
    \begin{gather*}
        \bigg(\Evz[D_1\partial_1\ell(\vx\cdot\vz,\vx^*\cdot\vz;\vx,\vx^*)] + \cos\theta\Evz[D_2\partial_1\ell(\vx\cdot\vz,\vx^*\cdot\vz;\vx,\vx^*)]\bigg)(\vx^\top\mSigma\vv) \leq -\mu_\theta^{(1)}\\
        \Evz[D_2\partial_1\ell(\vx\cdot\vz,\vx^*\cdot\vz;\vx,\vx^*)](\vv^\top\mSigma\vv)\leq -\mu_\theta^{(2)}.
    \end{gather*} 
\end{assumption}

{Note here that \Cref{assum:ell-Gaussian-sigma} uses distributional derivatives, which we defined and discussed in \Cref{fact:distributional-derivatives}.}
We argue that when \Cref{assum:ell-Gaussian-sigma} and \Cref{assum:smoothness-of-regression-loss} are satisfied,  the Riemannian correlation vector field $\vg_\ell$ satisfies RGA.
\begin{theorem}\label{lem:regression-loss-satisfy-sharpness-Gaussian-sigma}
    Given \eqref{problem:general-regression-formulation} and a corresponding Riemannian correlation vector field $\vg_\ell(\vx,\vx^*)$, 
    suppose that \Cref{assum:ell-Gaussian-sigma} holds for the associated target $\vx^* \in \spd$ and some $\vx \in \spd$.   
    \begin{enumerate}
        \item If \Cref{assum:smoothness-of-regression-loss}(i) is satisfied, then $\vg_\ell(\vx,\vx^*)\cdot\vx^*\leq -((\mu_\theta^{(1)} + \sin\theta\mu_\theta^{(2)})/L_1)\sin\theta\|\vg_\ell(\vx,\vx^*)\|_2$; 
        \item If \Cref{assum:smoothness-of-regression-loss}(ii) is satisfied, then $\vg_\ell(\vx,\vx^*)\cdot\vx^*\leq -((\mu_\theta^{(1)} + \sin\theta\mu_\theta^{(2)})/(2L_2))\|\vg_\ell(\vx,\vx^*)\|_2$. 
    \end{enumerate}
\end{theorem}
\begin{proof}
    The proof is a direct consequence of \Cref{claim:nablarf-x*-bounded-by-sin2theta} (stated and proved below), following the same reasoning as in the proof of \Cref{lem:regression-loss-satisfy-sharpness-logconcave}. \end{proof}

It remains to state and prove the following lemma. 
\begin{lemma}\label{claim:nablarf-x*-bounded-by-sin2theta}
Given \eqref{problem:general-regression-formulation} and a corresponding Riemannian correlation vector field $\vg_\ell(\vx,\vx^*)$, 
    suppose that \Cref{assum:ell-Gaussian-sigma} holds for target  $\vx^* \in \spd$ and some $\vx \in \spd$.  Then, 
    \begin{equation*}
        \vg_\ell(\vx,\vx^*)\cdot\vx^*\leq -\mu_\theta^{(1)}\sin\theta - \mu_\theta^{(2)}\sin^2\theta.
    \end{equation*}
\end{lemma}
\begin{proof}
    Let $\mLambda = \mSigma^{1/2}$ and $\vz = \mLambda\vp$ for $\vp\sim\calN(\vzero,\mI)$. Let $\mP_\vx = \mI - \vx\vx^\top$ be the projection matrix mapping to the space orthogonal to $\vx$. Writing down the definition of $\vg_\ell(\vx,\vx^*)$ and substituting $\vz$ by $\mLambda\vp$, we have
    \begin{align*}
        \vg_\ell(\vx,\vx^*) &= \mP_\vx\Evz[\partial_1\ell(\vx\cdot\vz,\vx^*\cdot\vz;\vx,\vx^*)\vz]=\mP_\vx\E_{\vp\sim\calN(0,\mI)}[\partial_1\ell((\mLambda\vx)^\top\vp,(\mLambda\vx^*)^\top\vp;\vx,\vx^*)\mLambda\vp].
    \end{align*}
    Applying Stein's lemma (\Cref{fact:stein}), we get:
    \begin{align*}
        \vg_\ell(\vx,\vx^*) &= \mP_\vx\mLambda\E_{\vp\sim\calN(0,\mI)}[\partial_1\ell((\mLambda\vx)^\top\vp,(\mLambda\vx^*)^\top\vp;\vx,\vx^*)\vp]\\
        &=\mP_\vx\mLambda\bigg(\E_{\vp\sim\calN(0,\mI)}\bigg[D_1\partial_1\ell((\mLambda\vx)^\top\vp,(\mLambda\vx^*)^\top\vp;\vx,\vx^*)\bigg]\mLambda\vx\\
        &\quad + \E_{\vp\sim\calN(0,\mI)}\bigg[D_2\partial_1\ell((\mLambda\vx)^\top\vp,(\mLambda\vx^*)^\top\vp;\vx,\vx^*)\bigg]\mLambda\vx^*\bigg)\\
        &=\Evz\bigg[D_1\partial_1\ell(\vx\cdot\vz,\vx^*\cdot\vz;\vx,\vx^*)\bigg]\mP_\vx\mSigma\vx + \Evz\bigg[D_2\partial_1\ell(\vx\cdot\vz,\vx^*\cdot\vz;\vx,\vx^*)\bigg]\mP_\vx\mSigma\vx^*.
    \end{align*}
    Therefore, the inner product $\vg_\ell(\vx,\vx^*)\cdot\vv$ equals:
    \begin{align}\label{eq:Gaussian-sigma-expression-vg-cdot-v}
        \vg_\ell(\vx,\vx^*)\cdot\vv 
        &=\Evz[D_1\partial_1\ell(\vx\cdot\vz,\vx^*\cdot\vz;\vx,\vx^*)]\vx^\top\mSigma\vv  + \Evz[D_2\partial_1\ell(\vx\cdot\vz,\vx^*\cdot\vz;\vx,\vx^*)]\vv^\top\mSigma\vx^* \nonumber\\
        &=\bigg(\Evz[D_1\partial_1\ell(\vx\cdot\vz,\vx^*\cdot\vz;\vx,\vx^*)] + \cos\theta\Evz[D_2\partial_1\ell(\vx\cdot\vz,\vx^*\cdot\vz;\vx,\vx^*)]\bigg)\vx^\top\mSigma\vv \nonumber\\
        &\quad + \Evz[D_2\partial_1\ell(\vx\cdot\vz,\vx^*\cdot\vz;\vx,\vx^*)]\vv^\top\mSigma\vv\sin\theta.
    \end{align}
    Then by \Cref{assum:ell-Gaussian-sigma}, it holds that $\vg_\ell(\vx,\vx^*)\cdot\vv\leq -\mu_\theta^{(1)} - \mu_\theta^{(2)}\sin\theta.$
    It remains to use that $\vg_\ell(\vx,\vx^*)\cdot\vx^* = \sin\theta\,\vg_\ell(\vx,\vx^*)\cdot\vv$. 
\end{proof}

\subsubsection{Sufficient Conditions under Standard Gaussian Distribution}\label{subsec:Gaussian}

In this subsection, we focus on the case where the rows of $\mZ\sim\D$ are independent  Gaussian random vectors, i.e., $\mZ^\top = [\vz^1,\dots,\vz^m]$, $\vz_i\sim\calN(\vec 0, \vec I)$. We show that a simple negative curvature condition (\Cref{assum:ell-Gaussian}) is sufficient for the RGA property to hold with a constant RGA parameter (\Cref{lem:regression-loss-satisfy-sharpness-Gaussian}), regardless of the smoothness of the vector field.

\begin{assumption}\label{assum:ell-Gaussian}
Let $\D$ be the distribution of random $m \times d$ matrices with i.i.d.\ standard Gaussian entries and $\vx^*\in\spd$ be the target solution, both associated with a given \eqref{problem:general-regression-formulation} problem. Let $\vx\in\spd$ be a given vector such that $\theta(\vx,\vx^*)\in(0,\pi/2)$. 
Assume that $\ell(\cdot,\cdot;\vx,\vx^*)$ satisfies $$\E_{\mZ\sim\D}[\tr(D_2\partial_1\ell(\mZ\vx, \mZ\vx^*;\vx,\vx^*))]<0.$$
\end{assumption}
Observe here that \Cref{assum:ell-Gaussian} can be seen as a second-order condition. It is, however, substantially different than typical second-order conditions like strong convexity or `restricted' strong convexity (where the definition of strong convexity applies with one of the points restricted to $\mZ\vx^*$) used in traditional Euclidean optimization. First, the matrix $D_2\partial_1\ell(\mZ\vx, \mZ\vx^*;\vx,\vx^*)$ is not necessarily symmetric. Second,  $D_2\partial_1\ell(\mZ\vx, \mZ\vx^*;\vx,\vx^*)$ does not contain second derivatives w.r.t.\ the vector of variables $\mZ\vx.$ Instead, it is obtained by sequentially taking partial derivatives w.r.t.\ the first and the second argument (the latter is fixed to $\mZ\vx^*$). Third, the condition does not impose any restrictions on individual singular values of $D_2\partial_1\ell(\mZ\vx, \mZ\vx^*;\vx,\vx^*)$; all that is required is that the trace of $\E_{\mZ \sim \cD}[D_2\partial_1\ell(\mZ\vx, \mZ\vx^*;\vx,\vx^*)]$ is strictly negative. 
Finally, the condition only involves the differentiation of the random variables $\mZ\vx$ and $\mZ\vx^*$, without accounting for the implicit dependence of the loss on $\vx,\vx^*$.

\Cref{assum:ell-Gaussian} alone implies RGA, as stated and proved below. 
\begin{theorem}[Sufficient Conditions for RGA Under Standard Gaussian Distribution]\label{lem:regression-loss-satisfy-sharpness-Gaussian}
Given \eqref{problem:general-regression-formulation} and a corresponding Riemannian correlation vector field $\vg_\ell(\vx,\vx^*)$ and target solution $\vx^*\in\spd$,  
     for any $\vx$ such that $\theta = \theta(\vx,\vx^*)\in(0, \pi/2)$ and \Cref{assum:ell-Gaussian} holds, we have: 
$$\vg_\ell(\vx,\vx^*)\cdot\vx^*= - \|\vg_\ell(\vx,\vx^*)\|_2\sin\theta.$$
    In other words, the RGA property holds with $\mu=r = 1.$
\end{theorem}
\begin{proof}
Define $\vv = \vx^*_{\perp\vx}/\|\vx^*_{\perp\vx}\|_2$.
We first argue that the vector field $\vg_\ell$ is always parallel to $\vv$. Note that $\vg_\ell$ is orthogonal to $\vx$ by its definition. Let $\vu$ be any direction orthogonal to both $\vv$ and $\vx$; then using the definition of $\vg_\ell$, we have:
\begin{align*}
    \vg_\ell(\vx,\vx^*)\cdot\vu = \E_{\mZ\sim\D}[\partial_1\ell(\mZ\vx,\mZ\vx^*;\vx,\vx^*)^\top\mZ_{\perp\vx}\vu] = \E_{\mZ\sim\D}[\partial_1\ell(\mZ\vx,\mZ\vx^*;\vx,\vx^*)]^\top\E_{\mZ\sim\D}[\mZ\vu] = 0,
\end{align*}
as $\mZ\vu$ is independent of $\mZ\vx$ and $\mZ\vx^*$ and has mean zero. Thus, we conclude that $\vg_\ell(\vx,\vx^*)$ is parallel to $\vv$.

Next, we prove that:
\begin{align*}
    \vg_\ell(\vx,\vx^*)\cdot\vv = \sin\theta\E_{\mZ\sim\D}[\tr(D_2\partial_1\ell(\mZ\vx, \mZ\vx^*;\vx,\vx^*))].
\end{align*}
To this aim, we apply Stein's lemma (\Cref{fact:stein}).
Note that by the definition of $\vv$, we can decompose $\vx^* = \vx\cos\theta + \vv \sin\theta$ and $(\mI - \vx\vx^\top)\vx^* = (\vx^*)_{\perp\vx} = \sin\theta\vv$. 
Then, using the definition of $\vg_\ell(\vx,\vx^*)$ again, a direct calculation yields:
    \begin{align}\label{eq:nablarf-x*-bounded-by-sin2theta-Gaussian-0}
        \vg_\ell(\vx,\vx^*)\cdot\vv& =\E_{\mZ\sim\D}[\partial_1\ell(\mZ\vx,\mZ\vx^*;\vx,\vx^*)^\top\mZ\vv] \nonumber\\
        & = \E_{\mZ\sim\D}\bigg[\sum_{i=1}^m (\partial_1\ell(\mZ\vx, \mZ\vx\cos\theta + \mZ\vv \sin\theta;\vx,\vx^*))_i(\vz^i\cdot\vv)\bigg] \nonumber\\
        &=\sum_{i=1}^m \E_{\mZ\vx\sim\calN(\vec 0, \mI)}\bigg[\E_{\mZ\vv\sim\calN(\vec 0, \mI)}\bigg[\bigg(\partial_1\ell(\mZ\vx, \mZ\vx\cos\theta + \mZ\vv\sin\theta;\vx,\vx^*)\bigg)_i(\vz^i\cdot\vv)\bigg| \,\mZ\vx\bigg]\bigg],
    \end{align}
    where in the last equation we used the tower rule of expectations.  
    Given $\mZ\vx$, $(\partial_1\ell(\mZ\vx, \mZ\vx\cos\theta + \mZ\vv\sin\theta;\vx,\vx^*))_i$ is a function of $\vz^1\cdot\vv,\dots,\vz^m\cdot\vv$. Therefore, we can denote $h_i(\vz^1\cdot\vv,\dots,\vz^m\cdot\vv)\eqdef (\partial_1\ell(\mZ\vx, \cos\theta\mZ\vx + \sin\theta\mZ\vv;\vx,\vx^*))_i$, $h_i:\R^m\to\R$. Note that $\mZ\vx$ and $\mZ\vv$ are independent standard Gaussian vectors since $\vv$ and $\vx$ are orthogonal to each other. Further, $\vz^i$'s are also independent standard Gaussian random vectors. Therefore, applying Stein's lemma (\Cref{fact:stein}),  we obtain
    \begin{align*}
        &\quad \E_{\mZ\vv\sim\calN(\vec 0, \mI)}\bigg[h_i(\vz^1\cdot\vv,\dots,\vz^m\cdot\vv)(\vz^i\cdot\vv)\bigg]\\
        &=\E_{\vz^j\cdot\vv\sim\calN(0, 1), j\neq i}\bigg[\E_{\vz^i\cdot\vv\sim\calN(0,1)}\bigg[h_i(\vz^1\cdot\vv,\dots,\vz^m\cdot\vv)(\vz^i\cdot\vv)\bigg|\,\vz_j\cdot\vv, j\neq i\bigg]\bigg]\\
        &=\E_{\vz^j\cdot\vv\sim\calN(0, 1), j\neq i}\bigg[\E_{\vz^i\cdot\vv\sim\calN(0,1)}\bigg[D_{\vz^i\cdot\vv}h_i(\vz^1\cdot\vv,\dots,\vz^m\cdot\vv)\bigg|\,\vz_j\cdot\vv, j\neq i\bigg]\bigg] \\
        &=\E_{\mZ\vv\sim\calN(\vec 0, \mI)}\bigg[D_{\vz^i\cdot\vv} h_i(\vz^1\cdot\vv,\dots,\vz^m\cdot\vv)\bigg].
    \end{align*}
    Now, using the chain rule, the partial derivative of $h_i$ equals:
    \begin{align*}
        D_{\vz^i\cdot\vv} h_i(\vz^1\cdot\vv,\dots,\vz^m\cdot\vv) &= \bigg(D_{2}(\partial_1 \ell(\mZ\vx, \mZ\vx\cos\theta + \mZ\vv\sin\theta;\vx,\vx^*))_i \bigg)^\top \bigg(\frac{\partial }{\partial \vz^i\cdot\vv} (\mZ\vv\sin\theta)\bigg) \\
        &= \sin\theta \big(D_2\partial_1\ell(\mZ\vx, \mZ\vx^*;\vx,\vx^*)\big)_{i,i}.
    \end{align*}
    Plugging this equality back into \Cref{eq:nablarf-x*-bounded-by-sin2theta-Gaussian-0}, we obtain
    \begin{align*}
        \vg_\ell(\vx,\vx^*)\cdot\vv = \sin\theta \sum_{i=1}^m \E_{\mZ\sim\D}[(D_2\partial_1\ell(\mZ\vx, \mZ\vx^*;\vx,\vx^*))_{i,i}] = \sin\theta\E_{\mZ\sim\D}[\tr(D_2\partial_1\ell(\mZ\vx, \mZ\vx^*;\vx,\vx^*))].
    \end{align*}

    Now since $\vg_\ell(\vx,\vx^*)$ is parallel to $\vv$ and $\E_{\mZ\sim\D}[\tr(D_2\partial_1\ell(\mZ\vx, \mZ\vx^*;\vx,\vx^*))]<0$ by  \Cref{assum:ell-Gaussian}, we conclude that:
    \begin{align*}
        \|\vg_\ell(\vx,\vx^*)\|_2 = -\sin\theta\E_{\mZ\sim\D}[\tr(D_2\partial_1\ell(\mZ\vx, \mZ\vx^*;\vx,\vx^*))] = -\vg_\ell(\vx,\vx^*)\cdot\vv.
    \end{align*}
    Finally, since $\vx^*_{\perp\vx} = \vv \sin\theta$, we obtain:
    \begin{align*}
        \vg_\ell(\vx,\vx^*)\cdot\vx^* = \vg_\ell(\vx,\vx^*)\cdot(\sin\theta\vv) = -\|\vg_\ell(\vx,\vx^*)\|_2\sin\theta,
    \end{align*}
    completing the proof.
\end{proof}

\subsection{Sufficient Conditions for General Vector Fields}\label{subsec:sufficient-condition-general-vector-field}

In this subsection, we consider general vector fields $\vg(\vx, \vy)$ that need not be associated with any minimization problem, deterministic or stochastic. Instead, $\vg$ can be an arbitrary map of pairs of vectors $\vx, \vy$ on the sphere to elements of $\mathrm{T}_\vx\spd.$ This means that the provided sufficient condition (\Cref{assum:ell-general-vector-field}) can be used not only for minimization problems (that are the focus of this paper), but also, possibly, for min-max optimization problems, variational inequalities, and root-finding problems on the sphere.
\begin{assumption}\label{assum:ell-general-vector-field}
Let $\vg(\vx,\vy)$ be a mapping from $\R^d\times\R^d$ to $\R^d$ such that $\vg(\vx,\vy)\in\mathrm{T}_\vx\spd$ for all $\vx,\vy\in\spd$.
Let $B,\mu$ be some positive constants. 
Fix $\vx^*\in\spd, 0<\bar{\theta}<\pi/2$. Assume all of the following:
\begin{enumerate}
    \item (Stationarity) $\vg(\vx,\vy) \mid_{\vy = \vx} = \vzero$, for all $\vx\in\spd$.\item (Bounded Curvature) 
    For any $\vx\in\spd$ with $\theta = \theta(\vx,\vx^*)\in(0, \Bar{\theta}]$, define the geodesic 
    $$\cG_{\vx,\vx^*}(t) = \cos(t \theta) \vx + \sin(t \theta) \vv, t\in[0,1], \vv = {(\vx^*)_{\perp\vx}}/{\|(\vx^*)_{\perp\vx}\|_2}.$$
For each fixed $\vx$, the map $\vy \mapsto \vg(\vx, \vy)$ has a locally Lipschitz extension $\tilde{\vg}_\vx: \cU_\vx \to \R^d$ to an open Euclidean neighborhood $\cU_\vx$ of the geodesic $\cG_{\vx,\vx^*}(t), t\in[0,1]$. 
    Furthermore, for $\vy\in\cU_\vx$, define the Clarke Jacobian by:
\begin{align*}
    \partial_{2}^C \vg(\vx,\vy)\eqdef \partial^C \tilde{\vg}_\vx(\vy) = \mathrm{Conv}\bigg\{\lim_{k\to\infty} \partial_{2} \tilde{\vg}_\vx(\vy_k): \vy_k\to\vy, \text{ $\tilde{\vg}_\vx(\vy_k)$ differentiable at $\vy_k$}\bigg\}.
\end{align*}
   For almost every $t\in[0,1]$ and any $\partial_{2}\vg(\vx,\vy(t))\in\partial^C_{2}\vg(\vx,\vy(t))$  
    \begin{equation*}
        \vv^\top \partial_{2}\vg(\vx,\vy(t))\vv\leq -\mu\;\text{ and }\;\|\partial_{2}\vg(\vx,\vy(t))\|_2\leq B.
    \end{equation*}
\end{enumerate}
\end{assumption}

Our main result from this subsection is summarized in the following theorem. Compared to the results from the previous subsection, this result is more `local', requiring the angle between $\vx$ and $\vx^*$ to be sufficiently small ($\tan(\theta(\vx,\vx^*)/2)\leq  \mu/(2B)$), thus, from the aspect of algorithms, appropriate initialization becomes important to ensure convergence of DS-RGD.

\begin{theorem}\label{lem:regression-loss-satisfy-sharpness-general-distr}
    Let $\vg,\vx,\vx^*$ satisfy \Cref{assum:ell-general-vector-field}, where $\theta = \theta(\vx,\vx^*)>0$ satisfies $\tan(\theta/2)\leq \min\{\mu/(2B), \tan(\bar{\theta}/2)\}$.
    Then, $\vg(\vx,\vx^*)\cdot\vx^*\leq -(\mu/(4B))\|\vg(\vx,\vx^*)\|_2\sin\theta$, i.e., $\vg(\vx,\vx^*)$ is $\mu/(4B)$-RGA.
\end{theorem}
\begin{proof}
    We first prove that $\vg(\vx,\vx^*)\cdot\vx^*\leq -(\mu/2)\sin^2\theta$ under \Cref{assum:ell-general-vector-field} when $\tan(\theta/2)\leq \min\{\mu/(2B), \tan(\bar{\theta}/2)\}$.

    Denote $\vv\eqdef (\vx^*)_{\perp\vx}/\|(\vx^*)_{\perp\vx}\|_2$. 
{Since $\vg(\vx,\vx^*)\in\mathrm{T}_\vx\spd$,} we have $\vg(\vx,\vx^*)\cdot\vx^* = \vg(\vx,\vx^*)\cdot\vv\sin\theta$.

    Now consider the geodesic $\vy(t)$ on the sphere defined by $\vy(t)=  \cos(t\theta)\vx + \sin(t\theta)\vv$. This curve satisfies $\vy(0) = \vx$ and $\vy(1) = \vx^*$. Furthermore, though $\vg(\vx,\vy(t))$ might not be differentiable for every $t$, since $\vg(\vx,\vy)$ is locally Lipschitz around $\vy(t)$, the vector field $\mF(t)\eqdef\vg(\vx,\vy(t))$ is absolutely continuous in $t$ on the geodesic, and hence the fundamental theorem of calculus (see, e.g., \cite{rudin-principles}) is applicable, leading to 
\begin{align*}
        \vg(\vx,\vy(1))&= \mF(1) = \mF(0) + \int_0^1 \mF'(t)\diff{t} = \vg(\vx,\vy(0)) + \int_0^1 \mF'(t)\diff{t}.
    \end{align*}
    By the stationarity part of \Cref{assum:ell-general-vector-field},  noting that by construction $\vy(0) = \vx$, we have $\vg(\vx,\vy(0)) = 0$. Therefore, only the second term in the equation above survives. In addition, for almost every $t\in[0,1]$, the Clarke chain rule gives some matrices $\partial_{2}\vg(\vx,\vy(t))\in\partial^C_{2} \vg(\vx,\vy(t))$ such that $\mF'(t) = \partial_{2}\vg(\vx,\vy(t))(\diff{}\vy(t)/\diff{t})$. 
By the definition of $\vy(t)$, we have $\vy'(t) = -\theta\sin(t\theta)\vx + \theta\cos(t\theta)\vv$. Therefore, plugging in, we get:
\begin{align*}
    \vg(\vx,\vx^*)\cdot\vv&=\int_0^1\partial_{2}\vg(\vx,\vy(t))\vy'(t)\cdot\vv\diff{t} \\
    &= -\underbrace{\int_{0}^1 \vx^\top \partial_{2} \vg(\vx,\vy(t))^\top\vv \sin(t\theta)\theta\diff{t}}_{Q_1} + \underbrace{\int_0^1 \vv^\top \partial_{2} \vg(\vx,\vy(t))^\top\vv \cos(t\theta)\theta\diff{t}}_{Q_2}.
\end{align*}
For the first term, since $\|\partial_{2} \vg(\vx,\vy(t))\|_2\leq B$ for a.e.\ $t\in[0,1]$ according to \Cref{assum:ell-general-vector-field}, we have 
\begin{align*}
    |Q_1|&\leq \int_{0}^1 \bigg|\vx^\top \partial_{2} \vg(\vx,\vy(t))^\top\vv\bigg| \sin(t\theta)\theta\diff{t}\leq B\int_{0}^1 \sin(t\theta)\theta\diff{t} = 2B\sin^2(\theta/2).
\end{align*}
On the other hand, for the term $Q_2$, using the curvature assumption in \Cref{assum:ell-general-vector-field}, which requires that $\vv^\top \partial_{2}\vg(\vx,\vy(t))^\top\vv\leq -\mu$ for a.e.\ $t\in[0,1]$, we obtain
\begin{align*}
    Q_2\leq -\mu\int_0^1 \cos(t\theta)\theta\diff{t} = -\mu\sin\theta.
\end{align*}
Hence, combining the upper bounds on $Q_1$ and $Q_2$, we have
\begin{align*}
    \vg(\vx,\vx^*)\cdot\vv\leq 2B\sin^2(\theta/2) - \mu\sin\theta = -2\sin(\theta/2)(\mu\cos(\theta/2) - B\sin(\theta/2))
\end{align*}
Therefore, whenever $\tan(\theta/2)\leq \mu/(2B)$, it holds
\begin{align*}
    \vg(\vx,\vx^*)\cdot\vx^* = \sin\theta\vg(\vx,\vx^*)\cdot\vv\leq -\mu\sin\theta\sin(\theta/2)\cos(\theta/2) = -(\mu/2)\sin^2\theta.
\end{align*}

Next, we prove that $\|\vg(\vx,\vx^*)\|_2\leq B\theta.$ This directly comes from the assumption that $\|\partial_2\vg(\vx,\vy(t))\|_2\leq B$ for a.e. $t\in[0,1]$ in \Cref{assum:ell-general-vector-field}. To see this, letting $\mF(t)= \vg(\vx,\vy(t))$ and using the fundamental theorem of calculus as well as the stationarity assumption in \Cref{assum:ell-general-vector-field}, again, we have:
\begin{align*}
    \|\vg(\vx,\vx^*)\|_2 = \bigg\|\vg(\vx,\vy(0)) + \int_0^1 \mF'(t)\diff{t}\bigg\|_2\leq \int_0^1 \|\partial_2\vg(\vx,\vy(t))\|_2\|\vy'(t)\|_2\diff{t}\leq B\int_0^1 \|\vy'(t)\|_2\diff{t}.
\end{align*}
Recalling that $\vy'(t) = \theta(-\sin(t\theta)\vx + \cos(t\theta)\vv)$, hence $\|\vy'(t)\|_2 = \theta$, we obtain $\|\vg(\vx,\vx^*)\|_2\leq B\theta$. 

We conclude that
\begin{align*}
    \vg(\vx,\vx^*)\cdot\vx^*\leq -\frac{\mu\sin\theta}{2B\theta}\|\vg(\vx,\vx^*)\|_2\sin\theta\leq -(\mu/(4B))\|\vg(\vx,\vx^*)\|_2\sin\theta,
\end{align*}
since $\sin\theta\geq \theta/2$ when $\theta\in(0,\pi/2)$.
\end{proof}

\section{Example Problems that Satisfy Sufficient Conditions}\label{sec:examples}

In this section, we apply the sufficient conditions developed in \Cref{sec:sufficient conditions} to concrete learning and estimation problems. For each problem, we identify an appropriate vector field and verify the relevant sufficient conditions, thereby obtaining an explicit RGA parameter. 
The discussion in this section unifies a wide range of results, ranging from dictionary learning ~\cite{sun2015complete,bai2018subgradient,gilboa2019efficient,Li2019WeaklyCO,qu2019analysis,Zhu2019LinearConvergenceGrassman} to a line of work on learning SIMs/GLMs~\cite{mei2018landscape,diakonikolas2020approximation,yehudai2020learning,diakonikolas2022learning,WZDD2023,WZDD24b,ZWDD2024,zarifis2025robustly,wang2026robustly}, additionally providing extensions of results for learning SIMs/GLMs under structured distributional families \cite{WZDD2023,WZDD24b,ZWDD2024,zarifis2025robustly,wang2026robustly}, and new results including: multi-instance learning/max pooling~\cite{andrews2002SVMforMIL,SolveMultiInstanceProblem1997,carbonneau2018multiple,gupta25learningMIL,levien2026based} under the Gaussian marginal, learning Gaussian SIMs and phase retrieval under the nonsmooth $\ell_1$ loss, learning SIMs under logconcave distributions with a variety of vector fields, and learning GLMs under near-isotropic distributions for $L_2^2$ and $\alpha$-Huber loss.

Observe that to fully specify \eqref{problem:general-regression-formulation}, we need to specify: (P1) the covariate distribution $\D$, (P2) the value of $m$, and (P3) the loss function $\ell.$  
We begin this section by providing the necessary background and definitions used in the rest of the section. We then prove the RGA property for various problem classes defined by specifying parameters (P1)-(P3) of \eqref{problem:general-regression-formulation}. We group the problem classes by the assumptions on the covariate distribution $\D$, as stated in (P1). \Cref{sec:ex:standard Gaussian} and \Cref{sec:ex:general Gaussian} provide results under Gaussian distributions, while \Cref{sec:ex:general distributions} provides results under more general (structured) distributions.

\subsection{Definitions and Technical Considerations}\label{sec:defs-for-examples}

In this subsection, we introduce example problems and provide necessary definitions pertaining to loss/penalty functions discussed in the rest of the section.

\paragraph{Example Problems.}

We first consider a planted Gaussian specialization of multiple-instance halfspace learning that is captured by \eqref{problem:general-regression-formulation} for $m \geq 1$.
The learner is given a set of bags where each bag contains $m$ independent data vectors, with each bag  labeled positive  if it intersects an unknown target halfspace. 
\begin{problem}[Multi-Instance Learning/Max-pooling (MIL)]\label{prob:mil}
    Let $\mZ\in\R^{m\times d}$ be a matrix drawn from a distribution $\D$, $\mZ^\top = [\vz_1,\dots,\vz_m]$. Let $t\in\R$ be a given parameter
    \footnote{Throughout, we assume that $t$ is known. Under suitable anti-concentration assumptions, unknown $t$ can be handled using a one-dimensional grid search; see  \cite{DDKWZ2023}. }. 
    For a bag of vectors $\{\vz_i,i\in[m]\}$, the label $y$ is defined by $y = \sign(\max_{i\in[m]}\vx^*\cdot\vz_i + t)$. 
    The MIL asks to find a vector $\vx\in\spd$ that minimizes the objective $f_h(\vx) \eqdef \E_{\mZ\sim\D}[h(-y\max_{i\in [m]}(\vx\cdot\vz_i + t))]$,
    where $h:\R\to\R$ is a penalty function.
\end{problem}

Namely, the label of the bag is positive if it contains at least one positive element, and the label is negative if there are no positive elements in the bag.
The existential bag-labeling rule is classical in MIL, while hard-max bag-level hinge penalty and logistic penalty underlie MI-SVM and probabilistic or neural max-pooling MIL; see e.g., \cite{andrews2002SVMforMIL}.
MIL problems have numerous applications in areas such as  drug activity prediction and image and video classification \cite{SolveMultiInstanceProblem1997,carbonneau2018multiple}. 
In this work, we analyze the resulting population objective over the unit sphere.
Prior works that are most closely related to our formulation are 
\cite{gupta25learningMIL}, which also considers planted linear regression with i.i.d.\ Gaussian bags, but their bag label is generated by a uniformly selected primary instance; and Levien’s recent query–value model \cite{levien2026based}, which contains Gaussian max-linear regression as an aligned special case, but does not study our thresholded margin risk or its spherical geometry.

\Cref{prob:mil} is a generalization of the classical problem of learning halfspaces/binary linear classification. When the number of elements in the bag $m$ is equal to 1, the label $y =\sign(\vx^*\cdot\vz + t)$ is exactly the halfspace separating positive and negative samples. Formally, we have:
\begin{problem}[Learning General Halfspaces]\label{prob:halfspace}
    Let $\D$ be a distribution on $\R^d$. Given sample access to $(\vz,y)$ where $\vz\sim\D$ and $y = \sign(\vx^*\cdot\vz + t)$ for some given $t$, let $h:\R\to\R$ be a penalty function. Find a parameter vector $\vx\in\spd$ that minimizes $f_h(\vx) = \Evz[h(-y(\vx\cdot\vz + t))]$.
\end{problem}

When $m = 1$, \eqref{problem:general-regression-formulation} 
is also an abstraction of the widely-studied regression problems corresponding to generalized linear models and single index models~\cite{kakade2011efficient,mei2018landscape,diakonikolas2020approximation,yehudai2020learning,BenArous2021OnlineSGD,diakonikolas2022learning,WZDD2023,GGKS23,damian2023smoothing,WZDD24b,ZWDD2024,hu2025omnipredicting,zarifis2025robustly,wang2026robustly}, defined below. \begin{problem}[Generalized Linear Models and Single Index Models]\label{prob:glm+sim}
    Let $\mathcal{F}$ be a class of real non-linear univariate functions $\sigma:\R\to\R$, referred to as ``link'' or ``activation'' functions. Let $\D$ be a (often unknown, sometimes restricted to standard Gaussian) distribution on $\R^d$ that belongs to some distribution class $\mathfrak{D}$. Let $h:\R\to\R$ be a given penalty function. \begin{enumerate}
        \item Generalized Linear Model (GLM): Given a function $\sigma\in\mathcal{F}$ and sample access to data $(\vz,y)$, where $\vz\sim\D$ and $y = \sigma(\vx^*\cdot\vz)$ for some $\vx^*\in\spd$, find a parameter vector $\vx\in\spd$ that minimizes $f(\vx) = \Ez[h(\sigma(\vx\cdot\vz) - \sigma(\vx^*\cdot\vz))]$.
        \item Single Index Model (SIM): Given sample access to data $(\vz,y)$, where $\vz\sim\D$ and $y = \sigma^*(\vx^*\cdot\vz)$ for some $\vx^*\in\spd$ and $\sigma^*\in\mathcal{F}$, find a parameter vector $\vx\in\spd$ and a candidate activation $\sigma\in\mathcal{F}$ that minimize $f(\vx) = \inf_{\sigma\in \mathcal{F}}\Ez[h(\sigma(\vx\cdot\vz) - \sigma^*(\vx^*\cdot\vz))]$.
    \end{enumerate}
\end{problem}
Observe that in GLM problems, the activation function is explicitly given as part of the problem, and so it suffices to find the parameter vector $\vx$ to have a fully specified model (i.e., we are in the setting of \emph{parametric regression}). In SIM problems, the activation function is not known, so the problem involves finding both the parameter vector and the unknown activation (i.e., we are in the setting of \emph{semi-parametric regression}).

Another related optimization problem on the sphere concerns sparse recovery, where the goal is to learn a sparse code of the features that represents the label. We consider the complete orthogonal dictionary learning problem under the sparse Gaussian-Bernoulli distribution, following a long line of related work \cite{sun2015complete,bai2018subgradient,gilboa2019efficient,Li2019WeaklyCO,qu2019analysis,Zhu2019LinearConvergenceGrassman}.
\begin{problem}[Complete Orthogonal Dictionary Learning (COD) Under Sparse Gaussian-Bernoulli Distribution]\label{prob:dictionary}
    Let $\mA = [\va_1,\dots,\va_d]\in\R^{d\times d}$ be an orthonormal dictionary, where each column $\va_i$ represents a letter. Given words $\vy|\vz$ generated by $\vy = \mA\vz$, where $\vz$ follows the $\eta$-Gaussian-Bernoulli distribution: $\vz = \vp\odot\vb$, $\vp\sim\calN(0,\mI_d)$ and $\vb = (b_1,\dots,b_d)$, $b_i \sim \mathrm{Bernoulli}(\eta)$, $\eta\in(0,1)$, COD asks to learn one letter of the dictionary $\mA$ via optimizing the loss $f_{\mathsf{dic}}(\vx) = \Evz[|\vx\cdot\vy(\vz)|]$ on the sphere $\vx\in\spd$. 
    
    Since $\vy= \mA\vz$ and $\mA$ is an orthonormal matrix, we have $\vx\cdot\vy = (\mA^\top\vx)\cdot\vz = \Tilde{\vx}\cdot\vz$, where $\Tilde{\vx}$ is also on the unit sphere. Therefore, optimizing $f_{\mathsf{dic}}(\vx)$ is equivalent to minimizing $f_{\mathsf{dic}}(\tilde{\vx}) = \Evz[|\tilde{\vx}\cdot\vz|]$ (see \cite{bai2018subgradient}). We use this transformed formulation and focus on optimizing $f_{\mathsf{dic}}({\vx}) = \Evz[|{\vx}\cdot\vz|]$.
\end{problem}

{
In addition to dictionary learning, we consider another sparse regression problem:
\begin{problem}[Spiked Sparse Covariance]\label{prob:spiked-sparse-pca}
Let $\mSigma = \mI + \beta\vx^*\vx^{*\top}$, where $\beta>0$ and $\vx^*\in\spd$. In particular, $\vx^*$ is a sparse vector with equal magnitudes; i.e., let $\cI\subset[d]$ be a set of indices with cardinality $|\cI| = s$, $\vx^* = (x^*_1,\dots,x^*_d)$ satisfies $x^*_i \in \{-1/\sqrt{s},1/\sqrt{s}\}$ for $i\in\cI$ and $x^*_i = 0$ otherwise. Minimize the lasso-regularized Rayleigh quotient to recover the planted sparse vector $\vx^*$: $f_\lambda(\vx) = -\vx^\top\mSigma\vx + \lambda\|\vx\|_1$, $\lambda\geq 0$.
\end{problem}

This formulation of the spiked sparse covariance problem is deterministic and does not fall into the model of \eqref{problem:general-regression-formulation} at a first glance. In \Cref{claim:convert-spiked-sparse-pca-to-regression}, we show that \Cref{prob:spiked-sparse-pca} can be transformed into a regression problem (in the form of \eqref{problem:general-regression-formulation}) whose data follow Gaussian distributions with general covariances.
Using the regression formulation of \Cref{prob:spiked-sparse-pca} derived in \Cref{claim:convert-spiked-sparse-pca-to-regression}, we apply our sufficient conditions (\Cref{assum:ell-Gaussian-sigma} and \Cref{lem:regression-loss-satisfy-sharpness-Gaussian-sigma}) and prove that locally around the planted sparse vector $\vx^*$, any vector field $\vg(\vx)\in\partial^{C,R}f_\lambda(\vx)$ is $(\mu,2)$-RGA where $\mu \gtrsim (\beta \cos(\bar{\theta})- \lambda\sqrt{s})/(\beta + \lambda\sqrt{d})$ ($\bar{\theta}$ denotes the upper bound on the distance from $\vx$ to the target $\vx^*$, see \Cref{lem:spiked-pca-order-2-rga} for details). 
This higher-order characterization of the gradient field is tight---as we prove in \Cref{lem:spiked-sparse-pca-cannot-have-lower-order}, RGA cannot hold with any order $r<2$ near $\vx^*$.

Sparse PCA has been studied extensively in the past decades; see, for example~\cite{jolliffe2003modified,zou2006sparse,zou2018selective,chen2020proximal}. 
Our result is mostly orthogonal to prior work. 
Most closely related to our results is~\cite {chen2020proximal}, which studied general non-smooth composite Riemannian optimization problems on the Stiefel manifolds (where the lasso-regularized Rayleigh quotient serves as a special example). They proposed a proximal gradient-type algorithm, called ManPG, that converges to an $\eps$-stationary point in $O(1/\eps^2)$ steps globally.
The stationary point to which their algorithm converges is not guaranteed to be the planted sparse vector in the problem formulation; note that, by comparison, our result guarantees iterate convergence (i.e., convergence to solutions).   
Their iteration complexity matches our result for RGD when order-2 RGA is satisfied (see \Cref{lem:descent-rsi}); however, ManPG requires solving  subproblems with second-order methods in each iteration. 
We further point out that, unlike prior work, our result characterizes the benign landscape of the objective. 
Namely, for the spiked-sparse covariance model in \Cref{prob:spiked-sparse-pca}, we prove that all Clarke subgradient fields of the non-smooth objective are structured and benign (in the RGA sense) near the target planted sparse vector.
Furthermore, our characterization of this benign structure is tight, as shown in \Cref{lem:spiked-sparse-pca-cannot-have-lower-order}.
The identification of this benign landscape was not previously known.
}

{
Our last problem is the eigenvector/PCA problem of a real symmetric matrix $\mA$, arguably the most basic optimization problem on the sphere.
\begin{problem}[Eigenvector Problems/Optimizing Rayleigh Quotient]\label{prob:eigenvector-problem}
    Let $\mA$ be a real symmetric matrix in $\R^{d\times d}$ with eigenvalues $\lambda_1\geq \lambda_2\geq \dots\geq \lambda_{r-1}>\lambda_r = \dots = \lambda_d$, and corresponding eigenvectors $\vq_1,\dots,\vq_r,\dots,\vq_d$. Define the Rayleigh quotient function $f_{\mathsf{RQ}}(\vx) = \vx^\top\mA\vx$. The goal is to find one vector in the eigenspace of $\mathrm{Span}(\vq_r,\dots,\vq_d)$ by optimizing $f(\vx)$ on the unit sphere.
\end{problem}
Though \Cref{prob:eigenvector-problem} is defined deterministically, similar to \Cref{prob:spiked-sparse-pca}, we convert the problem into a regression problem in the form of \eqref{problem:general-regression-formulation}, via which we show that the Rayleigh Quotient function is RGA, recovering the results from~\cite{zhang2016riemannian,ADVA21,alimisis2022geodesic};  see details in \Cref{sec:eigenvector-problem-and-optimizing-rayleigh-quotient}.
}

\subsection{Example Regression Problems under the Standard Gaussian Distribution}\label{sec:ex:standard Gaussian}

We begin with examples under the standard Gaussian distribution. The first set of examples is for multi-instance learning, where the number of rows in the random Gaussian matrix $\mZ$ is larger than one. These examples illustrate the invariance of the RGA parameters under increasing $m$ and that this property can hold even for functions that are not smooth. 
The rest of the examples are for the case $m = 1.$

\subsubsection{Multi-Instance Learning/Max Pooling under the Gaussian Marginal}\label{sec:mil}

Our first example is MIL (\Cref{prob:mil}), with the marginal distribution of $\vz$ being the standard Gaussian.

\begin{theorem}[MIL under Gaussian is RGA]\label{lem:MIL-is-RGA}
    Let $\mZ\in\R^{m\times d}\sim\D$, $\mZ^\top = [\vz_1,\dots,\vz_m]$ where rows $\vz_1, \dots, \vz_m \stackrel{i.i.d.}{\sim}\calN(0,\mI)$. Let $y = \sign(\max_i\vx^*\cdot\vz_i + t)$, where $t\in\R$ is a given parameter. Let $\cH_+$ be the class of functions $h:\R\to\R$ that are locally Lipschitz and such that $h'(z) + h'(-z)>0$ for a.e.\ $z>0$ and for every $a \in \R,$ $\E_{z \sim \cN(0, 1)}[|h'(a + z)| + |h'(-a - z)|] < \infty$. Then for any $h\in\cH_+$, the loss $f_h(\vx) = \E_{\mZ\sim\D}[h(-y(\max_{i \in [m]}\vx\cdot\vz_i + t))]$ is $1$-RGA for every $\theta(\vx,\vx^*)\in(0,\pi/2)$.
\end{theorem}

\begin{proof}
Let us denote $\ell_h(\vu,\vv) = h(-\sign(\max_{i\in[m]}v_i + t)(\max_{i \in [m]} u_i + t))$ for any penalty function $h\in\cH_+$, $\vv,\vu\in\R^m$.
Our theorem applies to the Riemannian correlation vector field $\vg_{h}(\vx,\vx^*)$ (\Cref{def:grad-field}); therefore, the first task is to prove that $\nablar f(\vx) = \vg_h(\vx,\vx^*)$.
By the definition of $\ell_h$, we have
\begin{align*}
    \partial^{C}_1\ell_h(\vu,\vv) \subseteq \partial^C h(-y(\max_{i\in[m]} u_i + t))(-y\mathrm{Conv}\{\ve_i: i \in\argmax_{j\in [m]} u_j\}). 
\end{align*}
Fix a measurable selection of the maximizing index (e.g., the smallest maximizing index) and a Borel measurable selection from $\partial^C h.$ This gives a measurable sample subgradient. Moreover, denoting the maximizing index of $u_i = \vx\cdot \vz_i$ by $i^*$, its norm is bounded by $\max_{i \in [m]}\norm{\vz_i}_2 (|h'(\vx\cdot \vz_{i^*} + t)| + |h'(-(\vx\cdot \vz_{i^*} + t))|)$. The expectation of this quantity is finite by Gaussian integrability assumption on $h'$, and the same bound implies that $f_h$ is locally Lipcshitz. Hence, \Cref{fact:Aumann-expectation} applies, and we have 
\begin{align*}
    \partial^{C,R}f(\vx) &\subseteq \E_{\mZ\sim\D}[\partial^{C,R}_\vx\ell(\mZ\vx,\mZ\vx^*; \vx, \vx^*)]\\
    &\subseteq\E_{\mZ\sim\D}[\partial^C h(-y(\max_{i\in[m]}\vx\cdot\vz_i + t))(-y\mP_{\perp\vx}\mZ^\top\mathrm{Conv}\{\ve_i: i \in\argmax_{j\in [m]} \vx\cdot\vz_j\})],
\end{align*}
where the expectation is the Aumann expectation.
Furthermore, since $h$ is locally Lipschitz, it is differentiable everywhere except on a set with Gaussian measure zero. 
In addition, since $\vx^*\cdot\vz_i,\vx\cdot\vz_i$ are standard Gaussian random variables for all $i\in[m]$ with continuous density functions, the measure of the set of  arguments $-y\max_i \vx\cdot\vz_i + t$ at which $h$ is non-differentiable is zero, and so is the measure of non-singleton sets $\argmax_{i\in[m]}\vx\cdot\vz_i$. 
Therefore, denoting the max-index indicator $i^*(\vu) = \argmax_{j\in[m]} u_j$ (the max-index set is a singleton set almost surely), the Aumann expectation of $\partial^{C,R}_\vx\ell(\mZ\vx,\mZ\vx^*)$ is simply:
\begin{align*}
    \E_{\mZ\sim\D}[\partial_\vx^{C,R}\ell(\mZ\vx,\mZ\vx^*; \vx, \vx^*)] &= -\E_{\mZ\sim\D}[\mP_{\perp\vx}\mZ^\top y\ve_{i^*(\mZ\vx)} h'(-y(\ve_{i^*(\mZ\vx)}^\top\mZ\vx + t))]\\
    &= \E_{\mZ\sim\D}[\mP_{\perp\vx}\mZ^\top\partial_1\ell_h(\mZ\vx,\mZ\vx^*; \vx, \vx^*)] = \vg_h(\vx,\vx^*).
\end{align*}
Since the Aumann expectation is a singleton set, we  conclude that $\nablar f(\vx) = \vg_h(\vx,\vx^*)$.

    According to \Cref{lem:regression-loss-satisfy-sharpness-Gaussian}, we only need to prove that $\E_{\mZ\sim\D}[\tr(D_2\partial_1\ell(\mZ\vx,\mZ\vx^*; \vx, \vx^*))]<0$. Note that by the definition of the max-index indicator $i^*(\vu)$, the label equals $y = \sign(\ve_{i^*(\mZ\vx^*)}^\top\mZ\vx^* + t) = \sign(\vx^*\cdot\vz_{i^*(\mZ\vx^*)} + t)$ almost surely, and hence the partial derivative $\partial_{(\mZ\vx)_i}\ell(\mZ\vx,\mZ\vx^*; \vx, \vx^*)$ equals (almost surely):
    \begin{align*}
       &\quad \frac{\partial}{\partial{u_i}}\ell_h(\vu,\vv; \vx, \vx^*)|_{\vu = \mZ\vx,\vv = \mZ\vx^*}\\
       &= -\sign(\ve_{i^*(\vv)}^\top\vv + t) \1\{i = i^*(\vu)\}h'(-\sign(\ve_{i^*(\vv)}^\top\vv + t)(\ve_{i^*(\vu)}^\top\vu + t))|_{\vu = \mZ\vx,\vv = \mZ\vx^*}\\
        &=\begin{cases}
            -\1\{i = i^*(\vu)\}h'(-(\ve_{i^*(\vu)}^\top\vu + t))|_{\vu = \mZ\vx} & \vx^*\cdot\vz_{i^*(\mZ\vx^*)} + t\geq 0\\
            \1\{i = i^*(\vu)\}h'(\ve_{i^*(\vu)}^\top\vu + t)|_{\vu = \mZ\vx} & \vx^*\cdot\vz_{i^*(\mZ\vx^*)} + t< 0.
        \end{cases}
    \end{align*}
    Therefore, when differentiating with respect to $v_i$, the joint differential is zero except at the decision boundary $\vx^*\cdot\vz_{i^*(\mZ\vx^*)} + t = 0$, at which place the function $\partial_{u_i}\ell_h(\mZ\vx,\mZ\vx^*; \vx, \vx^*)$ jumps from $\1\{i = i^*(\mZ\vx)\}h'(\vx\cdot\vz_i + t)$ to $-\1\{i = i^*(\mZ\vx)\}h'(-(\vx\cdot\vz_i + t))$. Therefore, by the definition of distributional derivatives and \Cref{fact:distributional-derivatives}, we have:
    \begin{align*}
        &\frac{D}{D {v_i}}\frac{\partial}{\partial{u_i}}\ell_h(\vu,\vv; \vx, \vx^*)|_{\vu = \mZ\vx,\vv = \mZ\vx^*} \\ 
        =& -\1\{i = i^*(\mZ\vx) = i^*(\mZ\vx^*)\}(h'(\vx\cdot\vz_i + t) + h'(-(\vx\cdot\vz_i + t)))\delta(\vx^*\cdot\vz_i + t),
    \end{align*}
    where $\delta$ denotes the Dirac delta function. 
    Thus, denoting $z_{i, \vu}\eqdef \vu\cdot\vz_i$, $i\in[m]$, the expectation of this joint differential is:
    \begin{align*}
        &\quad \E_{\mZ\sim\D}[(D_2\partial_1\ell(\mZ\vx,\mZ\vx^*; \vx, \vx^*))_{ii}]\\
        &=-\E_{\mZ\sim\D}[(h'(z_{i, \vx} + t) + h'(-(z_{i, \vx} + t))\delta(z_{i, \vx^*} + t)\1\{z_{j, \vx}\leq z_{i, \vx}, z_{j, \vx^*}\leq -t,\forall j\neq i\}]\\
        &=-\E_{\mZ\sim\D}\Big[(h'(z_{i, \vx} + t) + h'(-(z_{i, \vx} + t))\delta(z_{i, \vx^*} + t)\E\big[\1\{z_{j, \vx}\leq z_{i, \vx}, z_{j, \vx^*}\leq -t,\forall j\neq i\} | z_{i, \vx}, z_{i, \vx^*}\big]\Big].
    \end{align*}
    Noting that $\vz_{i,\vu},\, \vz_{j, \vu}$ are independent standard Gaussians for any fixed unit vector $\vu$, $i\neq j$, and $\E[\vz_{i, \vx}\vz_{i, \vx^*}] = \cos\theta\in(0,1)$, we have 
    \begin{align*}
        \Phi(a,t) &= \pr[z_{j, \vx}\leq a, z_{j, \vx^*}\leq -t] \\
        &= \int_{-\infty}^a\int_{-\infty}^{-t} \frac{1}{2\pi\sin\theta}\exp(-(z_1^2 - 2z_1z_2\cos\theta  + z_2^2)/(2\sin^2\theta))\diff{z_1}\diff{z_2}>0, \;\forall a, t\in\R.
    \end{align*}
    Then, letting $z\sim\calN(0,1)$ be an independent Gaussian, we have that $z_{i, \vx} = z_{i, \vx^*} \cos\theta  + z \sin\theta$ in distribution, and it holds:
    \begin{align*}
        &\quad \E_{\mZ\sim\D}[(D_2\partial_1\ell(\mZ\vx,\mZ\vx^*; \vx, \vx^*))_{ii}]\\
        &=-\iint_{z_{i, \vx},z_{i, \vx^*}}(h'(z_{i, \vx^*}\cos\theta  + z \sin\theta  + t) + h'(-(z_{i, \vx^*} \cos\theta  + z \sin\theta  + t))\delta(z_{i, \vx^*} + t)\times\\
        &\quad\quad\times\Phi^{m-1}(z_{i, \vx^*}\cos\theta  + z\sin\theta ,t)p_\calN(z_{i, \vx^*})p_\calN(z)\diff{z_{i, \vx^*}}\diff{z}\\
        &=-\int_z (h'((1-\cos\theta) t + z \sin\theta ) + h'(-((1-\cos\theta) t + z \sin\theta ))\Phi^{m-1}(-t\cos\theta  + z\sin\theta ,t)p_\calN(-t)p_\calN(z)\diff{z}.
    \end{align*}
    Since by our assumption $h'(z) + h'(-z)>0$ a.e., and since $\Phi^{m-1}(-t\cos\theta  + z\sin\theta, t)>0$, $p_\calN(-t)>0$ and $p_\calN(z)>0$, we conclude that $\E_{\mZ\sim\D}[(D_2\partial_1\ell(\mZ\vx,\mZ\vx^*; \vx, \vx^*))_{ii}]<0$ and therefore \Cref{assum:ell-Gaussian} is satisfied. Using \Cref{lem:regression-loss-satisfy-sharpness-Gaussian} we immediately conclude that $f_h(\vx)$ is $1$-RGA.
\end{proof}

\Cref{lem:MIL-is-RGA} proves that the MIL loss $f_h(\vx) = \E_{\mZ\sim\D}[h(-y(\max_{i\in[m]}\vx\cdot\vz_i + t))]$ is $1$-RGA for any $h\in\cH_+$ and any $\vx$ on the hemisphere $\theta(\vx,\vx^*)\in(0,\pi/2)$. The function class $\cH_+$ includes a broad range of penalty functions $h$; for example, ReLU, LeakyReLU, and the logistic loss,  which are frequently used in SVM applications \cite{andrews2002SVMforMIL,carbonneau2018multiple}, are all contained in $\cH_+$.

\subsubsection{Learning Gaussian Halfspaces}

When $m=1$, the MIL problem reduces to the classical  problem of learning halfspaces (\Cref{prob:halfspace}). Therefore, \Cref{lem:MIL-is-RGA} immediately characterizes the benign landscape of learning general Gaussian halfspaces under a variety of loss/penalty functions, as discussed above for \Cref{lem:MIL-is-RGA}. 

\begin{corollary}
    Let $\vz\sim\D$ where $\D = \calN(0,\mI_d)$ is the standard Gaussian distribution. Let $y\mid \vz = \sign(\vx^*\cdot\vz + t)$, where $t\in\R$ is a given real number.
Let $\cH_+$ be the class of functions $h:\R\to\R$ that are locally Lipschitz and such that $h'(z) + h'(-z)>0$ for a.e.\ $z>0$ and for every $a \in \R,$ $\E_{z \sim \cN(0, 1)}[|h(a + z)| + |h'(a + z)|] < \infty$. Then, for any $h\in\cH_+$, the loss $f_h(\vx) = \Evz[h(-y(\vx\cdot\vz + t))]$ is $1$-RGA at any $\vx$ such that $\theta(\vx,\vx^*)\in(0,\pi/2)$.
\end{corollary}
\begin{proof}
    Setting $m=1$ in \Cref{lem:MIL-is-RGA} yields the result.
\end{proof}

\subsubsection{GLMs and SIMs with General Activations}\label{subsec:glm-sim-Gaussian}

Our next example is the SIM regression problem defined in \Cref{prob:glm+sim}, with $\D$ being the standard Gaussian distribution. We consider two most commonly used loss/penalty functions: the $L_2^2$ loss and the $L_1$ loss, applied to broad link function classes. Since SIM is a more general model of regression problems, the results also apply to GLMs, by specialization.
In addition, we also consider the phase retrieval problems with $L_1$ loss/penalty, which is formulated as a specific SIM example.  

\begin{theorem}\label{thm:Gaussian-glm-sim-RGA}
    Let $\vz\sim\D$ where $\D = \calN(0,\mI_d)$ is standard Gaussian. 
    Let $\cF$ be the class of all non-constant square-integrable functions under the Gaussian measure.
In addition, let $\cF_+\subset\cF$ be the subset of $\cF$ that contains all square-integrable (under the Gaussian measure), strictly increasing, and locally Lipschitz functions. Fix any $\vx^*\in\spd$.
    \begin{enumerate}
        \item Consider the $L_2^2$ loss of Gaussian SIMs on $\cF$: $f_\Ltwo(\vx)\eqdef \min_{\sigma\in \cF}\Evz[(\sigma(\vx\cdot\vz) - \sigma^*(\vx^*\cdot\vz))^2]$ where $\sigma^*$ is any function in $\cF$. Then $f_\Ltwo(\vx)$ is 1-RGA for any $\theta(\vx,\vx^*)\in(0,\pi/2)$.
        \item Consider the $L_1$ loss of Gaussian SIMs on $\cF_+$: $f_\Lone(\vx)\eqdef \min_{\sigma\in \cF_+}\Evz[|\sigma(\vx\cdot\vz) - \sigma^*(\vx^*\cdot\vz)|]$ where $\sigma^*$ is any function in $\cF_+$. Then $f_\Lone(\vx)$ is $1$-RGA for any $\theta(\vx,\vx^*)\in(0,\pi/2)$.
        \item Consider the $L_1$ loss of Gaussian Phase Retrieval: $f_{\Lone,\mathsf{PR}}(\vx) = \min_{r\in\cF_+} \Evz[|r(|\vx\cdot\vz|) - r^*(|\vx^*\cdot\vz|)|]$ where $r^*$ is any function in $\cF_+$. Then $f_{\Lone,\mathsf{PR}}(\vx)$ is $1$-RGA for any $\theta(\vx,\vx^*)\in(0,\pi/2)$.
    \end{enumerate}    
\end{theorem}
Before proving the theorem, we state the following fact about Gaussian smoothing, which is repeatedly used in our proof. For completeness, its proof is provided in \Cref{appx:omitted-facts}.
\begin{restatable}[Gaussian Smoothing]{fact}{factgaussiansmoothing}\label{fct:Gaussian-smoothing}
Let $\D = \calN(0,1)$. For $\sigma\in L_2(\D)$ and $\rho\in(0,1)$, define 
\begin{equation}\label{eq:Trho-def}
    \Tr_\rho(\sigma)(u) := \Ez[\sigma(\rho u + \sqrt{1 - \rho^2}z)].
\end{equation} 
Then for any $u\in\R$ and $\rho\in(0,1)$, $(\rho,u)\mapsto\Tr_\rho\sigma(u)$ is infinite-times  continuously differentiable. Furthermore, for any positive integers $a,b$ and any $\rho\in(0,1)$,  we have that the function $u \mapsto \partial_\rho^a\partial_u^b \Tr_\rho\sigma$ is in $L_2(\D)$.

In addition, given $f(\vx) = \Evz[F(\vx, \vz; \vx^*)]$, $F(\vx, \vz; \vx^*)\eqdef(\Tr_{\vx\cdot\vx^*}\sigma(\vx\cdot\vz) - \sigma(\vx^*\cdot\vz))^2$, we have $\nablar f(\vx) = \Evz[\nablar F(\vx, \vz; \vx^*)]$, for any $\sigma\in L_2(\D)$.
\end{restatable}
We will also make use of the following fact.
\begin{fact}[Theorem 6.1.5\cite{Hormander2003}]\label{fact:dirac-inetgral}
    Given an integrable function $f(z_1,z_2)$ and a continuously differentiable $h(z_1,z_2)$  whose gradient $\nabla_{z_1,z_2} h$ is not zero when $h = 0$, we have:
    \begin{align*}
        \iint \delta(h(z_1,z_2))f(z_1,z_2)\diff{z_1,z_2} = \int_{h = 0} \frac{f(z_1,z_2)}{\|\nabla h(z_1,z_2)\|_2}\diff{\Sigma},
    \end{align*}
    where $\diff{\Sigma}$ is the Euclidean measure on the surface $h(z_1,z_2) = 0$.
\end{fact}
\begin{proof}[Proof of \Cref{thm:Gaussian-glm-sim-RGA}]
    Let us denote $h_\Ltwo(u) = u^2$ and $h_\Lone(u) = |u|$, and define: \begin{gather*}
        \sigma_\Ltwo(\cdot;\vx,\vx^*) \in \argmin_{\sigma\in\cF}\Evz[h_\Ltwo(\sigma(\vx\cdot\vz) - \sigma^*(\vx^*\cdot\vz))],\, \ell_\Ltwo(u,v;\vx,\vx^*) = h_\Ltwo(\sigma_\Ltwo(u;\vx,\vx^*) - \sigma^*(v));\\
        \sigma_\Lone(\cdot;\vx,\vx^*) \in \argmin_{\sigma\in\cF_+}\Evz[h_\Lone(\sigma(\vx\cdot\vz) - \sigma^*(\vx^*\cdot\vz))],\, \ell_\Lone(u,v;\vx,\vx^*) = h_\Lone(\sigma_\Lone(u;\vx,\vx^*) - \sigma^*(v));\\
        r(\cdot;\vx,\vx^*)\in\argmin_{r\in\cF_+}\Evz[h_\Lone(r(|\vx\cdot\vz|) - r^*(|\vx^*\cdot\vz|))],\, \sigma(u;\vx,\vx^*) = r(|u|;\vx,\vx^*),\,\sigma^*(u) = r^*(|u|).
    \end{gather*}
    Throughout the proof, we denote the angle between $\vx$ and $\vx^*$ by $\theta$ and note that $\theta\in (0, \pi/2)$ by assumption.
    We discuss the $L_2^2$ loss and the $L_1$ loss of SIMs separately.
    
    \textbf{(I): $L_2^2$ SIM Loss. } Let us investigate the $L_2^2$ loss first. We first characterize the best-fitting link function. \begin{claim}\label{claim:Gaussian-l2-bestfitting}
Let $\sigma^*\in\cF$ and let $\vx, \vx^* \in \spd$ be such that $\theta:= \theta(\vx, \vx^*) \in (0, \pi/2).$ Then the best-fitting activation under the $L_2^2$ loss, up to a.e.\ equality under the standard Gaussian measure, is $\sigma_\Ltwo(u;\vx,\vx^*) = \mathrm{T}_{\cos\theta}\sigma^*(u)$. 
   Furthermore, 
        \begin{align*}
            \nablar f_\Ltwo(\vx) = \Evz[\partial_1\ell_{\Ltwo}(\vx\cdot\vz,\vx^*\cdot\vz;\vx,\vx^*)\vz_{\perp\vx}] = \vg_\Ltwo(\vx,\vx^*).
        \end{align*}
    \end{claim}
    \begin{proof}
Given $\vx\in\spd$, by the standard conditional expectation identity 
        \[
        \E[(\sigma(\vx \cdot \vz) - \sigma^*(\vx^*\cdot \vz))^2] = \E[(\sigma(\vx \cdot \vz) - \E[\sigma^*(\vx^*\cdot \vz)|\vx\cdot\vz])^2] + \E[\Var(\sigma^*(\vx^*\cdot \vz)|\vx\cdot\vz)]. 
        \]
        Thus, the minimizing activation among all Gaussian square-integrable functions is $\sigma_\Ltwo(u;\vx,\vx^*) = \Evz[\sigma^*(\vx^*\cdot\vz)\mid\vx\cdot\vz = u]$.
        Since $\vz$ is standard Gaussian, the random variables $\vx\cdot\vz$ and $\vv\cdot\vz$ are independent standard Gaussians provided that $\|\vv\|_2 = 1$ and $\vv\perp\vx$. Now let $\vv := \vx^*_{\perp\vx}/\|\vx^*_{\perp\vx}\|_2$. Then $\sigma_\Ltwo$ equals:
        \begin{align*}
            \sigma_\Ltwo(u;\vx,\vx^*) = \Evz[\sigma^*(\vx\cdot\vz\cos\theta + \vv\cdot\vz\sin\theta)\mid\vx\cdot\vz = u] = \E_{\vv\cdot\vz\sim\cN(0,1)}[\sigma^*(u\cos\theta  + \vv\cdot\vz\sin\theta)] = \mathrm{T}_{\cos\theta}\sigma^*(u).
        \end{align*}
        This proves the first claim. Furthermore, by \Cref{fct:Gaussian-smoothing},  $(\rho,u)\mapsto\mathrm{T}_{\rho}\sigma^*(u)$ is infinite-times differentiable in both arguments $(\rho,u)$, for any $\rho\in(0,1)$ and $u\in\R$. 
        Moreover, by the same fact, $\nablar f(\vx) = \Evz[\nablar_\vx \ell(\vx\cdot\vz,\vx^*\cdot\vz;\vx,\vx^*)]$, where
\begin{align*}
          &  \nablar_\vx \ell_\Ltwo(\vx\cdot\vz,\vx^*\cdot\vz;\vx,\vx^*)\\
          =& \partial_1\ell_\Ltwo(\vx\cdot\vz,\vx^*\cdot\vz;\vx,\vx^*)\vz_{\perp\vx} + (\partial_3\ell_\Ltwo(\vx\cdot\vz,\vx^*\cdot\vz;\vx,\vx^*))_{\perp\vx}\\
            =& 2(\mathrm{T}_{\cos\theta}\sigma^*(\vx\cdot\vz) - \sigma^*(\vx^*\cdot\vz))\bigg((\mathrm{T}_{\cos\theta}\sigma^*)'(\vx\cdot\vz)\vz_{\perp\vx} + \Big(\frac{\diff{}}{\diff{\rho}}\mathrm{T}_{\rho}\sigma^*(\vx\cdot\vz)\Big)\Big|_{\rho = \vx\cdot \vx^*}\vx^*_{\perp\vx}\bigg),
        \end{align*}
        where we have used the chain rule, recalling that $\sigma_\Ltwo(u;\vx,\vx^*) = \mathrm{T}_{\cos\theta}\sigma^*(u)$ and $\cos\theta = \vx \cdot \vx^*,$ as discussed earlier in the proof.   
        
        Now, taking the expectation with respect to the term $\partial_3\ell$ and noting that by the definition of operator $\mathrm{T}_\rho$ it holds $\E[(\mathrm{T}_{\cos\theta}\sigma^*(\vx\cdot\vz) - \sigma^*(\vx^*\cdot\vz))\mid \vx\cdot\vz] = 0$, we have:
        \begin{align*}
            \Evz[(\partial_3\ell_\Ltwo(\vx\cdot\vz,\vx^*\cdot\vz;\vx,\vx^*))_{\perp\vx}]&=\Evz\bigg[2(\mathrm{T}_{\cos\theta}\sigma^*(\vx\cdot\vz) - \sigma^*(\vx^*\cdot\vz))\bigg(\frac{\diff{}}{\diff{\rho}}\mathrm{T}_{\vx\cdot\vx^*}\sigma^*(\vx\cdot\vz)\bigg)\vx^*_{\perp\vx}\bigg]\\
            &=2\Evz\bigg[\E\bigg[(\mathrm{T}_{\cos\theta}\sigma^*(\vx\cdot\vz) - \sigma^*(\vx^*\cdot\vz))\mid \vx\cdot\vz\bigg]\bigg(\frac{\diff{}}{\diff{\rho}}\mathrm{T}_{\vx\cdot\vx^*}\sigma^*(\vx\cdot\vz)\bigg)\vx^*_{\perp\vx}\bigg]\\
            &= 0.
        \end{align*}
        Therefore,  
        \begin{align*}
            \nablar f_\Ltwo(\vx) &=2\Evz[(\mathrm{T}_{\cos\theta}\sigma^*(\vx\cdot\vz) - \sigma^*(\vx^*\cdot\vz))(\mathrm{T}_{\cos\theta}\sigma^*)'(\vx\cdot\vz)\vz_{\perp\vx}]\\
            &=\Evz[\partial_1\ell_\Ltwo(\vx\cdot\vz,\vx^*\cdot\vz;\vx,\vx^*)\vz_{\perp\vx}] = \vg_\Ltwo(\vx,\vx^*),
        \end{align*}
        proving the second claim.
    \end{proof}

    Now given \Cref{claim:Gaussian-l2-bestfitting}, it only remains to check the joint differential of $\ell_{\Ltwo}$. We have:
    \begin{align*}
        D_2\partial_{1}\ell(\vx\cdot\vz,\vx^*\cdot\vz;\vx,\vx^*) &= D_2\bigg(h'_{\Ltwo}(\sigma_\Ltwo(u;\vx,\vx^*) - \sigma^*(v))\sigma_\Ltwo'(u;\vx,\vx^*)\bigg)\bigg|_{u = \vx\cdot\vz, v = \vx^*\cdot\vz}\\
        &=-h_{\Ltwo}''(\sigma_\Ltwo(u;\vx,\vx^*) - \sigma^*(v))\sigma_\Ltwo'(u;\vx,\vx^*) D\sigma^*(v)\bigg|_{u = \vx\cdot\vz, v = \vx^*\cdot\vz}\\
        &=-2(\mathrm{T}_{\cos\theta}\sigma^*)'(\vx\cdot\vz)D\sigma^*(\vx^*\cdot\vz).
    \end{align*}
    Let $z_1= \vx\cdot\vz$ and $z_2 = \vv\cdot\vz$ where $\vv = \vx^*_{\perp\vx}/\|\vx^*_{\perp\vx}\|_2$, which are independent Gaussian random variables. 
    Therefore, by the definition of weak (distributional) derivatives, the expectation is defined as the distribution pairing, and we have:
    \begin{align*}
        \Evz[D_2\partial_{1}\ell(\vx\cdot\vz,\vx^*\cdot\vz;\vx,\vx^*)]&= \Evz[-2(\mathrm{T}_{\cos\theta}\sigma^*)'(\vx\cdot\vz)D\sigma^*(\vx^*\cdot\vz)]\\
        &=-2\la D\sigma^*(z_1 \cos\theta + z_2 \sin\theta), (\mathrm{T}_{\cos\theta}\sigma^*)'(z_1)p_{\calN}(z_1,z_2)\ra\\
        &=-2\int_{z_1}\bigg(\int_{z_2}D\sigma^*(z_1 \cos\theta + z_2 \sin\theta) p_\calN(z_2)\diff{z_2}\bigg)(\mathrm{T}_{\cos\theta}\sigma^*)'(z_1) p_\calN(z_1)\diff{z_1}.
    \end{align*}
    We show that given any fixed $z_1$, the integral in the parentheses above equals  $(1/\cos\theta)(\mathrm{T}_{\cos\theta}\sigma^*)'(z)$. To this aim, noting that by integration by parts (or the definition of distributional weak derivatives), and observing that for the Gaussian density it holds $p_\calN'(z) = -z p_\calN(z)$, we have:
    \begin{align*}
        \int_{z_2}D\sigma^*(z_1 \cos\theta + z_2 \sin\theta) p_\calN(z_2)\diff{z_2}&=\int_{y} D\sigma^*(y)p_{\calN}\bigg(\frac{y - z_1 \cos\theta}{\sin\theta}\bigg)\frac{1}{\sin\theta}\diff{y}\\
        &=-\int_y \sigma^*(y)\frac{\diff{}}{\diff{y}}p_{\calN}\bigg(\frac{y - z_1 \cos\theta}{\sin\theta}\bigg)\frac{\diff{y}}{\sin\theta}\\
        &=\int_y \sigma^*(y)\bigg(\frac{y - z_1 \cos\theta}{\sin\theta}\bigg)p_{\calN}\bigg(\frac{y - z_1 \cos\theta}{\sin\theta}\bigg)\frac{\diff{y}}{\sin^2\theta}.
    \end{align*}
    On the other hand, writing down explicitly the convolutional integral form of $\mathrm{T}_{\cos\theta}\sigma^*(z)$ we have:
    \begin{align*}
        \frac{\diff{}}{\diff{z}}\mathrm{T}_{\cos\theta}\sigma^*(z)&=\frac{\diff{}}{\diff{z}}\int_{z_2}\sigma^*(z \cos\theta + z_2 \sin\theta)p_\calN(z_2)\diff{z_2} =\frac{\diff{}}{\diff{z}}\int_y \sigma^*(y)p_\calN\bigg(\frac{y - z \cos\theta}{\sin\theta}\bigg)\frac{\diff{y}}{\sin\theta}\\
        &=\int_y \sigma^*(y)\frac{\diff{}}{\diff{z}}p_\calN\bigg(\frac{y - z \cos\theta}{\sin\theta}\bigg)\frac{\diff{y}}{\sin\theta} = \cos\theta \int_y \sigma^*(y)\bigg(\frac{y - z_1 \cos\theta}{\sin\theta}\bigg)p_{\calN}\bigg(\frac{y - z_1 \cos\theta}{\sin\theta}\bigg)\frac{\diff{y}}{\sin^2\theta}.
    \end{align*}
    Thus, comparing with the equality above we have:
    \begin{align*}
        \int_{z_2}D\sigma^*(z_1 \cos\theta + z_2 \sin\theta) p_\calN(z_2)\diff{z_2} = \frac{1}{\cos\theta}(\mathrm{T}_{\cos\theta}\sigma^*)'(z_1).
    \end{align*}
    Having this, we immediately obtain the expectation of the joint differential:
    \begin{align*}
        \Evz[D_2\partial_{1}\ell(\vx\cdot\vz,\vx^*\cdot\vz;\vx,\vx^*)] &= -2\int_{z_1}\frac{1}{\cos\theta}((\mathrm{T}_{\cos\theta}\sigma^*)'(z_1))^2 p_\calN(z_1)\diff{z_1}\\
        &=-\frac{2}{\cos\theta}\|(\mathrm{T}_{\cos\theta}\sigma^*)'(z_1)\|_{L_2(\calN)}^2<0.
    \end{align*}
    The last inequality is due to the fact that since $\sigma^*$ is not a constant function, $\mathrm{T}_{\cos\theta}\sigma^*$ is a smooth non-constant function and hence $\|(\mathrm{T}_{\cos\theta}\sigma^*)'(z_1)\|_{L_2(\calN)}^2>0$.
    Thus, \Cref{assum:ell-Gaussian} is satisfied for all $\vx$ such that $\theta(\vx,\vx^*)\in(0,\pi/2)$. Applying \Cref{lem:regression-loss-satisfy-sharpness-Gaussian} we obtain that $f_\Ltwo$ is 1-RGA.

    \textbf{(II): $L_1$ SIM Loss for Monotone Link Functions. } Similar to the $L_2^2$ loss argument, we first study and characterize the best-fitting link functions under the $L_1$ loss.
    \begin{claim}\label{claim:Gaussian-l1-bestfitting}
         Let $\sigma^*\in\cF_+$ be a strictly increasing function and $\vx^*$ be a fixed unit vector. Then, for any unit vector $\vx$ with $\theta = \theta(\vx,\vx^*)\in(0,\pi/2)$, it holds $\sigma(z;\vx,\vx^*) = \sigma^*(z\cos\theta)$. Furthermore,
    \begin{align*}
         \nablar f_\Lone(\vx) = \Evz[\partial_1\ell_\Lone(\vx\cdot\vz,\vx^*\cdot\vz;\vx,\vx^*)\vz_{\perp\vx}].
    \end{align*}
    \end{claim}
    \begin{proof}
        The best-fitting link function under the $L_1$ loss is the conditional median: $\sigma(u;\vx,\vx^*) = \mathrm{median}[\sigma^*(\vx^*\cdot\vz)\mid \vx\cdot\vz = u]$; in other words, $y = \sigma(u;\vx,\vx^*)$ if $\pr[\sigma^*(\vx\cdot\vz^*) \leq y\mid \vx\cdot\vz = u] = 1/2$. This holds because for any real-valued random variable $X$ and any $a \in \R,$ the minimizers of $a \mapsto \E[|a - X|]$ are precisely the medians of $X.$  Let $\vv = \vx^*_{\perp\vx}/\norm{\vx^*_{\perp\vx}}_2,$  $z_1 = \vx\cdot\vz$ and $z_2 = \vv\cdot\vz$. By Gaussianity, $z_1,z_2$ are independent standard Gaussians, and by the (strict) monotonicity, we have:
    \begin{align*}
        \pr[\sigma^*(\vx^*\cdot\vz) \leq y\mid \vx\cdot\vz = u] = \pr_{z_2}\bigg[u\cos\theta  + z_2\sin\theta \leq (\sigma^*)^{-1}(y)\bigg] = \pr\bigg[z_2\leq \frac{(\sigma^*)^{-1}(y) - u\cos\theta}{\sin\theta}\bigg].
    \end{align*}
    Since $z_2$ is a standard Gaussian random variable, it holds that:
    \begin{align*}
        \pr\bigg[z_2\leq \frac{(\sigma^*)^{-1}(y) - u\cos\theta }{\sin\theta}\bigg] = 1/2 \Leftrightarrow (\sigma^*)^{-1}(y) - u\cos\theta  = 0 \Leftrightarrow y = \sigma^*(u\cos\theta).
    \end{align*}
    Thus, we conclude that $\sigma(u;\vx,\vx^*) = \sigma^*(u\cos\theta) = \sigma^*(u(\vx\cdot\vx^*))$. It is immediate that $\sigma(u;\vx,\vx^*)\in\cF_+$.

    Next, we argue that $\nablar f(\vx)$ is a vector field of the form $\vg_\Lone(\vx,\vx^*)$. Similar to the proof of \Cref{claim:Gaussian-l2-bestfitting}, we first show that the differentiation and expectation are interchangeable, i.e., $\nablar f_\Lone(\vx) = \Evz[\nablar_\vx\ell_\Lone(\vx\cdot\vz;\vx^*\cdot\vz;\vx,\vx^*)]$. Since $\sigma^*$ is locally Lipschitz, we have that $\sigma^*$ is differentiable almost everywhere. By the definition of $\ell_\Lone(u,v;\vx,\vx^*)$, we have:
    \begin{align*}
        \nablar_\vx\ell_\Lone(\vx\cdot\vz,\vx^*\cdot\vz;\vx,\vx^*) &=\partial_1\ell_\Lone(u,v;\vx,\vx^*)|_{u = \vx\cdot\vz,v = \vx^*\cdot\vz}(\nabla_\vx(\vx\cdot\vz))_{\perp\vx} + (\partial_3\ell_\Lone(u,v;\vx,\vx^*)|_{u = \vx\cdot\vz,v = \vx^*\cdot\vz})_{\perp\vx}\\
        &= \sign(\sigma^*(\vx\cdot\vz\cos\theta) - \sigma^*(\vx^*\cdot\vz))(\sigma^*)'(\vx\cdot\vz\cos\theta)(\cos\theta\vz_{\perp\vx} + (\vx\cdot\vz)\vx^*_{\perp\vx}).
    \end{align*}
Since $\sigma^*$ is locally Lipschitz, strictly increasing, and square-integrable, it follows that for every compact $K \subset (0, 1),$ $\sup_{\rho \in K}\E[(1 + |z|)|(\sigma^*)'(\rho z)|] < \infty$ (by Gaussian integration by parts). 
    Therefore, the differentiation operator and expectation are interchangeable, and we have $\nablar f(\vx) =\Evz[\nablar_\vx\ell_\Lone(\vx\cdot\vz,\vx^*\cdot\vz;\vx,\vx^*)]$. Finally, we show that $\Evz[(\partial_3\ell_\Lone(\vx\cdot\vz,\vx^*\cdot\vz;\vx,\vx^*))_{\perp\vx}] = 0$, leading to $\nablar f_\Lone(\vx) = \vg_\Lone(\vx,\vx^*)$. To this end, observe that:
    \begin{align*}
        &\quad\Evz[(\partial_3\ell_\Lone(u,v;\vx,\vx^*)|_{u = \vx\cdot\vz,v = \vx^*\cdot\vz})_{\perp\vx}]\\
        &=\Evz[\sign(\sigma^*(\vx\cdot\vz\cos\theta) - \sigma^*(\vx^*\cdot\vz))(\sigma^*)'(\vx\cdot\vz\cos\theta)(\vx\cdot\vz)\vx^*_{\perp\vx}]\\
        &=\Evz[\E[\sign(\sigma^*(\vx\cdot\vz\cos\theta) - \sigma^*(\vx^*\cdot\vz))|\vx\cdot\vz](\sigma^*)'(\vx\cdot\vz\cos\theta)(\vx\cdot\vz)]\vx^*_{\perp\vx}.
    \end{align*}
    Since by definition, $\sigma(u;\vx,\vx^*) = \sigma^*(u\cos\theta) = \mathrm{median}(\sigma^*(\vx^*\cdot\vz)\mid\vx\cdot\vz = u)$, it holds that 
    $$\pr[\sigma^*(u\cos\theta)\geq \sigma^*(\vx^*\cdot\vz)\mid \vx\cdot\vz = u] = \pr[\sigma^*(u\cos\theta )\leq \sigma^*(\vx^*\cdot\vz)\mid \vx\cdot\vz = u] = 1/2,$$ 
    and therefore, $\E[\sign(\sigma^*(\vx\cdot\vz\cos\theta) - \sigma^*(\vx^*\cdot\vz))|\vx\cdot\vz] = 1/2 - 1/2 = 0$. Thus, we conclude that $\nablar f_\Lone(\vx) =\Evz[\partial_1\ell_\Lone(\vx\cdot\vz,\vx^*\cdot\vz;\vx,\vx^*)\vz_{\perp\vx}] =  \vg_\Lone(\vx,\vx^*)$.
    \end{proof}
     It remains to argue that the joint curvature is negative. By \Cref{claim:Gaussian-l1-bestfitting}, the best-fitting activation is $\sigma_\Lone(u; \vx, \vx^*) = \sigma^*(u \cos\theta)$. Since $\sigma^*$ is strictly increasing, $\sign(\sigma^*(u \cos\theta) - \sigma^*(v)) = \sign(u \cos\theta - v).$ It follows that for almost every $u$,
    \[
    \partial_1 \ell(u, v; \vx, \vx^*) = \cos\theta(\sigma^*)'(u\cos\theta)\sign(u \cos\theta - v),
    \]
    and, thus, in the distributional sense,
    \[
    D_2\partial_1 \ell(u, v; \vx, \vx^*) = - 2 \cos\theta(\sigma^*)'(u\cos\theta)\delta(v - u \cos\theta).
    \]
    Observe that, since $\vz$ is standard Gaussian, we have that the distribution of $\vx^*\cdot \vz$ conditional on $\vx\cdot\vz = u$ is $\cN(u\cos\theta, \sin^2\theta).$ As a consequence, $\E[\delta(v - u \cos\theta)|\vx\cdot\vz = u] = \frac{1}{\sin\theta \sqrt{2\pi}}.$ It thus follows that
    \[
    \E[D_2\partial_1 \ell(\vx\cdot\vz, \vx^*\cdot\vz; \vx, \vx^*)] = -\frac{2\cos\theta}{\sin\theta \sqrt{2\pi}}\E[(\sigma^*)'(\vx\cdot\vz\cos\theta)] <0,
    \]
    where the last inequality follows from the assumed strict monotonicity of $\sigma^*$, by which $(\sigma^*)'\geq 0$ a.e.\ and, together with local absolute continuity, $(\sigma^*)' > 0$ on a set of positive Lebesgue measure. The expectation is finite by the square integrability of $\sigma^*$. Thus, \Cref{assum:ell-Gaussian} applies. 
Applying \Cref{lem:regression-loss-satisfy-sharpness-Gaussian}, we immediately obtain that $f_\Lone$ is $1$-RGA.

    \textbf{(III): $L_1$ SIM Loss for Phase Retrieval. } For the phase retrieval problems, the best-fitting link function takes a more complicated form compared to the monotone SIMs:
    \begin{claim}\label{claim:Gaussian-l1-PR-bestfitting}
         Let $r^*\in\cF_+$ be a strictly increasing function and $\vx^*$ be a fixed unit vector. Let $r(u;\vx,\vx^*) = \argmin_{r\in\cF_+}\Evz[|r(|\vx\cdot\vz|) - r^*(|\vx^*\cdot\vz|))|]$. Then, for any unit vector $\vx$ with $\theta = \theta(\vx,\vx^*)\in(0,\pi/2)$, and any $u\in[0,\infty)$,
         \begin{align*}
             r(u;\vx,\vx^*) = y \;\Leftrightarrow\; \Phi\bigg(\frac{(r^*)^{-1}(y) - u \cos\theta}{\sin\theta}\bigg) - \Phi\bigg(\frac{-(r^*)^{-1}(y) - u \cos\theta}{\sin\theta}\bigg) = \frac{1}{2},
         \end{align*}
         where $\Phi(t) = \pr_{z \sim \cN(0, 1)}[z\leq t]$ denotes the Gaussian cumulative density function. 

         Furthermore,  $\nablar f_{\Lone, \mathsf{PR}} = \vg_{\Lone,\mathsf{PR}}(\vx,\vx^*)$.
\end{claim}
    \begin{proof}
        Recall that the best-fitting link function under $L_1$ loss is $r(u;\vx,\vx^*) = \mathrm{median}(r^*(|\vx^*\cdot\vz|)\mid |\vx\cdot\vz| = u)$. Note that given $|\vx\cdot\vz| = u$, in distribution it holds $|\vx^*\cdot\vz| = |\vx\cdot\vz\cos\theta  + \vv\cdot\vz\sin\theta| = |\cos\theta |\vx\cdot\vz| + \vv\cdot\vz\sin\theta| = |u\cos\theta + \vv\cdot\vz\sin\theta|$, as $\vv\cdot\vz$ is an independent Gaussian of $\vx\cdot\vz$. In other words, given $|\vx\cdot\vz| = u$, if $r(u;\vx,\vx^*) = y$ then by the definition of a median, we have:
        \begin{align*}
            \pr[r^*(|\vx^*\cdot\vz|)\geq y\mid |\vx\cdot\vz| = u] =\pr[r^*(|\vx\cdot\vz\cos\theta + \vv\cdot\vz\sin\theta|)\geq y\mid \vx\cdot\vz = u] = \frac{1}{2}.
        \end{align*}
        Since $r^*$ is strictly increasing, we have:
        \begin{align*}
            \pr[r^*(|\vx^*\cdot\vz|)\geq y\mid \vx\cdot\vz = u]&=\pr_z[|u \cos\theta + z \sin\theta|\geq (r^*)^{-1}(y)]\\
            &=\pr_z\bigg[z\geq \frac{(r^*)^{-1}(y) - u \cos\theta}{\sin\theta}\bigg] + \pr_z\bigg[z\leq \frac{-(r^*)^{-1}(y) - u\cos\theta}{\sin\theta}\bigg].
        \end{align*}
        Recalling that $\Phi(t) = \pr[z\leq t]$, we obtain:
        \begin{align}\label{eq:equation-r(uxx*)}
            r(u;\vx,\vx^*) = y \;\Leftrightarrow\; \Phi\bigg(\frac{(r^*)^{-1}(y) - u \cos\theta}{\sin\theta}\bigg) - \Phi\bigg(\frac{-(r^*)^{-1}(y) - u \cos\theta}{\sin\theta}\bigg) = \frac{1}{2}.
        \end{align}
        It remains to argue that $r(u;\vx,\vx^*)$ is locally Lipschitz and strictly increasing so that $r(u;\vx,\vx^*)\in\cF_+$. For $c \in (0, 1)$, define 
        $$F(c, u, m) := \Phi\bigg(\frac{m-cu}{\sqrt{1-c^2}}\bigg) - \Phi\bigg(\frac{-m-cu}{\sqrt{1-c^2}}\bigg)-\frac{1}{2}.$$
        Then, the conditional median $m_{\cos\theta}(u) = (r^*)^{-1}(r(u;\vx,\vx^*))$ is the unique solution to $F(c, u, m_c(u)) = 0$ when $c = \cos\theta.$ Let: 
        \begin{align*}
            a = \frac{m_c(u) - cu}{\sqrt{1-c^2}},\,b = \frac{-m_c(u) - cu}{\sqrt{1-c^2}}, \text{ and observe that }\partial_3 F(c, u, m_c(u)) = \frac{p_\cN(a) + p_\cN(b)}{\sqrt{1-c^2}} > 0.
        \end{align*}
Therefore, the implicit function theorem implies that $(c, u)\mapsto m_c(u)$ is continuously differentiable (hence, locally Lipschitz) on $(0, 1) \times [0, \infty)$. Moreover, by the implicit function theorem,
        \begin{align*}
            m_c'(u) = -\frac{\partial_2 F(c, u, m_c(u))}{\partial_3 F(c, u, m_c(u))} = c\frac{p_\cN(a) - p_\cN(b)}{p_\cN(a) + p_\cN(b)}.
        \end{align*}
For $u > 0,$ $|a| = \frac{|m_c(u) - c u|}{\sqrt{1-c^2}}< \frac{m_c(u) + c u}{\sqrt{1-c^2}} = |b|,$ so $p_\cN(a) - p_\cN(b) > 0.$ It follows that in this case $m_c'(u) \in (0, c),$ implying that $m_c(u)$ is strictly increasing and Lipschitz in $u.$ Since $r(u; \vx, \vx^*) = r^*(m_{\vx \cdot \vx^*}(u))$ and $r^*$ is strictly increasing and locally Lipschitz, so is $r$, in both $u$ and $\vx.$

Fix $\vx$ and let $c_0 = \vx \cdot \vx^* \in (0, 1).$ For $c$ in a compact neighborhood $K \subset (0, 1)$ of $c_0$, the implicit-function formulas above imply $0\leq \partial_u m_c(u) \leq c$ and $|\partial_c m_c(u)| \leq C_K(1 + u)$ for some positive constant $C_K.$ Since $r^*$ is increasing, locally Lipschitz, and square integrable, Gaussian integration by parts gives a Gaussian-integrable envelope for the derivatives of $r^*(m_c(|\vx \cdot \vz|)).$   
        This implies that we can safely interchange differentiation and expectation, and it holds that
        \begin{align*}
            \nablar f_{\Lone,\mathsf{PR}}(\vx) &= \Evz[\nablar_\vx\ell_{\Lone,\mathsf{PR}}(\vx\cdot\vz,\vx^*\cdot\vz;\vx,\vx^*)]\\
            &=\Evz[\partial_1\ell_{\Lone,\mathsf{PR}}(\vx\cdot\vz,\vx^*\cdot\vz;\vx,\vx^*)\vz_{\perp\vx}] + \Evz[(\partial_3\ell_{\Lone,\mathsf{PR}}(\vx\cdot\vz,\vx^*\cdot\vz;\vx,\vx^*))_{\perp\vx}].
        \end{align*}
        To prove that $\nablar f_{\Lone,\mathsf{PR}} = \vg_{\Lone,\mathsf{PR}}$, we observe that the expectation with respect to $\partial_3\ell_{\Lone,\mathsf{PR}}$ is zero:
        \begin{align*}
          &  \Evz[(\partial_3\ell_{\Lone,\mathsf{PR}}(\vx\cdot\vz,\vx^*\cdot\vz;\vx,\vx^*))_{\perp\vx}] \\
          =&\; \Evz[\sign(r(|\vx\cdot\vz|;\vx,\vx^*) - r^*(|\vx^*\cdot\vz|))(\partial_\vx r(u;\vx,\vx^*)|_{u = |\vx\cdot\vz|})_{\perp\vx}]\\
            =&\; \Evz[\E[\sign(r(|\vx\cdot\vz|;\vx,\vx^*) - r^*(|\vx^*\cdot\vz|))|\vx\cdot\vz](\partial_\vx r(u;\vx,\vx^*)|_{u = |\vx\cdot\vz|})_{\perp\vx}].
        \end{align*}
        By the definition of $r(|\vx\cdot\vz|;\vx,\vx^*)$, it holds that given $\vx\cdot\vz$, the probability that the sign of the difference between $r$ and $r^*$ being positive and negative is the same: $\pr[r(|\vx\cdot\vz|;\vx,\vx^*)\geq r^*(|\vx^*\cdot\vz|)|\vx\cdot\vz] = \pr[r(|\vx\cdot\vz|;\vx,\vx^*)\leq r^*(|\vx^*\cdot\vz|)|\vx\cdot\vz]  =1/2$, hence $\E[\sign(r(|\vx\cdot\vz|;\vx,\vx^*) - r^*(|\vx^*\cdot\vz|))|\vx\cdot\vz] = 0$. Therefore, $\Evz[(\partial_3\ell_{\Lone,\mathsf{PR}}(\vx\cdot\vz,\vx^*\cdot\vz;\vx,\vx^*))_{\perp\vx}] = 0$, and the gradient of $f_{\Lone, \mathsf{PR}}$ is exactly in the form of vector field $\vg$: $\nablar f_{\Lone, \mathsf{PR}}(\vx) = \vg_{\Lone, \mathsf{PR}}(\vx,\vx^*)$.
    \end{proof}

    Note here that while \Cref{claim:Gaussian-l1-PR-bestfitting} sets the domain of $r(u; \vx, \vx^*)$ to be the nonnegative real line, this is without loss of generality. The reason is that the objective in the phase retrieval problem only ever evaluates $r(|\vx \cdot \vz|),$ thus the function constructed on $[0, \infty)$ can be extended arbitrarily over the negative real line to obtain an element of $\cF_+.$ 

    Now given \Cref{claim:Gaussian-l1-PR-bestfitting}, it only remains to check that the joint curvature is negative. As before, let $c = \cos\theta$ and $u \geq 0,$ then  $m_{c}(u) = (r^*)^{-1}(r(u;\vx,\vx^*))$ and $r(u;\vx,\vx^*) = r^*(m_{c}(u))$. Since $r^*$ is strictly increasing, it holds that $\sign(r(|u|; \vx, \vx^*) - r^*(|v|)) = \sign(m_c(|u|)-|v|).$ Therefore, for almost every $u,$ we have
\[
\partial_1 \ell_{\Lone, \mathsf{PR}}(u, v; \vx, \vx^*) = r'(|u|; \vx, \vx^*)\sign(u)\sign(m_c(|u|) - |v|),
\]
and, distributionally,
\[
D_2\partial_1 \ell_{\Lone, \mathsf{PR}}(u, v; \vx, \vx^*) = - 2r'(|u|; \vx, \vx^*)\sign(u)\big(\delta(v - m_c(|u|)) - \delta(v + m_c(|u|))\big).
\]
Now let $u = \vx \cdot \vz$, $v = \vx^*\cdot\vz$. Since the distribution of $v$ given $u$ is $\cN(u \cos\theta, \sin^2\theta),$ we get
\begin{align*}
    &\quad \Evz[D_2\partial_1 \ell_{\Lone, \mathsf{PR}}(\vx \cdot \vz, \vx^*\cdot\vz; \vx, \vx^*)|\vx\cdot\vz = u] \\
    &= - 2r'(|u|; \vx, \vx^*)\sign(u)\big(p_{\cN(u \cos\theta, \sin^2\theta)}(m_c(|u|)) - p_{\cN(u \cos\theta, \sin^2\theta)}(- m_c(|u|))\big).
\end{align*}
For $u \neq 0,$ letting $m = m_c(|u|)>0,$ we have 
\begin{align*}
    \frac{p_{\cN(u \cos\theta, \sin^2\theta)}(m)}{p_{\cN(u \cos\theta, \sin^2\theta)}(-m)} = \exp\bigg(\frac{2m u \cos\theta}{\sin^2\theta}\bigg) \;\Rightarrow\; \sign(u)\big(p_{\cN(u \cos\theta, \sin^2\theta)}(m_c(|u|)) - p_{\cN(u \cos\theta, \sin^2\theta)}(- m_c(|u|))\big) > 0.
\end{align*}
Finally, because $r$ is strictly increasing, we have that $r'(|u|; \vx, \vx^*) \geq 0$ a.e., with strict inequality on a set with positive measure. Thus, we can conclude that for a.e. $u$, \[
    \Evz[D_2\partial_1 \ell_{\Lone, \mathsf{PR}}(\vx \cdot \vz, \vx^*\cdot\vz; \vx, \vx^*)|\vx\cdot\vz = u] < 0,
\]
and so \Cref{assum:ell-Gaussian} holds. 
 It remains to apply \Cref{lem:regression-loss-satisfy-sharpness-Gaussian} to conclude that $f_{\Lone, \mathsf{PR}}(\vx)$ is $1$-RGA.
\end{proof}

\subsection{Example Regression Problems with General-Covariance Gaussians}\label{sec:ex:general Gaussian}

{We now proceed to problems that can be formulated as regression problems with general-covariance Gaussian covariates. It is worth noting that for all the problems considered in this section, the RGA property is determined by the second order `Hessian' curvature condition with respect to the first argument of the loss.}

\subsubsection{Learning Complete Orthogonal Dictionaries (COD)}

Our {first example} is Learning Complete Orthogonal Dictionaries (COD), which was defined in \Cref{prob:dictionary}. 
Recall that under \Cref{prob:dictionary}, the underlying sparse code vector $\vz$ follows the Gaussian-Bernoulli distribution such that $\vz = \vb\odot\vp$ where $\vb$ is the mask vector whose entries  follow $b_i\sim{\rm Bernoulli}(\eta)$ and $\vp\sim\calN(0,\mI)$.
Let us define $\vz_\vb$ as the sparse code given the masks $\vb$, then it is easy to see $\vz_\vb$ follows zero-mean Gaussian with non-identity covariance $\vz_\vb\sim\calN(0,\mSigma_\vb)$, where $\mSigma_\vb = \diag(\vb)$.
To address COD, the idea is to verify \Cref{assum:ell-Gaussian-sigma} for the 
conditioned random variable $\vz_\vb$ and then take the expectation with respect to $\vb$. To streamline the argument, we first extend the sufficient conditions stated in \Cref{assum:ell-Gaussian-sigma} to `conditioned' general variance Gaussians:
\begin{corollary}\label{coro:conditioned-gaussian-Sigma}
    Suppose we are given \eqref{problem:general-regression-formulation} whose underlying distribution $\vz$ satisfies that, conditioned on some auxiliary random variable $\vb$, it holds $\vz_\vb = \vz|\vb\sim\calN(0,\mSigma_\vb)$, with $\mSigma_\vb\succeq 0$ almost surely.
    Fixing any target vector $\vx^*\in\spd$ and any $\vx\in\spd$ such that $\theta:=\theta(\vx,\vx^*)\in(0,\pi/2)$, suppose
    \begin{align*}
        C_\vb^{11}\eqdef \E_{\vz_\vb}[D_1\partial_1\ell(\vx\cdot\vz_\vb,\vx^*\cdot\vz_\vb;\vx,\vx^*)|\vb],\; C_\vb^{12}\eqdef \E_{\vz_\vb}[D_2\partial_1\ell(\vx\cdot\vz_\vb,\vx^*\cdot\vz_\vb;\vx,\vx^*)|\vb],
    \end{align*}
    {exists for any $\vb$ in a given event $\vb\in\cE_\vb$.}
    Let $\vv := \vx^*_{\perp\vx}/\|\vx^*_{\perp\vx}\|_2$, and suppose that the following conditions hold:
    \begin{align*}
        \E_{\vb}[(C_\vb^{11} + \cos(\theta)C_{\vb}^{12})(\vx^\top\mSigma_\vb\vv)\1\{\vb\in\cE_\vb\}]\leq -\mu^{(1)},\;\;\E_{\vb}[C_{\vb}^{12}(\vv^\top\mSigma_\vb\vv)\1\{\vb\in\cE_\vb\}]\leq -\mu^{(2)}.
    \end{align*}
    Given the corresponding Riemannian correlation vector field {$\vg_\ell(\vx,\vx^*) = \E_{\vb}[\vg_\ell(\vx,\vx^*|\vb)\1\{\vb\in\cE_\vb\}]$, where $\vg_\ell(\vx,\vx^*|\vb) = \E_{\vz_\vb}[\partial_1\ell(\vx\cdot\vz_\vb,\vx^*\cdot\vz_\vb;\vx,\vx^*)(\vz_\vb)_{\perp\vx}|\vb]$}, we have:
    \begin{enumerate}
        \item If \Cref{assum:smoothness-of-regression-loss}(i) is satisfied, then $\vg_\ell(\vx,\vx^*)\cdot\vx^*\leq -((\mu^{(1)} + \sin\theta\mu^{(2)})/L_1)\sin\theta\|\vg_\ell(\vx,\vx^*)\|_2$; 
        \item If \Cref{assum:smoothness-of-regression-loss}(ii) is satisfied, then $\vg_\ell(\vx,\vx^*)\cdot\vx^*\leq -((\mu^{(1)} + \sin\theta\mu^{(2)})/(2L_2))\|\vg_\ell(\vx,\vx^*)\|_2$. 
    \end{enumerate}
\end{corollary}
\begin{proof}
    The proof is a direct application of \Cref{claim:nablarf-x*-bounded-by-sin2theta} and \Cref{lem:regression-loss-satisfy-sharpness-Gaussian-sigma}. Given $\vb$, since $\vz_\vb = \vz|\vb\sim\calN(0,\mSigma_\vb)$, then by \Cref{eq:Gaussian-sigma-expression-vg-cdot-v} in \Cref{claim:nablarf-x*-bounded-by-sin2theta}, we obtain that for this given $\vb$,
    \begin{align*}
        \vg_\ell(\vx,\vx^*|\vb)\cdot\vx^* = (C_\vb^{11} + \cos(\theta)C_{\vb}^{12})\vx^\top\mSigma_\vb\vv \sin\theta + C_\vb^{12}\vv^\top\mSigma_\vb\vv \sin^2\theta.\end{align*}
    Therefore, taking the expectation with respect to $\vb$ and using our assumptions, we have 
    \begin{align*}
        \vg_\ell(\vx,\vx^*)\cdot\vx^* &= \E_{\vb}[\vg_\ell(\vx,\vx^* | \vb)\cdot\vx^*\1_{\cE_\vb}] = \E_{\vb}[\1_{\cE_\vb}(C_\vb^{11} + \cos(\theta)C_{\vb}^{12})\vx^\top\mSigma_\vb\vv] \sin\theta + \E_{\vb}[\1_{\cE_\vb} C_\vb^{12}\vv^\top\mSigma_\vb\vv] \sin^2\theta\\
        &\leq -\mu^{(1)}\sin\theta - \mu^{(2)}\sin^2\theta.
    \end{align*}
    The rest follows from the proof of \Cref{lem:regression-loss-satisfy-sharpness-Gaussian-sigma}.
\end{proof}

We can now proceed to show that the $L_1$ loss for COD is RGA, recovering the result from~\cite{bai2018subgradient,Zhu2019LinearConvergenceGrassman} within our more general framework.

\begin{theorem}\label{lem:COD-is-RGA}
    Let $\vz\sim\D$ where 
    $\vz = \vp\odot\vb$, $\vp\sim\calN(0,\mI_d)$ and $\vb = (b_1,\dots,b_d),$ with $b_1, \dots, b_d \stackrel{i.i.d.}{\sim}\mathrm{Bernoulli}(\eta)$.
    Let $\vx^* = \ve_1$ and denote $\theta = \theta(\vx,\vx^*)$. Consider the loss function $f_{\mathsf{dic}}(\vx) = \Evz[|\vx\cdot\vz|]$.  Furthermore, let $\xi > 1$ be an absolute constant  and define the region
    \begin{align*}
        S_\alpha = \bigg\{\vx\in\spd: x_1\geq \xi\max_{i\geq 2}|x_i|, \cos\theta\geq \alpha\bigg\}, \;\; \alpha > 0.
    \end{align*}
Then, $f_{\mathsf{dic}}(\vx)$ is $\mu$-RGA on $S_\alpha$ with $\mu = \mu_\alpha = {\alpha^3\eta(1 - \eta)(1 - 1/\xi^2)}/{\sqrt{2\pi}}$.
\end{theorem}

\begin{proof}
{
The objective function $f_{\mathsf{dic}}(\vx)$ is non-smooth, so we need to prove that any choice of subgradient field $\nablar f_{\mathsf{dic}}(\vx)\in \partial^{C,R}f_{\mathsf{dic}}(\vx)$ satisfies RGA.
Let $\ell_{\mathsf{dic}}(\vx\cdot\vz,\vx^*\cdot\vz;\vx,\vx^*) = \ell_{\mathsf{dic}}(\vx\cdot\vz) = |\vx\cdot\vz|$. Since $\ell_{\mathsf{dic}}$ is Lipschitz, by Aumann expectation (\Cref{fact:Aumann-expectation}), 
\begin{align*}
    \nablar f(\vx)\in \partial^{C,R}f_{\mathsf{dic}}(\vx) &\subseteq\bigg\{\Evz[s_\vx(\vz)\vz_{\perp\vx}]: \text{$s_\vx(\vz)$ is an integrable and measurable selection from $\partial_1^C\ell_{\mathsf{dic}}(\vx\cdot\vz)$}\bigg\}.
\end{align*}

We first observe that for any choice of the Aumann expectation, it holds that $\Evz[s_\vx(\vz)\vz_{\perp\vx}]\cdot\vx^*=\vg_\ell(\vx,\vx^*)\cdot \vx^*$, where $\vg_\ell(\vx,\vx^*) = \E[\partial_1\ell_{\mathsf{dic}}(\vx\cdot\vz)\vz_{\perp\vx}]$ is the correlation vector field. In particular, we choose $\partial_1\ell_{\mathsf{dic}}(u) = \sign(u)\in\partial^C\ell_{\mathsf{dic}}(u)$ with the convention that $\sign(0) = 0$.
By the problem definition, the sparse code vector conditioned on the masking $\vb$ follows a Gaussian distribution: $\vz_\vb = \vz|\vb \sim\calN(\vzero,\mSigma_\vb)$ with $\mSigma_\vb = \diag(\vb)$. 
Let us denote $\vg_\ell(\vx,\vx^*|\vb) \eqdef \E[\sign(\vx\cdot\vz)\vz_{\perp\vx}|\vb]$, and let $\cE_\vb = \{\vb:\vx^\top\mSigma_\vb\vx > 0\}$.

First, condition on the event $\vb\in\cE_\vb$. Then since $\vx\cdot\vz_\vb\sim\calN(\vzero,\vx^\top\mSigma_\vb\vx)$ we have $\vx\cdot\vz_\vb\neq 0$ a.s., implying that $\partial_1^C\ell_{\mathsf{dic}}(\vx\cdot\vz_\vb)=\sign(\vx\cdot\vz_\vb)$ a.s., and hence for any choice of $s_\vx(\vz)$ we have $\E[s_\vx(\vz_\vb)(\vz_\vb)_{\perp\vx}|\vb] = \vg_\ell(\vx,\vx^*|\vb)$. 
Now condition on the degenerate event $\cE_\vb^c = \{\vb:\vx^\top\mSigma_\vb\vx = \sum_{i=1}^d b_ix_i^2 = 0\}$, i.e., $\vx\cdot\vz_\vb = 0$ is a degenerate distribution. 
In this case, it must be $b_i = 0$ whenever $x_i \neq 0$. Note that for any $\vx\in\cS_\alpha$, $x_1>0$, hence it must be $b_1 = 0$. Therefore, since $\vx^* = \ve_1$, it further holds $(\vz_{\vb})_{\perp\vx}\cdot\vx^* = \vz_{\vb}\cdot\vx^* - (\vz_\vb\cdot\vx)(\vx\cdot\vx^*) = 0$. On the other hand, when $\vx\cdot\vz_\vb = 0$, by definition of the sign function we have $\vg_\ell(\vx,\vx^*|\vb) = \E[\sign(0)(\vz_\vb)_{\perp\vx}|\vb] = 0$. 
Thus, for $\vx\in\cS_\alpha$ we have $\E[s_\vx(\vz_\vb)((\vz_\vb)_{\perp\vx}\cdot\vx^*)|\vb]\1\{\vb\in\cE_\vb^c\} = 0 = \vg_\ell(\vx,\vx^*|\vb)\cdot\vx^*\1\{\vb\in\cE_\vb^c\}$.
We can now conclude that for any selection $s_\vx(\vz)$, $\Evz[s_\vx(\vz)\vz_{\perp\vx}]\cdot\vx^*=\E_\vb[\vg_\ell(\vx,\vx^*|\vb)]\cdot\vx^* = \vg_\ell(\vx,\vx^*)\cdot \vx^*$.}

Therefore, to prove that every vector field $\nablar f(\vx)\in \partial^{C,R}f_{\mathsf{dic}}(\vx)$ satisfies RGA, it suffices to show that the vector field $\vg_\ell(\vx,\vx^*)$ is RGA. To this aim, we apply \Cref{coro:conditioned-gaussian-Sigma}. 
{Note that, since $\vg_\ell(\vx,\vx^*|\vb)\1\{\vb\in\cE_\vb^c\} = 0$ as we have argued above, it holds that $\vg_\ell(\vx,\vx^*) = \E_{\vb}[\vg_\ell(\vx,\vx^*|\vb)\1\{\vb\in\cE_\vb\}]$.}
Since the loss $\ell_{\mathsf{dic}}(u,v;\vx,\vx^*) = |u|$ does not depend on the second argument $v$, it always holds that $D_2\partial_1\ell_{\mathsf{dic}}(u,v;\vx,\vx^*) = 0$, and hence, the random variable $C_\vb^{12} = \E_{\vz_\vb}[D_2\partial_1\ell_{\mathsf{dic}}(\vx\cdot\vz_\vb,\vx^*\cdot\vz_\vb;\vx,\vx^*)] = 0$.
    Therefore, according to \Cref{coro:conditioned-gaussian-Sigma}, to prove that $\vg_\ell$ is RGA, it only remains to bound $\E_\vb[C_\vb^{11}(\vx^\top\mSigma_\vb\vv)\1_{\cE_\vb}]$ and the norm of the vector field $\vg_\ell$.
    {Consider $\vb\in\cE_\vb$, i.e., $\vx^\top\mSigma_\vb\vx>0$.} Recall that as we have stated in \Cref{fact:distributional-derivatives}, the distributional derivative of the $\sign$ function is the Dirac delta function, hence $D_1\partial_1\ell_{\mathsf{dic}}(\vx\cdot\vz_\vb,\vx^*\cdot\vz_\vb;\vx,\vx^*) = D(\sign(\vx\cdot\vz_\vb)) = 2\delta(\vx\cdot\vz_\vb)$. Given a mask vector $\vb$ and a unit vector $\vx$, we have $\vx\cdot\vz_\vb\sim\calN(0, \vx^\top\mSigma_\vb\vx)$, and the conditioned random variable $C_\vb^{11}\vx^\top\mSigma_\vb\vv$ is (recalling that $\mSigma_\vb = \diag(\vb)$):
    \begin{align*}
        C_\vb^{11}\vx^\top\mSigma_\vb\vv = 2\vx^\top\mSigma_\vb\vv \E_{\vz_\vb}[\delta(\vx\cdot\vz_\vb)] = 2\vx^\top\diag(\vb)\vv\int_\R \delta(u)p_{\vx\cdot\vz_\vb}(u)\diff{u} = \frac{2\vx^\top\diag(\vb)\vv}{\sqrt{2\pi(\vx^\top\diag(\vb)\vx)}}.
    \end{align*}
    Since $\vx^*= \ve_1$, $\vv = \vx^*_{\perp\vx}/\|\vx^*_{\perp\vx}\|_2$, and $\|\vx^*_{\perp\vx}\|_2 = \sin\theta = \sqrt{1 - x_1^2}$, we have that 
    \begin{align*}
        \vv = \frac{(\mI - \vx\vx^\top)\ve_1}{\sin\theta} = \frac{1}{\sqrt{1 - x_1^2}}(1 - x_1^2, -x_1x_2,\dots, -x_1x_d)^\top.
    \end{align*}
    Therefore, the expectation $\E_\vb[C_\vb^{11}(\vx^\top\mSigma_\vb\vv)\1_{\cE_\vb}] = \E_\vb[2\frac{\sum_{i=1}^d b_ix_iv_i}{\sqrt{2\pi\sum_{i=1}^d b_ix_i^2}}\1_{\cE_\vb}]$. {Since on the event $\cE_\vb^c$ it holds $\sum_{i=1}^d b_ix_iv_i = \sum_{i=1}^d b_i x_i^2 = 0$, using the convention that $0/\sqrt{0} = 0$}, we further have:
    \begin{align*}
        \E_\vb[C_\vb^{11}(\vx^\top\mSigma_\vb\vv) \1_{\cE_\vb}] = \frac{2}{\sqrt{2\pi}}\E_\vb\bigg[\frac{\sum_{i=1}^d b_ix_iv_i}{\sqrt{\sum_{i=1}^d b_ix_i^2}}(\1_{\cE_\vb} + \1_{\cE_\vb^c})\bigg] = \frac{2}{\sqrt{2\pi(1 - x_1^2)}}\E_\vb\bigg[\frac{b_1 x_1(1 - x_1^2) - \sum_{i=2}^d b_i x_1 x_i^2}{\sqrt{\sum_{i=1}^d b_ix_i^2}}\bigg].
    \end{align*}
Since $b_i\stackrel{i.i.d.}{\sim}{\rm Bernoulli}(\eta)$, conditioning on whether $b_1 = 0$, we have:
    \begin{align*} \E_\vb[C_\vb^{11}(\vx^\top\mSigma_\vb\vv)\1_{\cE_\vb}]
        &= \frac{2}{\sqrt{2\pi(1 - x_1^2)}}\bigg(\E_{b_2,\dots,b_d}\bigg[-\frac{\sum_{i=2}^d b_i x_1 x_i^2}{\sqrt{\sum_{i=2}^d b_ix_i^2}}\bigg|b_1 = 0\bigg]\pr(b_1 = 0)\\
       & \hspace{1.2in}+ \E_{b_2,\dots,b_d}\bigg[\frac{x_1(1 - x_1^2 -\sum_{i=2}^d b_i x_i^2)}{\sqrt{x_1^2 + \sum_{i=2}^d b_ix_i^2}}\bigg|b_1 = 1\bigg]\pr(b_1 = 1)\bigg)\\
        &=\frac{2 x_1}{\sqrt{2\pi(1 - x_1^2)}}\bigg(-(1 - \eta)\E_{b_2,\dots,b_d}\bigg[\frac{\sum_{i=2}^d b_i x_i^2}{\sqrt{\sum_{i=2}^d b_ix_i^2}}\bigg] + \eta\E_{b_2,\dots,b_d}\bigg[\frac{\sum_{i=2}^d (1 - b_i) x_i^2}{\sqrt{x_1^2 + \sum_{i=2}^d b_ix_i^2}}\bigg]\bigg)\\
        &=\frac{2 x_1}{\sqrt{2\pi(1 - x_1^2)}}\bigg(-(1 - \eta)\sum_{j=2}^d \E_{b_i,i\geq 2, i\neq j}\bigg[\frac{x_j^2}{\sqrt{x_j^2 + \sum_{i\geq 2, i\neq j} b_ix_i^2}}\bigg|b_j = 1\bigg]\pr[b_j = 1] +\\
        &\quad\quad + \eta\sum_{j=2}^d \E_{b_i,i\geq 2, i\neq j}\bigg[\frac{x_j^2}{\sqrt{x_1^2 + \sum_{i\geq 2, i\neq j} b_ix_i^2}}\bigg|b_j = 0\bigg]\pr[b_j = 0]\bigg)\\
        &=-\frac{2 x_1\eta (1 - \eta)}{\sqrt{2\pi(1 - x_1^2)}}\sum_{j=2}^d \E_{b_i,i\geq 2, i\neq j}\bigg[\frac{x_j^2}{\sqrt{x_j^2 + \sum_{i\geq 2, i\neq j} b_ix_i^2}} - \frac{x_j^2}{\sqrt{x_1^2 + \sum_{i\geq 2, i\neq j} b_ix_i^2}}\bigg].
    \end{align*}
    Using the assumption that $x_1\geq \xi\max_{i\geq 2}|x_i|$ for some constant $\xi>1$, it holds that $x_i^2\leq x_1^2/\xi^2 = \cos^2\theta/\xi^2$. Furthermore, observe $\sum_{i\geq 2, i\neq j}b_i x_i^2 \leq \sin^2\theta - x_j^2$, and the function $f(x) = 1/\sqrt{a+x} - 1/\sqrt{b+x}$ is non-increasing in $x$ when $0\leq a\leq b$. Thus, it holds that:
    \begin{align*}
        \frac{1}{\sqrt{x_j^2 + \sum_{i\geq 2, i\neq j}b_i x_i^2}} - \frac{1}{\sqrt{\cos^2\theta + \sum_{i\geq 2, i\neq j}b_i x_i^2}}&\geq \frac{1}{\sin\theta} - \frac{1}{\sqrt{1 - x_j^2}}\\
        &\geq \frac{1}{\sin\theta} - \frac{1}{\sqrt{1 - \cos^2\theta/\xi^2}}\geq \frac{(1 - 1/\xi^2)\cos^2\theta}{2\sin\theta}.
    \end{align*}

    Recalling that $\sum_{j\geq 2}x_j^2 = \sin^2\theta = 1-x_1^2$, we can finally conclude that in the region $S_\alpha$ defined in the statement of the theorem:
    \begin{align*}
        \E_\vb[C_\vb^{11}(\vx^\top\mSigma_\vb\vv)\1_{\cE_\vb}]\leq -\frac{2 x_1\eta (1 - \eta)}{\sqrt{2\pi(1 - x_1^2)}}\sum_{j=2}^d \frac{x_j^2(1 -1/\xi^2)x_1^2}{2\sqrt{1 - x_1^2}}\leq -\frac{2 x_1^3\eta(1 - \eta)(1 -1/\xi^2)}{2\sqrt{2\pi}}\leq -\frac{\alpha^3\eta(1 - \eta)(1 -1/\xi^2)}{\sqrt{2\pi}}.
    \end{align*}

    For the norm of the vector fields $\Evz[s_\vx(\vz)\vz_{\perp\vx}]$, $s_\vx(\vz)\in[-1,1]$, we have
    \begin{align*}
        \|\Evz[s_\vx(\vz)\vz_{\perp\vx}]\|_2 & = \sup_{\vu\in\spd}\E_{b_1,\dots, b_d \stackrel{i.i.d.}{\sim}\mathrm{Bernoulli}(\eta)}\bigg[\E_{\vz_\vb\sim\calN(0,\diag(\vb))}[s_\vx(\vz)(\vz_\vb\cdot\vu_{\perp\vx})|\vb]\bigg]\\
        &\leq \sup_{\vu\in\spd}\E_{b_1,\dots, b_d \stackrel{i.i.d.}{\sim}\mathrm{Bernoulli}(\eta)}\bigg[\sqrt{\E_{\vz_\vb\sim\calN(0,\diag(\vb))}[(\vz_\vb\cdot\vu_{\perp\vx})^2]}\bigg]\leq 1,
    \end{align*}
where the last inequality is by Cauchy-Schwarz. 

Thus, applying \Cref{coro:conditioned-gaussian-Sigma} with $\mu^{(1)} = \frac{\alpha^3\eta(1 - \eta)(1 -1/\xi^2)}{\sqrt{2\pi}}, \mu^{(2)} = 0, $ and $L_1 = 1$, we conclude that the correlation vector field $\vg_\ell(\vx,\vx^*)$ is $\mu$-RGA on $S_\alpha$ with $\mu = \frac{\alpha^3\eta(1 - \eta)(1 - 1/\xi^2)}{\sqrt{2\pi}}$. 
    Therefore, all subgradient fields in $\partial^{C,R}f_{\mathsf{dic}}(\vx)$
    are $\mu$-RGA on $S_\alpha$ with $\mu = \frac{\alpha^3\eta(1 - \eta)(1 - 1/\xi^2)}{\sqrt{2\pi}}$. \end{proof}

The intuition behind the proof of \Cref{lem:COD-is-RGA} is the following. Let us define $f_{\mathsf{dic}}(\vx|\vb)\eqdef \E_{\vz\sim\calN(0,\diag(\vb))}[|\vx\cdot\vz|]$ and for $\beta \in \{0, 1\}$, $f_{\mathsf{dic}}(\vx|b_1 = \beta) = \E_{b_2,\dots,b_d}[f_{\mathsf{dic}}(\vx|(\beta, b_2, \dots, b_d))]$. Note that $f_{\mathsf{dic}}(\vx) = (1-\eta) f_{\mathsf{dic}}(\vx|b_1 = 0) + \eta f_{\mathsf{dic}}(\vx|b_1 = 1).$ 
\Cref{coro:conditioned-gaussian-Sigma} tells us that the signal part of the loss, $f_{\mathsf{dic}}(\vx|b_1 = 0)$ provides a strong signal for $\vx^*$ with a large curvature parameter compared to the noise part of the loss, $f_{\mathsf{dic}}(\vx|b_1 = 1)$.

\subsubsection{Spiked Sparse Covariance PCA}\label{sec:spiked-sparse-pca}

In this section, we prove that the lasso-regularized Rayleigh Quotient objective for the spiked sparse covariance PCA problem (\Cref{prob:spiked-sparse-pca}) satisfies RGA with order $r = 2$, locally near the planted sparse vector. 
Recall that as formulated in \Cref{prob:spiked-sparse-pca}, our goal is to optimize the objective $f_\lambda(\vx) = -\vx^\top\mSigma\vx + \lambda\|\vx\|_1$, $\lambda\geq 0$, on the sphere, where the covariance matrix $\mSigma$ satisfies the spiked-sparse signal model: $\mSigma = \mI + \beta\vx^*\vx^{*\top}$ for some $\beta>0$. Here, $\vx^*\in\spd$ is the planted $s$-sparse vector, i.e., on an index set $\cI\in[d]$, $|\cI| = s$, it holds $x^*_i \in\{\pm1/\sqrt{s}\}$ and $x^*_i = 0$ for other indices.

To apply our sufficient conditions, the first step is to transform the objective $f_\lambda(\vx)$ into a regression loss under a Gaussian distribution.
\begin{claim}\label{claim:convert-spiked-sparse-pca-to-regression}
    Consider the spiked sparse covariance problem (\Cref{prob:spiked-sparse-pca}). Let $\vp_0,\vp_1,\dots,\vp_d$ be independent standard Gaussian random vectors in $\R^d$. Define: $\vz_0 = (\vx^*\vx^{*\top})\vp_0$, and $\vz_i = (\ve_i\ve_i^\top)\vp_i$ for $i\in[d]$. Let $\mZ\in\R^{(d+1)\times d}$, $\mZ^\top \eqdef [\vz_0,\vz_1,\dots,\vz_d]$, and define the loss $\ell(\vu,\vv):\R^{d+1}\times \R^{d+1}\to \R$:
    \begin{align*}
        \ell(\mZ\vx) = \ell_0(\vx\cdot\vz_0) + \sum_{i=1}^d \ell_i(\vx\cdot\vz_i),\;\text{where } \ell_0(u) = -\beta u^2,\,\ell_i(u) = \frac{\lambda}{\sqrt{2/\pi}}|u|.
    \end{align*}
    Let the population objective be $F(\vx) = \E_{\mZ}[\ell(\mZ\vx)]$.
    Then, 
$\partial^{C,R} F(\vx) = \partial^{C,R} f_\lambda(\vx)$.
\end{claim}
\begin{proof}

    By the additive structure of $\ell(\mZ\vx)$, the population objective is $F(\vx) = \sum_{j=0}^d \E[\ell_j(\vx\cdot\vz_j,\vx^*\cdot\vz_j)]$, which equals:
    \begin{align*}
        F(\vx) = -\beta\E[(\vx\cdot\vx^*)^2(\vx^*\cdot\vp)^2] + \frac{\lambda}{\sqrt{2/\pi}}\sum_{i=1}^d \E[|\vx\cdot\ve_i||\ve_i\cdot\vp|] = - \beta(\vx\cdot\vx^*)^2 + \lambda\|\vx\|_1.
    \end{align*}
    Thus, $f_\lambda(\vx) = F(\vx) - 1$. Since the Riemannian derivative of a constant is zero, we have $\partial^{C,R} f_\lambda = \partial^{C,R} F$.
\end{proof}

The RGA property only concerns the vector field (in this case the gradient field) rather than the function values. 
Therefore, though the population objective defined in \Cref{claim:convert-spiked-sparse-pca-to-regression} is different from the deterministic lasso-regularized Rayleigh quotient, since their gradient fields coincide, it is sufficient to show that any choice of $\partial^{C,R} F$ satisfies RGA, which will directly imply that $f_\lambda$ is also RGA.
\Cref{claim:convert-spiked-sparse-pca-to-regression} converts the problem to minimizing a regression loss whose data distributions are independent Gaussians with general covariances, which exactly falls into the scope of our \Cref{assum:ell-Gaussian-sigma}.
However, noting that \Cref{claim:convert-spiked-sparse-pca-to-regression} defines the regression problem on $d+1$ independent Gaussian random vectors rather than just one Gaussian data distribution, we need a little bit of modification of \Cref{assum:ell-Gaussian-sigma} and \Cref{lem:regression-loss-satisfy-sharpness-Gaussian-sigma} for the convenience of implementation.
\begin{corollary}\label{coro:summable-gaussian-Sigma}
    Suppose we are given \eqref{problem:general-regression-formulation} whose underlying distribution $\mZ\in\R^{m\times d}$, $\mZ^\top = [\vz_1,\dots,\vz_m]$, satisfies $\vz_j\sim\calN(0,\mSigma_j)$, $\mSigma_j\succeq 0$, $j\in[m]$. 
    Furthermore, fixing any target $\vx^*\in\spd$,   assume the loss function $\ell(\vu,\vv;\vx,\vx^*):\R^{m}\times \R^{m}\times \R^d\times\R^d\to\R$ is summable: $\ell(\vu,\vv;\vx,\vx^*) = \sum_{j=1}^m \ell_j(u_j,v_j;\vx,\vx^*).$
For any $\vx\in\spd$ such that $\theta\eqdef\theta(\vx,\vx^*)\in(0,\pi/2)$, and $\vx^\top\mSigma_j\vx\neq 0$, $j\in[m]$, let \begin{align*}
        C_j^{11} = \E_{\vz_j\sim\calN(0,\mSigma_j)}[D_1\partial_1\ell_j(\vx\cdot\vz_j,\vx^*\cdot\vz_j;\vx,\vx^*)], \; C_j^{12} = \E_{\vz_j\sim\calN(0,\mSigma_j)}[D_2\partial_1\ell_j(\vx\cdot\vz_j,\vx^*\cdot\vz_j;\vx,\vx^*)]. \end{align*}
Let $\vv = \vx^*_{\perp\vx}/\|\vx^*_{\perp\vx}\|_2$, and suppose the following conditions hold:
    \begin{align*}
        \sum_{j=1}^m (C_j^{11}  + \cos(\theta) C_j^{12})\vx^\top\mSigma_j\vv\leq -\mu_\theta^{(1)},\; \sum_{j=1}^m C_j^{12}\vv^\top\mSigma_j\vv\leq -\mu_\theta^{(2)},
    \end{align*}
    for some parameters $\mu_\theta^{(1)}, \mu_\theta^{(2)}\geq 0$ (possibly dependent on $\theta$). Given the corresponding Riemannian correlation vector field $\vg_\ell(\vx,\vx^*) = \E[P_{\perp\vx}\mZ^\top\partial_1\ell(\mZ\vx,\mZ\vx^*)]$, we have:
    \begin{equation*}
        \vg_\ell(\vx,\vx^*)\cdot\vx^* \leq -\mu_\theta^{(1)}\sin\theta - \mu_\theta^{(2)}\sin^2\theta.
    \end{equation*}
\end{corollary}
\begin{proof}
    Similar to \Cref{coro:conditioned-gaussian-Sigma}, the proof is a direct application of \Cref{claim:nablarf-x*-bounded-by-sin2theta} and \Cref{lem:regression-loss-satisfy-sharpness-Gaussian-sigma}. By the linearity of expectation, the Riemannian correlation vector field is 
    \begin{align*}
        \vg_\ell(\vx,\vx^*) &= \E_{\mZ}[P_{\perp\vx}\mZ^{\top}\partial_1\ell(\vu,\vv;\vx,\vx^*)|_{\vu = \mZ\vx,\vv = \mZ\vx^*}] = \sum_{j=1}^m \E_{\vz_j\sim\calN(0,\mSigma_j)}[\partial_1\ell_j(\vx\cdot\vz_j,\vx^*\cdot\vz_j;\vx,\vx^*)(\vz_j)_{\perp\vx}].
    \end{align*}
    Let us denote $\vg_{\ell_j}(\vx,\vx^*) = \E_{\vz_j\sim\calN(0,\mSigma_j)}[\partial_1\ell_j(\vx\cdot\vz_j,\vx^*\cdot\vz_j;\vx,\vx^*)(\vz_j)_{\perp\vx}]$ for $j\in[m]$, then $\vg_\ell = \sum_{j=1}^m \vg_{\ell_j}$.
    Applying \Cref{eq:Gaussian-sigma-expression-vg-cdot-v} in \Cref{claim:nablarf-x*-bounded-by-sin2theta} to each vector field $\vg_{\ell_i}(\vx,\vx^*)$ and using our assumptions, we obtain that
    \begin{align*}
        \vg_\ell(\vx,\vx^*)\cdot\vx^* = \sum_{j=1}^m (C_j^{11}  + \cos(\theta) C_j^{12})\vx^\top\mSigma_j\vv \sin\theta + \sum_{j=1}^m C_j^{12}\vv^\top\mSigma_j\vv \sin^2\theta\leq -\mu^{(1)}\sin\theta - \mu^{(2)}\sin^2\theta.\end{align*}
\end{proof}

Our main result is the following:
\begin{theorem}[Spiked PCA is high-order RGA]\label{lem:spiked-pca-order-2-rga}
    Consider \Cref{prob:spiked-sparse-pca} and suppose the signal strength $\beta$ and the regularization parameter $\lambda$ satisfy $2\beta>\lambda\sqrt{s}>0$.
    Let $\cS$ be a spherical cap centered at $\vx^*$ defined by:
    \begin{align*}
        \cS = \{\vx:\theta(\vx,\vx^*)\in(0,\Bar{\theta})\}, \;\Bar{\theta}\in\bigg(0,\arccos\bigg(\frac{\lambda\sqrt{s}}{2\beta}\bigg)\land \frac{\pi}{2}\bigg).
    \end{align*}
    Then, the lasso-regularized loss $f_\lambda(\vx)$ is $(\mu, 2)$-RGA with $\mu = \frac{2\beta\cos(\Bar{\theta}) -  \lambda\sqrt{s}}{\beta + \lambda\sqrt{d}}$ on $\cS$.
\end{theorem}
\begin{proof}
Consider the distribution and the loss defined in \Cref{claim:convert-spiked-sparse-pca-to-regression}, 
$$F(\vx) = \E[\ell(\mZ\vx)] = \E\bigg[\sum_{i=0}^d \ell_i(\vx\cdot\vz_i)\bigg],\;\text{where } \ell_0(u_0) = -\beta u_0^2,\,\ell_i(u_i) = \frac{\lambda}{\sqrt{2/\pi}}|u_i|,\, i = 1\dots,d.$$
We drop the vectors $\vx,\vx^*$ as well as $\mZ\vx^*$ in the expression of $\ell$ and $\ell_j$, $j = 0,\dots,d$ as they do not depend on $\vx$, $\vx^*$ and $\mZ\vx^*$.
By the definition of $\ell(\vu)= \sum_{j=0}^{d}\ell_j(u_j)$, it holds that $(\partial_1\ell(\vu))_j = \partial_{u_j}\ell_j(u_j)\in\partial_{u_j}^C \ell_j(u_j)$ for $j = 0,\dots, d$.
Since $\ell$ is locally Lipschitz, applying Aumann Expectation \Cref{fact:Aumann-expectation} we have that every choice $\nablar f(\vx)\in\partial^{C,R} F(\vx)$ of the Clarke subdifferential of $F(\vx)$ is contained in the set:
\begin{align*}
    \nablar f(\vx)\in\partial^{C,R}F(\vx)&\subseteq \E[\mP_{\perp\vx}\mZ^\top\partial_1^{C}\ell(\mZ\vx)] \\
    & =\bigg\{\sum_{j=0}^d\E[ s_{j,\vx}(\vz_j)(\vz_j)_{\perp\vx}]:  s_{j,\vx}(\vz_j)\in\partial_j^C\ell_j(\vx\cdot\vz_j), \text{ $s_{j,\vx}$ is $\vz_j$-measurable.}\bigg\} \end{align*}
Therefore, to prove that $\partial^{C,R}F(\vx)$ satisfies RGA, it suffices to prove that for any measurable choice of $\vs_\vx(\mZ) = (s_{0,\vx}(\vz_0),\dots,s_{d,\vx}(\vz_d))$, $s_{j,\vx}\in\partial_j^C\ell_j$, the vector field $\wh{\vg}(\vx;\vs_\vx) = \sum_{j=0}^d\E[ s_{j,\vx}(\vz_j)(\vz_j)_{\perp\vx}]$ is RGA.

Fix any $\vs_\vx(\vz)$. Note that since $\ell_0$ is a smooth function, it always holds $s_{0,\vx}(\vz_0) = -2\beta(\vx\cdot\vz_0)$.
In addition, recall that in the formulation of \Cref{claim:convert-spiked-sparse-pca-to-regression}, the random vectors $\vz_j$ follow Gaussian distribution: $\vz_j\sim\calN(0,\mSigma_j)$, where $\mSigma_0 = \vx^*\vx^{*\top}$ and $\mSigma_i = \ve_i\ve_i^\top$, $i\in[d]$.
Let $\cJ_\vx = \{j\in\{0,1,\dots,d\}: \vx^\top\mSigma_j\vx\neq 0\}$. For any $\vx\in\cS$, we have $\vx\cdot\vx^*\neq 0$ hence $\vx^\top\mSigma_0\vx\neq 0$, i.e., the index $0$ is always an element in $\cJ_\vx$. 
For any index $i$ in $\cJ_\vx\cap\{1,\dots,d\}$, the random variable $\vx\cdot\vz_i\sim\calN(0,\vx^\top\mSigma_i\vx)$ is not degenerate, and since $\vx\cdot\vz_i\neq 0$ a.s., any measurable choice $s_{i,\vx}(\vz_i) = \sign(\vx\cdot\vz_i)$ a.s., with the convention that $\sign(0) = 0$. However, for those indices $i\notin\cJ_\vx$, we have $\vx\cdot\vz_j = 0$, making the random variable degenerate, which requires more care.
Define:
\begin{align*}
    \vg_\ell(\vx,\vx^*;\cJ_\vx) &= \sum_{j\in\cJ_\vx}\vg_{\ell_j}(\vx,\vx^*) = \sum_{j\in\cJ_\vx}\E[\partial_1\ell_j(\vx\cdot\vz_j)(\vz_j)_{\perp\vx}]\\
    &=-2\beta\E[(\vx\cdot\vz_0)(\vz_0)_{\perp\vx}] +  \sum_{j\in\cJ_\vx,j\neq 0}\frac{\lambda}{\sqrt{2/\pi}}\E[\sign(\vx\cdot\vz_j)(\vz_j)_{\perp\vx}]. 
\end{align*}
Then, the inner product $\wh{\vg}(\vx;\vs_\vx)\cdot\vx^*$ equals:
\begin{align}\label{eq:spiked-pca-curvature-eq-0}
    \wh{\vg}(\vx;\vs_\vx)\cdot\vx^* = \vg_\ell(\vx,\vx^*;\cJ_\vx)\cdot\vx^* + \sum_{i\notin \cJ_\vx, i\in[d]} \E[s_{i,\vx}(\vz_i)(\vz_i)_{\perp\vx}]\cdot\vx^*. 
\end{align}

The first term is the inner product of a correlation vector field, to which the \Cref{coro:summable-gaussian-Sigma} is directly applicable.  
First note that $\ell_j(u)$ does not depend on the second argument, therefore, $C_j^{12} = 0$ for all $j \in \cJ_\vx$. Hence, we only need to focus on $C_j^{11} = \E_{\vz_j}[D_1\partial_1\ell_j(\vx\cdot\vz_j)]$, $j \in\cJ_\vx$. Let $\vv = \vx^*_{\perp\vx}/\|\vx^*_{\perp\vx}\|_2$. For $j=0$:
    \begin{align*}
        C_0^{11} = \E_{\vz_0}[D_1\partial_1\ell_0(\vx\cdot\vz_0)] = \E_{\vz_0}[\partial^2_{u_0}(-\beta u_0^2)|_{u_0 = \vx\cdot\vz_0}] = -2\beta,\; \vx^\top\mSigma_0\vv = \vx^\top(\vx^*\vx^{*\top})\vv = \cos\theta\sin\theta.
    \end{align*}
    For $i \in\cJ_\vx\cap[d]$, since $\partial_1\ell_i(u_i) = (\lambda/\sqrt{2/\pi})\sign(u_i)$, by \Cref{fact:distributional-derivatives} we have that  $D_{u_i}\partial_{u_i}\ell_i(u_i) = (\lambda/\sqrt{2/\pi})2\delta(u_i)$ where $\delta$ is the Dirac function. 
    Noting that, according to \Cref{claim:convert-spiked-sparse-pca-to-regression}, $\vx^\top\mSigma_i\vx = x_i^2$, it holds $\vx\cdot\vz_i\sim\calN(0,x_i^2)$ and $x_i\neq 0$ for $i \in\cJ_\vx\cap[d]$, therefore, the second-order curvature expectation $C_i^{11}$ equals:
    \begin{align*}
        C_i^{11} = \E_{\vz_i}[D_1\partial_1\ell_i(\vx\cdot\vz_i)] = \E_{\vz_i}[D_{u_i}\partial_{u_i}\ell_i(u_i)|_{u_i = \vx\cdot\vz_i}] = \frac{\lambda}{\sqrt{2/\pi}}\int_\R 2\delta(z_i)\frac{1}{\sqrt{2\pi}|x_i|}\exp\bigg(-\frac{z_i^2}{2x_i^2}\bigg)\diff{z_i} = \frac{\lambda}{|x_i|}. 
    \end{align*}
    Furthermore, by the definition of $\mSigma_i = \ve_i\ve_i^\top$, we have $\vx^\top\mSigma_i\vv  = x_i v_i$.
Therefore, the sum of the joint curvature conditions equals:
    \begin{align*}\sum_{j\in\cJ_\vx} C_j^{11}\vx^\top\mSigma_j\vv &= -2\beta\cos(\theta)\sin(\theta) + \sum_{i\in[d]\cap\cJ_\vx} \frac{\lambda}{|x_i|}x_iv_i = -2\beta\cos(\theta)\sin(\theta) + \lambda \sum_{i\in[d]\cap\cJ_\vx}\sign(x_i)v_i=:-\mu_\theta^{(1)} . \end{align*}
    Therefore, $\vg_\ell(\vx,\vx^*;\cJ_\vx)$ satisfies \Cref{coro:summable-gaussian-Sigma} with $\mu_\theta^{(1)}$ presented above and $\mu_\theta^{(2)} = 0$. Applying \Cref{coro:summable-gaussian-Sigma}, we have $\vg_\ell(\vx,\vx^*;\cJ_\vx)\cdot\vx^*\leq -\mu_\theta^{(1)}\sin\theta$, and combining with \eqref{eq:spiked-pca-curvature-eq-0}, we get:
    \begin{align}\label{eq:spiked-pca-curvature-eq-1.5}
        \wh{\vg}(\vx;\vs_\vx)\cdot\vx^* \leq \bigg(-2\beta\cos(\theta)\sin(\theta) + \lambda \sum_{i\in[d]\cap\cJ_\vx}\sign(x_i)v_i\bigg)\sin\theta + \sum_{i\notin \cJ_\vx, i\in[d]} \E[s_{i,\vx}(\vz_i)(\vz_i\cdot\vv)]\sin\theta.
    \end{align}
    For the second term in \eqref{eq:spiked-pca-curvature-eq-1.5}, since $s_{i,\vx}(\vz_i)\in\partial_1^C\ell_i(0) = [-\lambda/\sqrt{2/\pi},\lambda/\sqrt{2/\pi}]$, we have that
    \begin{align*}
        \E[s_{i,\vx}(\vz_i)(\vz_i\cdot\vv)] \leq \E[|s_{i,\vx}(\vz_i)||\vz_i\cdot\vv|]\leq \frac{\lambda}{\sqrt{2/\pi}}\E_{\vp\sim\calN_d}[|\ve_i\cdot\vp|]|v_i|\leq \lambda |v_i|.
    \end{align*}
    Plugging this back into the right-hand side of \eqref{eq:spiked-pca-curvature-eq-1.5} above, we obtain:
    \begin{align}\label{eq:spiked-pca-curvature-eq-1}
        \wh{\vg}(\vx;\vs_\vx)\cdot\vx^* &\leq \bigg(-2\beta\cos(\theta)\sin(\theta) + \lambda \bigg(\sum_{i\in[d]\cap\cJ_\vx}\sign(x_i)v_i + \sum_{i\in[d]\cap\cJ_\vx^c}\sign(v_i)v_i\bigg)\bigg)\sin(\theta) \nonumber\\
        &= (-2\beta\cos(\theta)\sin(\theta) + \lambda\bm{\xi}\cdot\vv)\sin(\theta),
    \end{align}
where $\bm{\xi}\in\R^d$ is such that $\xi_i  = \sign(x_i)$ for $i\in[d]\cap\cJ_\vx$ and $\xi_i = \sign(v_i)$ when $i\in[d]\cap\cJ_\vx^c$. 
    Note that $x_i = 0$ for $i\in[d]\cap\cJ_\vx^c$, therefore,  $\bm{\xi}\cdot\vx = \|\vx\|_1$, and recalling that it holds $\vx^* = \cos(\theta)\vx + \sin(\theta)\vv$, we have: $\sin(\theta)\bm{\xi}\cdot\vv = \bm{\xi}\cdot\vx^* - \cos(\theta)\bm{\xi}\cdot\vx\leq\bm{\xi}\cdot\vx^*- \cos(\theta)\|\vx\|_1$. Furthermore, since $\vx^*$ is an $s$-sparse vector with uniform coordinate amplitudes (see \Cref{prob:spiked-sparse-pca}), we have $\bm{\xi}\cdot\vx^*\leq \|\vx^*\|_1 = \sqrt{s}$. Therefore, $\sin(\theta)\bm{\xi}\cdot\vv\leq \sqrt{s} - \cos(\theta)\|\vx\|_1$. The norm $\|\vx\|_1$ cannot be small when $\vx$ is far away from $\vx^*$. 
    In fact, denoting the index set on which $x^*_i\neq 0$ by $\cI$, $|\cI| = s$, we have:
    \begin{align*}
    	\|\vx\|_1 = \|\vx_{\cI}\|_1 + \|\vx_{\cI^c}\|_1 \geq \sum_{i\in\cI} x_i\sign(x_i)\geq \sum_{i\in \cI}x_i (\sqrt{s}  x^*_i) \geq \sqrt{s}\vx\cdot\vx^* = \sqrt{s}\cos\theta.
    \end{align*}
    This implies that $\sin(\theta)\bm{\xi}\cdot\vv\leq \sqrt{s} - \sqrt{s}\cos^2(\theta)$, yielding that $\bm{\xi}\cdot\vv\leq \sqrt{s}\sin(\theta)$.
    Thus, plugging the bound on $\bm{\xi}\cdot\vv$ back into \eqref{eq:spiked-pca-curvature-eq-1}, and noting that $\cos(\theta)\geq \cos(\bar{\theta})$ when $\vx\in\cS$, we conclude that 
\begin{align}\label{eq:spiked-pca-curvature-eq-2}
        \wh{\vg}(\vx;\vs_\vx)\cdot\vx^* \leq -(2\beta \cos(\bar{\theta}) - \lambda \sqrt{s})\sin^2(\theta).
    \end{align}
    Note that since $\Bar{\theta}\in(0,\arccos({\lambda\sqrt{s}}/{(2\beta)}))$, it is guaranteed that $\wh{\vg}(\vx;\vs_\vx)\cdot\vx^*<0$.

    It remains to bound the norm of the vector field $\wh{\vg}(\vx;\vs_\vx)$. For any $\vs_\vx\in\partial_1^C\ell$, by definition, we have
    \begin{align*}
    	\wh{\vg}(\vx;\vs_\vx) & = \sum_{j=0}^d \E_{\vz_j}[s_{j,\vx}(\vz_j)(\vz_j)_{\perp\vx}] =-2\beta\E_{\vp\sim\calN_d}[(\vx\cdot\vx^*)(\vx^*\cdot\vp)(\vx^*\vx^{*\top}\vp)_{\perp\vx}] + \sum_{i=1}^d\E_{\vz_i}[s_{i,\vx}(\vz_i)(\vz_i)_{\perp\vx}] \\
&=-2\beta\cos(\theta)\vx^*_{\perp\vx} + \sum_{i=1}^d\E_{\vz_i}[s_{i,\vx}(\vz_i)(\vz_i)_{\perp\vx}].\end{align*}
    The first term has its norm bounded by $\|2\beta\cos(\theta)\vx^*_{\perp\vx}\|_2 = \|2\beta\cos(\theta)\sin(\theta)\vv\|_2\leq \beta$.
    For the second term, since $|s_{,\vx}(\vz_i)|\leq \lambda/\sqrt{2/\pi}$, for any unit vector $\vu\perp\vx$, it holds:
    \begin{align*}
        \sum_{i=1}^d\E_{\vz_i}[s_{i,\vx}(\vz_i)(\vz_i)_{\perp\vx}]\cdot\vu\leq \sum_{i=1}^d\E_{\vz_i}[|s_{i,\vx}(\vz_i)||\vz_i\cdot\vu|]\leq \sum_{i=1}^d\frac{\lambda}{\sqrt{2/\pi}}\E_{\vp\sim\calN_d}[|\ve_i\cdot\vp|]||u_i| = \lambda\|\vu\|_1\leq \lambda\sqrt{d}.
    \end{align*}
    Thus, $\|\sum_{i=1}^d\E_{\vz_i}[s_{i,\vx}(\vz_i)(\vz_i)_{\perp\vx}]\|_2\leq \lambda\sqrt{d}$. Applying triangle inequality yields $\|\wh{\vg}(\vx;\vs_\vx)\|_2\leq \beta + \lambda\sqrt{d}$.

Combining with \eqref{eq:spiked-pca-curvature-eq-2}, we immediately obtain that, uniformly for any  measurable selection $\vs_\vx\in\partial^C\ell$:
    $$\wh{\vg}(\vx;\vs_\vx)\cdot\vx^*\leq -\frac{2\beta\cos(\bar{\theta}) - \lambda\sqrt{s}}{\beta + \lambda \sqrt{d}}\|\wh{\vg}(\vx;\vs_\vx)\|_2\sin^2(\theta).$$
    Hence any $
    \nablar f_\lambda (\vx)\in \partial^{C,R}f_\lambda(\vx)=\partial^{C,R}F(\vx)$ is $(\mu, 2)$-RGA with $\mu = \frac{2\beta\cos(\bar{\theta}) - \lambda\sqrt{s}}{\beta + \lambda \sqrt{d}}$, when $\vx\in\cS$.
\end{proof}

The characterization that the subgradient field of $f_\lambda(\vx)$ satisfies only high-order RGA is tight, in the sense that we can construct a geodesic near $\vx^*$ so that any choice of the subgradient field $\vg(\vx)\in\partial^C f_\lambda(\vx)$ does not satisfy order-$r$ RGA when $r<2$.
In particular, we have the following lemma:
\begin{restatable}[Spiked Sparse PCA Cannot Have Smaller Order]{lemma}{spikePCAMustBeOrderTwo}\label{lem:spiked-sparse-pca-cannot-have-lower-order}
Assume $2\leq s<d$ and $\lambda >0$. Under the condition of \Cref{lem:spiked-pca-order-2-rga}, no parameters $\mu>0$ and $r<2$ can satisfy 
    \begin{align*}
        \vg(\vx)\cdot\vx^*\leq -\mu\|\vg(\vx)\|_2\sin^r\theta(\vx,\vx^*)
    \end{align*}
    uniformly in a neighborhood of $\vx^*$ for any choice of subdifferential $\vg(\vx)$. In fact, there is a geodesic path $\vx(t)\to\vx^*$ such that for any choice of the subdifferential $\vg(\vx(t))$, and for any $r<2$ it holds that 
    \begin{align*}
        \frac{\vg(\vx(t))\cdot\vx^*}{\|\vg(\vx(t))\|_2\sin^r\theta(\vx(t),\vx^*)}\to 0,\; t\to 0.
    \end{align*}
\end{restatable}
The proof is deferred to \Cref{appx:omitted-facts}.

\subsubsection{Eigenvector Problem/Optimizing Rayleigh Quotient}\label{sec:eigenvector-problem-and-optimizing-rayleigh-quotient}

Finally, we revisit the classic eigenvector problem (\Cref{prob:eigenvector-problem}). 
Recall that in \Cref{prob:eigenvector-problem} we defined $\mA$ to be a real symmetric matrix in $\R^{d\times d}$ with eigenvalues $\lambda_1\geq \lambda_2\geq \dots\geq \lambda_{r-1}>\lambda_r = \dots = \lambda_d$ and corresponding eigenvectors $\vq_1,\dots,\vq_r,\dots,\vq_d$, and that our goal is to optimize the Rayleigh quotient $f_{\mathsf{RQ}}(\vx) = \vx^\top\mA\vx$ on the sphere.
Similar to \Cref{sec:spiked-sparse-pca}, since the Rayleigh quotient is not formulated as a regression objective, our first step is to convert \Cref{prob:eigenvector-problem} to a regression problem, as in \Cref{sec:spiked-sparse-pca}.

\begin{claim}\label{claim:convert-pca-to-regression}
    Consider the eigenvector problem (\Cref{prob:eigenvector-problem}). Let $\mSigma = \mA - \lambda_d\mI\succeq 0$, and $\vz\sim\calN(0,\mSigma)$. Let $F(\vx) = \Evz[\ell(\vx\cdot\vz,\vx^*\cdot\vz;\vx,\vx^*)]$ where $\ell(u,v;\vx,\vx^*) = u^2$. Then, $\nablar F(\vx) = \nablar f_{\mathsf{RQ}}(\vx)$.
\end{claim}
\begin{proof}
    By the definition of the objective $F(\vx)$, we have $F(\vx) = \Evz[(\vx\cdot\vz)^2] = \vx^\top\mSigma\vx = f_{\mathsf{RQ}}(\vx) - \lambda_d$. Since the Riemannian gradient of a constant is zero, we have $\nablar F(\vx) = \nablar f_{\mathsf{RQ}}(\vx)$.
\end{proof}

After converting \Cref{prob:eigenvector-problem} to a regression problem, we are now ready to 
show that the gradient field of $F(\vx)$ (hence $f_{\mathsf{RQ}}(\vx)$) is RGA by proving tat it satisfies our 
sufficient conditions (\Cref{assum:ell-Gaussian-sigma}). \begin{theorem}\label{lem:rayleigh-quotient-rga}
Let $\mA$ be a real symmetric matrix in $\R^{d\times d}$ with eigenvalues $\lambda_1\geq \lambda_2\geq \dots\geq \lambda_{r-1}>\lambda_r = \dots = \lambda_d$ and corresponding eigenvectors $\vq_1,\dots,\vq_r,\dots,\vq_d$. 
    Consider \Cref{prob:eigenvector-problem}. Let $\mQ = [\vq_{r},\vq_{r+1},\dots,\vq_d]$ and define $\cS_\alpha = \{\vx\in\spd:\|\mQ^\top\vx\|_2\geq \alpha\}$.
    Let $\mu = \alpha\sqrt{\frac{\lambda_{r-1} - \lambda_d}{\lambda_1 - \lambda_d}}$. 
    Then, for any $\vx\in\cS_\alpha$, the Rayleigh Quotient $f_{\mathsf{RQ}}(\vx)$ is $\mu$-RGA at $(\vx,\vx^*(\vx))$ where $\vx^*(\vx)$ is the normalized projection of $\vx$ to the eigenspace $\mathrm{Span}(\vq_r,\dots,\vq_d)$, i.e., $\vx^*(\vx) = \mQ\mQ^\top\vx/\|\mQ\mQ^\top\vx\|_2$. 
\end{theorem}

\begin{remark}
    Let $\vx^0$ be a randomly sampled vector from the uniform distribution on the sphere and let $m = d-r+1$, which is the dimension of the subspace spanned by the eigenvectors $\vq_r,\dots,\vq_d$. Then, $\|\mQ\vx^0\|_2^2\sim \mathrm{Beta}(m/2, (d-m)/2)$. By standard concentration inequalities for Beta distributions, we have~\cite{Dasgupta2003JohnsonLindenstrauss}:
    \begin{align*}
        \pr[\vx^0\notin \cS_\alpha] = \pr[\|\mQ^\top\vx^0\|_2 < \alpha]\leq \bigg( C\alpha\sqrt{\frac{d}{m}}\bigg)^m.
    \end{align*}
    Therefore, with probability at least $1 - \delta$, $\|\mQ^\top\vx^0\|_2\geq \alpha = c\delta^{1/m}\sqrt{m/d}$, i.e., $\vx^0\in\cS_\alpha$ with $\alpha = c\delta^{1/m}\sqrt{m/d}$.
    Furthermore, for any $\vq'_d\in\sQ \eqdef \mathrm{Span}(\vq_r,\dots,\vq_d)$ such that $\vq'_d\perp \vx^*(\vx^0)$, we observe that $\vq'_d$ is orthogonal to $\nablar f_{\mathsf{RQ}}(\vx^0)$: since $\vx^*(\vx^0)$ is the projection of $\vx^0$ to  $\sQ$, and $\vq_d'\in\sQ$, $\vq_d'\perp\vx^*(\vx^0)$, it must hold that $\vq_d'\perp\vx^0$. In addition, since $\vq_d'\in\sQ$, it is an eigenvector of $\mA$ of eigenvalue $\lambda_d$. Therefore, 
    \begin{align*}
        \nablar f_{\mathsf{RQ}}(\vx^0)\cdot\vq_d' = 2\vq_d'^\top\mA\vx^0 - 2((\vx^0)^\top\mA\vx^0)(\vq'_d\cdot\vx^0) = 2\lambda_d \vq'_d\cdot\vx^0 - 2((\vx^0)^\top\mA\vx^0)(\vq'_d\cdot\vx^0) = 0,
    \end{align*}
    in other words, $\vq'_d$ is orthogonal to $\nablar f_{\mathsf{RQ}}(\vx^0)$. Thus, by an inductive argument, we know that starting from $\vx^0$, for any vector $\vx^k$ generated by \Cref{alg:RGD}, its normalized projection to the subspace $\sQ$ is always $\vx^*(\vx^0)$. Therefore, applying \Cref{lem:descent-rsi} with $\mu = \alpha\sqrt{\frac{\lambda_{r-1} - \lambda_d}{\lambda_1 - \lambda_d}}$, $\alpha = c\delta^{1/m}\sqrt{m/d}$, we have that with probability at least $1 - \delta$, \Cref{alg:RGD} converges to $\vx^*(\vx^0)\in\sQ$ and after $k = O(\log(1/\eps)/\mu^2)$ steps, $\|\vx^k - \vx^*(\vx^0)\|_2\leq \eps$.
\end{remark}

\begin{proof}[Proof of \Cref{lem:rayleigh-quotient-rga}]
We have argued in \Cref{claim:convert-pca-to-regression} that to prove $\nablar f_{\mathsf{RQ}}(\vx)$ satisfies RGA, it suffices to show that $\nablar F(\vx)$ is RGA.
Since the loss $\ell(\vx\cdot\vz,\vx^*\cdot\vz;\vx,\vx^*)$ does not depend on $\vx^*\cdot\vz$, $\vx$, and $\vx^*$, we simply write $\ell(\vx\cdot\vz)$ and drop other arguments. 
    The loss $\ell$ is simply the quadratic function, therefore $\nablar F(\vx)$ is a correlation vector field as $\nablar F(\vx) = \vg_\ell(\vx,\vx^*) = \Evz[\nablar\ell(\vx\cdot\vz)] = \Evz[\partial_1\ell(\vx\cdot\vz)\vz_{\perp\vx}]$. 
    
    We verify \Cref{assum:ell-Gaussian-sigma} and apply \Cref{lem:regression-loss-satisfy-sharpness-Gaussian-sigma}. As $\ell(u)$ does not depend on the second argument, we always have $D_2\partial_1\ell = 0$, implying that $\mu_\theta^{(2)} = 0$. Thus, we only need to show that $\Evz[\partial_u^2\ell(u)|_{u = \vx\cdot\vz}](\vx^\top\mSigma\vv) \leq -\mu^{(1)}$, where $\vv = \vx^*_{\perp\vx}/\|\vx^*_{\perp\vx}\|_2$. 
    Since $\ell(u) = u^2$, $\Evz[\partial_u^2\ell(u)|_{u = \vx\cdot\vz}] = 2$.
    Recall that, as specified in the statement of the theorem, 
$\vx^*$ is the normalized projection of $\vx$ to the eigenspace $\mathrm{Span}(\vq_r,\dots,\vq_d)$, and since the eigenspace is rotationally invariant, we can assume without loss of generality that $\vx^* = \vq_d$. As $\vx^* = \vq_d$ is the projection of $\vx$ to $\mathrm{Span}(\vq_r,\dots,\vq_d)$, it holds that $\vx\cdot\vq_i = 0$ for any $i = r,\dots,d-1$, therefore, $\vv = (\vq_d - \cos(\theta)\vx)/\sin(\theta)$. 
    Hence, the geometric filter $\vx^\top\mSigma\vv$ is bounded  above by
    \begin{align*}
        \vx^\top\mSigma\vv = \vx^\top(\mA - \lambda_d\mI)\frac{\vq_d - \cos(\theta)\vx}{\sin(\theta)} = -\frac{\cos(\theta)}{\sin(\theta)}\vx^\top(\mA - \lambda_d\mI)\vx\leq -\frac{\alpha}{\sin(\theta)}\vx^\top(\mA - \lambda_d\mI)\vx,
    \end{align*}
    as we have assumed $\theta\in(0,\arccos(\alpha))$.
    Note that since
    $$\vx^\top(\mA - \lambda_d\mI)\vx = \sum_{i=1}^d (\lambda_i - \lambda_d)(\vx\cdot\vq_i)^2\geq (\lambda_{r-1} - \lambda_d)\sum_{i=1}^{r-1}(\vx\cdot\vq_i)^2 = (\lambda_{r-1} - \lambda_d)\sin^2(\theta),$$
    we conclude that 
    \begin{align*}
        \vx^\top\mSigma\vv\leq -\frac{\alpha}{\sin(\theta)}\sqrt{\vx^\top(\mA - \lambda_d\mI)\vx}\sqrt{\vx^\top(\mA - \lambda_d\mI)\vx}\leq - \alpha\sqrt{\vx^\top\mSigma\vx}\sqrt{\lambda_{r-1} - \lambda_d}.
    \end{align*}
    Thus, 
in summary, 
    \Cref{assum:ell-Gaussian-sigma} is satisfied with $\mu_\theta^{(1)} = 2\alpha\sqrt{\vx^\top\mSigma\vx}\sqrt{\lambda_{r-1} - \lambda_d}$ and $\mu_\theta^{(2)} = 0$.

    It remains to bound the norm of the vector field $\vg_\ell(\vx,\vx^*)$. By definition, we have $\vg_\ell(\vx,\vx^*) = 2\Evz[(\vx\cdot\vz)\vz_{\perp\vx}] = 2\mP_{\perp\vx}\mSigma\vx$. Therefore, 
    \begin{align*}
        \|\vg_\ell(\vx,\vx^*)\|_2^2 = 4\|\mP_{\perp\vx}\mSigma\vx\|_2^2\leq 4\|\mSigma\vx\|_2^2 = 4\vx^\top\mSigma^2\vx.
    \end{align*}
    Since $\mSigma = \mA -\lambda_d\mI\preceq(\lambda_1 - \lambda_d)\mI$ it holds that $\mSigma^2\preceq(\lambda_1 - \lambda_d)\mSigma$. This implies that the norm of the vector field is bounded by: $\|\vg_\ell(\vx,\vx^*)\|_2\leq 2\sqrt{\lambda_1 - \lambda_d}\sqrt{\vx^\top\mSigma\vx}$.

    Therefore, applying \Cref{lem:regression-loss-satisfy-sharpness-Gaussian-sigma} with $\mu_\theta^{(1)} = 2\alpha\sqrt{\vx^\top\mSigma\vx}\sqrt{\lambda_{r-1} - \lambda_d}$, $\mu_\theta^{(2)} = 0$, and $L_1 = 2\sqrt{\lambda_1 - \lambda_d}\sqrt{\vx^\top\mSigma\vx}$, we obtain that
    \begin{align*}
        \vg_\ell(\vx,\vx^*)\cdot\vx^*\leq -\alpha\sqrt{\frac{\lambda_{r-1} - \lambda_d}{\lambda_1 - \lambda_d}}\|\vg(\vx,\vx^*)\|_2\sin(\theta).
    \end{align*}
    In other words, $\vg_\ell(\vx,\vx^*) = \nablar F(\vx) = \nablar f_{\mathsf{RQ}}(\vx)$ is $\mu$-RGA with $\mu = \alpha\sqrt{(\lambda_{r-1} - \lambda_d)/(\lambda_1 - \lambda_d)}$.
\end{proof}

\subsection{Examples under More General Distributions}\label{sec:ex:general distributions}

For more general distributions (beyond Gaussians), we primarily consider GLM and SIM problems, defined in \Cref{prob:glm+sim}. 
We argue that RGA is satisfied by GLM and SIM problems for broad classes of link functions and distributions, using our sufficient conditions from \Cref{lem:regression-loss-satisfy-sharpness-logconcave} and \Cref{lem:regression-loss-satisfy-sharpness-general-distr}, vastly simplifying and extending the results from \cite{WZDD2023,ZWDD2024}.
The link function class we consider comprises the following $(a,b)$-unbounded functions, incorporating commonly used activations such as ReLU, LeakyReLU, ELU, etc. 
This function class was first introduced in~\cite{diakonikolas2022learning} and later used in \cite{WZDD2023,ZWDD2024}.
\begin{definition}[$(a,b)$-unbounded Activations~\cite{diakonikolas2022learning}]\label{def:ab-unbounded-link}
    Let $a, b$ be positive constants such that $0<a\leq b$. We say that a function $\sigma:\R\to\R$ is $(a,b)$-unbounded, i.e., $\sigma\in\mathcal{F}_{(a,b)}$, if $0\leq\sigma'(z)\leq b$ and $\sigma'(z)\geq a$ whenever $z > 0$. At points where $\sigma$ is not differentiable, the stated inequalities apply to arbitrary elements of its Clarke subdifferential. 
\end{definition}

\subsubsection{SIMs of $(a,b)$-Unbounded Activations under Well-behaved Distributions}\label{subsubsec:logconcave-distribution-ab-link-SLM-are-RGA}

We first consider learning SIMs corresponding to $(a,b)$-unbounded activations under well-behaved distributions, defined below.

\begin{definition}[Well-Behaved Distributions] \label{def:well-behaved}
Let $L, R >0$. 
An isotropic distribution $\D$ on $\R^d$
is called $(L,R)$-well-behaved, $\D\in\mathfrak{D}_{\mathrm{well}}(L,R)$, if for every orthonormal pair $\vx, \vv \in \R^d$, the random vector $(\vx \cdot \vz, \vv \cdot \vz),$ $\vz \sim \D$, has a density $p_{\vx, \vv}$ satisfying $p_{\vx, \vv}(s, t) \geq L$ for every $(s, t) \in [-R, R]^2.$\end{definition}

The distributional family in \Cref{def:well-behaved} covers a wide range of isotropic distributions, since the only requirement is the stated anti-anti-concentration property. In particular, this is a broader class than the well-behaved family introduced in \cite{DKTZ20} and later used in \cite{diakonikolas2022learning,ZWDD2024}, as those prior works, in addition to  anti-anti-concentration, also required both concentration and anti-concentration. As a result, all example distributions covered by \cite{DKTZ20,diakonikolas2022learning,ZWDD2024} are also covered by \Cref{def:well-behaved}; for instance, all isotropic log-concave distributions and s-concave distributions belong to this family, with parameters $L$ and $R$ being universal constants, independent of the dimension.

We now proceed to the main theorem of this subsection. \cite{ZWDD2024} showed that in the Euclidean space, choosing $\sigma_{\Ltwo}(\vx\cdot\vz;\vx,\vx^*)\in\argmin_{\sigma\in\cF_{(a,b)}}\Evz[(\sigma(\vx\cdot\vz) - \sigma^*(\vx^*\cdot\vz))^2]$ to be the best-fitting link function under the $L_2^2$ loss, the empirical vector field\footnote{We use $\wh{\D}_N$ to denote the empirical distribution on $N$ i.i.d.\ samples from $\D$.} 
\begin{align}\label{eq:empirical-sim-vector-field-logconcave}
    \wh{\vg}(\vx) = \E_{\vz\sim\wh{\D}_N}[h_2'(\sigma_{\Ltwo}(\vx\cdot\vz;\vx,\vx^*) - \sigma^*(\vx^*\cdot\vz))\vz],\; h_2(u) = u^2/2
\end{align} 
satisfies the (Euclidean) gradient alignment property, that is: $\wh{\vg}(\vx)\cdot(\vx - \vx^*)\geq \mu_{\rm Euc}\|\wh{\vg}(\vx)\|_2\|\vx - \vx^*\|_2$ in a Euclidean ball, 
with parameter $\mu_{\rm Euc} = c (a^2/b^2) R^4L$ for some small constant $c > 0$; see Proposition 3.1 and Lemma 4.3 therein. 
Though they look similar, this Euclidean gradient alignment property does not directly imply RGA, as RGA relies on the Riemannian correlation vector field, which is defined differently and excludes the component parallel to $\vx$. 

Using the joint curvature conditions (\Cref{assum:ell-logconcave}), our \Cref{thm:well-behaved-sims-are-RGA} simplifies, extends, and generalizes the results from \cite{ZWDD2024} to the spherical manifold. 
We show that, in fact, vector fields similar to \Cref{eq:empirical-sim-vector-field-logconcave} but with an arbitrary smooth strongly convex penalty and defined as in \Cref{def:grad-field} satisfy the RGA property on the whole hemisphere centered at $\vx^*$, with a much larger RGA parameter $\mu$ that improves the scaling in $a/b$ from quadratic to linear and removes the dependence on the small constant $c$.

\begin{theorem}\label{thm:well-behaved-sims-are-RGA}
    Let $h:\R\to\R$ be any $m_h$-strongly convex and $L_h$-smooth penalty function with $h'(0) = 0$. Let $\sigma^*\in\cF_{(a,b)}$, and let $\D\in\mathfrak{D}_{\mathrm{well}}(L,R)$. Fix $\vx^*\in\spd$. Define the best-fitting link function under the penalty  $h$ by $\sigma_h(\cdot;\vx,\vx^*)\in\argmin_{\sigma\in L_2(\D_{\vx\cdot\vz})}\Evz[h(\sigma(\vx\cdot\vz) - \sigma^*(\vx^*\cdot\vz))]$, with minimizing functions understood up to the $\D_{\vx \cdot \vz}$-almost everywhere equality. Consider the vector field:
    \begin{align*}
        \vg_h(\vx,\vx^*) = \Evz[\partial_1\ell_h(\vx\cdot\vz,\vx^*\cdot\vz;\vx,\vx^*)\vz_{\perp\vx}],\;\text{where}\; \ell_h(u,v;\vx,\vx^*)\eqdef \int_0^{u}h'(\sigma_h(r;\vx,\vx^*) - \sigma^*(v))\diff{r}.
    \end{align*}
    Then, for any $\vx\in\spd$ such that $\theta(\vx,\vx^*)\in(0,\pi/2)$, $\vg_h(\vx,\vx^*)$ is $(m_h/L_h)(a/b)(R^4 L/6)$-RGA.
\end{theorem}
\begin{proof}
    We first verify that \Cref{assum:ell-logconcave} is satisfied by the vector field $\vg_h(\vx,\vx^*)$.
    By the definition of the loss $\ell_h(u,v;\vx,\vx^*)$, we have $\partial_1\ell_h(u,v;\vx,\vx^*) = h'(\sigma_h(u;\vx,\vx^*) - \sigma^*(v))$. Strong convexity of $h$ implies that $\sigma_h$ is unique up to $\D_{\vx \cdot \vz}$-a.e.\ equality.  Subtracting the constant $h(0)$ does not change either the best-fitting link or the vector field, so we assume w.l.o.g.\ that $h(0) = 0.$ 

    Our first claim shows that \Cref{assum:ell-logconcave}(i) is satisfied by this vector field:
    \begin{claim}\label{claim:best-fitting-activation-property}
    Let $h:\R\to\R$ be a convex and continuously differentiable function with linear derivative growth: $|h'(t)|\leq C(1+|t|)$ for some constant $C<\infty$. Then, for any square-integrable link function $\sigma^*:\R\to\R$ and unit vector $\vx^*\in\spd$, the best-fitting link function $\sigma(z;\vx,\vx^*) \in \argmin_{\sigma\in L_2(\D(\vz\cdot\vx))}\Evz[h(\sigma(\vx\cdot\vz) - \sigma^*(\vx^*\cdot\vz))]$ satisfies $\Evz[h'(\sigma(\vx\cdot\vz;\vx,\vx^*) - \sigma^*(\vx^*\cdot\vz))\mid \vx\cdot\vz] = 0$ almost surely, assuming the displayed minimum is attained. 
\end{claim}
\begin{proof}
    Let $\tilde{\sigma}\in \cL_2(\D_{\vx\cdot\vz})$ be any bounded function that is $L_2$-measurable under $\vx\cdot\vz$. 
Let $g(t)\eqdef \sigma(\vx\cdot\vz;\vx,\vx^*) + t\tilde{\sigma}(\vx\cdot\vz) - \sigma^*(\vx^*\cdot\vz)$, which is also an $L_2$ integrable function, and denote $F(t) \eqdef \Evz[h(g(t))]$ where $t\in\R$. Since $h'$ has at most linear growth, we have that for some $\tau\in(0,t),$
    \begin{align*}
        \bigg|\frac{h(g(t)) - h(g(0))}{t}\bigg|\leq |h'(g(\tau))||\tilde{\sigma}(\vx\cdot\vz)|\leq C(1 + |g(\tau)|)|\tilde{\sigma}(\vx\cdot\vz)|.
    \end{align*}
    Since the right-hand side is a square-integrable function for any $t$, we know that differentiation and expectation can be interchanged, and it holds $F'(t) = \Evz[h'(g(t))\tilde{\sigma}(\vx\cdot\vz)]$, for any bounded square-integrable function $\tilde{\sigma}$. When $t = 0$, since $\sigma(z;\vx,\vx^*)$ is the minimizing link function, we have $F'(0) = 0$. This implies $\Evz[h'(\sigma(\vx\cdot\vz;\vx,\vx^*) - \sigma^*(\vx^*\cdot\vz))\tilde{\sigma}(\vx\cdot\vz)] = 0$, for any bounded $\tilde{\sigma}(z)$. Let $A$ be any Borel set on $\R$ that is $\D_{\vx\cdot\vz}$ measurable, and define $\tilde{\sigma}(u) = \1\{u\in A\}$. Then, we obtain that $\Evz[h'(\sigma(\vx\cdot\vz;\vx,\vx^*) - \sigma^*(\vx^*\cdot\vz))\1\{\vx \cdot \vz\in A\}] = 0$, and by the definition of conditional expectation, we have that $\Evz[h'(\sigma(\vx\cdot\vz;\vx,\vx^*) - \sigma^*(\vx^*\cdot\vz))|\vx\cdot\vz] = 0$.
\end{proof}

    By \Cref{claim:best-fitting-activation-property}, we have $\Evz[\partial_1\ell_h(\vx\cdot\vz,\vx^*\cdot\vz;\vx,\vx^*)|\vx\cdot\vz] = 0$, hence \Cref{assum:ell-logconcave}(i) is satisfied.

    For the second part of \Cref{assum:ell-logconcave}, note  first that since $h'$ is Lipschitz and $\sigma^*$ is $b$-Lipschitz (recall that $\sigma^*\in\cF_{(a,b)}$), the partial derivative  $\partial_1\ell_h(u,v;\vx,\vx^*)$ is also Lipschitz in $v$. Therefore, $\partial_2\partial_1\ell_h(u,v;\vx,\vx^*)$ exists a.e., and we have $\partial_2\partial_1\ell_h(u,v;\vx,\vx^*) = -h''(\sigma_h(u;\vx,\vx^*) - \sigma^*(v))(\sigma^*)'(v)$. Furthermore, since $h$ is strongly convex and smooth, we have $m_h\leq h''\leq L_h$ and in addition, $b\geq (\sigma^*)'\geq 0$, therefore, we conclude that $|\partial_2\partial_1\ell_h(u,v;\vx,\vx^*)|\leq L_h b$. Combining with the fact (\Cref{eq:kernel-positive}) that the kernel $c_{\vv\cdot\vz|\vx\cdot\vz}(t)\geq 0$ and its integral equals $\Var(\vv\cdot\vz|\vx\cdot\vz)<\infty$, we conclude that 
    \begin{align*}
        \E_{\vx \cdot \vz}\Big[\int_\R|\partial_2\partial_1\ell_h(\vx\cdot\vz,\vx\cdot\vz\cos\theta + t\sin\theta;\vx,\vx^*)|c_{\vv\cdot\vz|\vx\cdot\vz}(t)\diff{t}\Big] &\leq L_h b\E_{\vx \cdot\vz}\Big[\int_\R c_{\vv\cdot\vz|\vx\cdot\vz}(t)\diff{t}\Big] \\
        &= L_h b\, \E_{\vx\cdot\vz}[\Var(\vv\cdot\vz|\vx\cdot\vz)] \leq L_h b\E[(\vv\cdot\vz)^2]<\infty. 
    \end{align*}

    Since $h''\geq m_h$ and $(\sigma^*)'\geq 0$, we have $\partial_2\partial_1\ell_h(u,v;\vx,\vx^*)\leq 0$ for any $u,v\in\R$; and in addition, $c_{\vv\cdot\vz|\vx\cdot\vz}(t)\geq 0$ for all $t\in\R$. This implies that:
    \begin{align*}
&\E_{\vx\cdot\vz}\bigg[\int_\R\partial_2\partial_1\ell_h(\vx\cdot\vz,\vx\cdot\vz\cos\theta + t\sin\theta;\vx,\vx^*)c_{\vv\cdot\vz|\vx\cdot\vz}(t)\diff{t}\bigg]\\
\leq\; & -\E_{\vx\cdot\vz}\bigg[m_h \int_\R(\sigma^*)'(\vx\cdot\vz\cos\theta + t\sin\theta)c_{\vv\cdot\vz|\vx\cdot\vz}(t)\diff{t}\bigg].
    \end{align*}

    Let $V = \vv\cdot\vz$ and $X = \vx\cdot\vz$, and let $p_{X, V}$ be their joint density. By the well-behavedness assumption (\Cref{def:well-behaved}), we have that for all $(s, u) \in [-R, R]^2,$ $p_{X, V}(s, u) \geq L$. Let $m_s := \E[V|X = s].$ Then 
    \begin{equation}\label{eq:kernel-lb-under-well-behaved}
        c_{V|X = s}(t) \geq \frac{L}{2p_X(s)}(R - |t|)^2.
    \end{equation}
    (To see this, recall the representations of the conditional kernel from \Cref{eq:kernel-positive}: if $m_s < t$, integrate $(u-t)p_{V|X=s}(u)$ over $u \in [t, R]$; otherwise, integrate $(t-u)p_{V|X=s}(u)$ over $u \in [-R, t]$.)

    Let $H_\theta:=\{(s, t) \in [-R, R]^2: s\cos \theta + t \sin \theta > 0\}.$ Since, by definition, $(\sigma^*)'(u) \geq a > 0$ when $u > 0,$ we get from \eqref{eq:kernel-lb-under-well-behaved} that
\begin{equation}\notag
        \E_X\Big[\int_\R (\sigma^*)'(X\cos\theta + t \sin \theta)c_{V|X}(t)\diff{t}\Big] \geq \frac{a L}{2}\int_{H_\theta} (R - |t|)^2 \diff{s}\diff{t}. 
    \end{equation}
Now observe that $(s, t) \mapsto (-s, -t)$ maps $H_\theta$ to its complement in $[-R, R]^2$ (up to the measure-zero line $s\cos \theta + t \sin \theta = 0$ dividing the two sets). This map keeps $(R - |t|)^2$ unchanged. Thus:
\[
    \int_{H_\theta}(R - |t|)^2 \diff{t} = \frac{1}{2} \int_{[-R, R]^2}(R - |t|)^2\diff{s}\diff{t} = \frac{2R^4}{3}.
    \]
Therefore,
     \begin{equation}\notag
        \E_X\Big[\int_\R (\sigma^*)'(X\cos\theta + t \sin \theta)c_{V|X}(t)\diff{t}\Big] \geq \frac{a LR^4}{3}. 
    \end{equation}
As a result, \Cref{assum:ell-logconcave}(ii) holds with $\mu = \mu^{(2)} = \frac{m_h a LR^4}{3}.$

    It remains to show that the vector field $\vg_\ell(\vx,\vx^*)$ is restricted Lipschitz (\Cref{assum:smoothness-of-regression-loss}(ii)). By definition and applying the Cauchy-Schwarz inequality, we have:
    \begin{align*}
        \|\vg_h(\vx,\vx^*)\|_2 &= \sup_{\vr\in\spd}\Evz[h'(\sigma_h(\vx\cdot\vz;\vx,\vx^*) - \sigma^*(\vx^*\cdot\vz))\vz_{\perp\vx}\cdot\vr]\\
        &\leq \sqrt{\Evz[h'(\sigma_h(\vx\cdot\vz;\vx,\vx^*) - \sigma^*(\vx^*\cdot\vz))^2]\Evz[(\vz_{\perp \vx}\cdot\vr)^2]}.
    \end{align*}
Because $\D$ is isotropic, we have $\Evz[(\vz_{\perp \vx}\cdot\vr)^2] \leq 1.$ Moreover, 
    since $h$ is $L_h$-smooth , and $h'(0) = h(0) = 0$, we have $h(u) \leq \frac{L_h}{2}u^2.$ As $\sigma_h(z;\vx,\vx^*)$ is the best fitting link function, we get 
    \begin{align*}
        \Evz[h'(\sigma_h(\vx\cdot\vz;\vx,\vx^*) - \sigma^*(\vx^*\cdot\vz))^2]&\leq 2L_h\Evz[h(\sigma_h(\vx\cdot\vz;\vx,\vx^*) - \sigma^*(\vx^*\cdot\vz))]\\
        &\leq 2L_h\Evz[h(\sigma^*(\vx\cdot\vz) - \sigma^*(\vx^*\cdot\vz))]\\
        &\leq L_h^2 \Evz[(\sigma^*(\vx\cdot\vz) - \sigma^*(\vx^*\cdot\vz))^2]\\
        &\leq L_h^2 b^2 \norm{\vx - \vx^*}_2^2. 
    \end{align*}
    Hence, $\norm{\vg_h(\vx, \vx^*)} \leq L_h b \norm{\vx - \vx^*}_2,$ meaning that \Cref{assum:smoothness-of-regression-loss}(ii) holds with $L_2 = L_h b.$

    Thus, applying \Cref{lem:regression-loss-satisfy-sharpness-logconcave}, we obtain that $\vg_h(\vx,\vx^*)$ is $(m_h/L_h)(a/b)(R^4 L/6)$-RGA for all $\vx$ such that $\theta(\vx,\vx^*)\in(0,\pi/2)$.
\end{proof}

Noting that the vector fields used in \Cref{thm:well-behaved-sims-are-RGA} need not correspond to the (Riemannian) gradient of the SIM objective $f(\vx)$ in \Cref{prob:glm+sim}, \Cref{thm:well-behaved-sims-are-RGA} highlights the versatility of the RGA framework: the convergence property of \Cref{alg:RGD} is not restricted to a specific gradient field---it is applicable to any vector field satisfying the RGA property.

\subsubsection{GLMs of $(a,b)$-Unbounded Activations under Near-Isotropic Distributions}\label{subsubsec:general-distribution-ab-link-GLM-are-RGA}

For GLM problems where the target link function $\sigma^*$ is a given function in $\cF_{(a,b)}$, we show that RGA holds under any distribution that has a bounded fourth moment and has a margin in the direction of $\vx^*$. The set of such distributions was first introduced in~\cite{WZDD2023}, and here we denote it by $\mathfrak{D}_{\mathrm{nearIso}}$. 

\begin{definition}[Near-Isotropic Distributions~\cite{WZDD2023}]\label{def:near-iso-distr}
    Let $\vx^*$ be a fixed (unknown) unit vector in $\R^d$. Let $\mathfrak{D}_{\mathrm{nearIso}}(B_4,\gamma,\lambda)$ be a class of distributions over $\vz\in\R^d$ such that any distribution  $\D\in\mathfrak{D}_{\mathrm{nearIso}}(B_4,\gamma,\lambda)$ satisfies: 
    \begin{enumerate}
        \item (Margin covariance lower bound) $\Evz[\vz\vz^\top\1\{\vx^*\cdot\vz\geq \gamma\}]\succeq \lambda\mI$;
        \item (Concentration)   $\Evz[(\vr\cdot\vz)^4]\leq B_4$ for any unit vector $\vr$.
    \end{enumerate}
\end{definition}

\begin{remark} 
Unlike $\mathfrak{D}_{\mathrm{well}}$, $\mathfrak{D}_{\mathrm{nearIso}}$ does not require the distribution to have a density, and therefore includes discrete distributions like $\mathrm{Unif}\{-1,0,1\}^{d}$ (see \cite{WZDD2023} and \Cref{claim:lattice-relu-glm-not-geo-convex}), which cannot be included in $\mathfrak{D}_{\mathrm{well}}$. 
The non-regularity of near-isotropic distributions, especially the lack of anti-anti-concentration property, introduces challenges in verifying \Cref{assum:ell-logconcave} and extending the results to SIMs. The main obstacle is that the conditional kernel used in \Cref{assum:ell-logconcave} might be degenerate for near-isotropic distributions. 
For example, consider $\D = \mathrm{Unif}\{-1,0,1\}^{d}$. For almost every $\vx \in \spd,$ the map $\vz \mapsto \vx \cdot \vz$ is injective on the finite set $\{-1,0,1\}^{d}.$ Hence, for every $u$ in the support of $\vx \cdot \vz,$ the set $\cZ_u := \{\vz \in \{-1,0,1\}^{d}: \vx \cdot \vz = u\}$ is singleton. As a consequence, conditional on $\vx \cdot \vz,$ the random variables $\vv\cdot\vz$ (where, as before, $\vv := \vx^*_{\perp \vx}/\norm{\vx^*_{\perp \vx}}$) and $\vx^* \cdot \vz$ are deterministic, leading to the conclusion that $c_{\vv\cdot\vz|\vx\cdot\vz}(t) = 0$ for a.e.\ $t.$ This prevents \Cref{assum:ell-logconcave} from providing a nontrivial curvature certificate for the SIM loss. Moreover, for every such $\vx,$ there exists $\sigma \in \cL_2(\cD_{\vx \cdot \vz})$ such that $\sigma(\vx \cdot \vz) = \sigma^*(\vx^*\cdot\vz)$ on the support of $\D.$ Thus, the SIM loss $f(\vx) = \min_{\sigma\in \cL_2(\D_{\vx \cdot \vz})}\Evz[(\sigma(\vx\cdot\vz) - \sigma^*(\vx^*\cdot\vz))^2]$ equals zero for a.e.\ $\vx.$ \end{remark}
Nevertheless, using the conditions for general vector fields (\Cref{assum:ell-general-vector-field}), we are able to show that various GLM losses satisfy RGA, extending the results from \cite{WZDD2023} that only studied the (Euclidean) gradient alignment property of the \emph{convex surrogate loss} (defined below) on a \emph{Euclidean ball}.

\begin{theorem}\label{thm:general-distribution-ab-link-GLM-are-RGA}
Let $\D\in\mathfrak{D}_{\mathrm{nearIso}}(B_4,\gamma,\lambda)$, where $B_4,\gamma,\lambda$ are positive constants. Let $\sigma\in\cF_{(a,b)}$ be any $(a,b)$-unbounded activation, and fix any $\vx^*\in\spd$. Then the following losses are RGA:
    \begin{enumerate}
        \item The surrogate loss 
        $$f_\sur(\vx)\eqdef\Evz[\ell_\sur(\vx\cdot\vz,\vx^*\cdot\vz)],\; \ell_\sur(u_1,u_2)= \int_0^{u_1}(\sigma(r) - \sigma(u_2))\diff{r}.$$
        In particular, every $\vg(\vx) \in \partial^{C,R}f_\sur(\vx)$ is $\frac{a\lambda}{8b\sqrt{B_4}}$-RGA when \[0<\theta\leq \min\Big\{\arctan\Big(\frac{a\lambda}{4b\sqrt{B_4}}\Big), \arctan\Big(\gamma\sqrt{\frac{\lambda}{4B_4}}\Big)\Big\};\]
        \item The $L_2^2$ loss
        \begin{align*}
            f_\Ltwo(\vx)\eqdef\Evz[\ell_\Ltwo(\vx\cdot\vz,\vx^*\cdot\vz)],\;\ell_\Ltwo(u_1,u_2) = \frac{1}{2}(\sigma(u_1) - \sigma(u_2))^2.
        \end{align*}
        In particular, every $\vg(\vx) \in \partial^{C,R}f_\Ltwo(\vx)$ is $\frac{a^2\lambda}{8b^2\sqrt{B_4}}$-RGA when \[0< \theta\leq \min\Big\{\arctan\Big(\frac{a^2\lambda}{4b^2\sqrt{B_4}}\Big), \arctan\Big(\gamma\sqrt{\frac{\lambda}{4B_4}}\Big)\Big\};\]
        \item The $\alpha$-Huber loss, for constant $\alpha>0$, $h_\hub(u)= (u^2/2)\1\{|u|\leq \alpha\} + \alpha(|u| - \alpha/2)\1\{|u|>\alpha\}$,
        \begin{align*}
            f_\hub(\vx) \eqdef\Evz[\ell_\hub(\vx\cdot\vz,\vx^*\cdot\vz)], \; \ell_\hub(u_1,u_2) = h_\hub(\sigma(u_1) - \sigma(u_2)). 
        \end{align*}
        In particular, every $\vg(\vx) \in \partial^{C,R}f_\hub(\vx)$ is $\frac{a^2\lambda}{8b^2\sqrt{B_4}}$-RGA when 
        \[0<\theta\leq \min\Big\{\arctan\big(\frac{a^2\lambda}{4b^2\sqrt{B_4}}\big), \arctan\Big(\frac{\gamma}{2}\sqrt{\frac{\lambda}{B_4}}\Big), \arcsin\Big(\min\Big\{1, \frac{\alpha}{8b}\sqrt{\frac{\lambda}{B_4}}\Big\}\Big)\Big\}.\]
    \end{enumerate}
\end{theorem}
\begin{proof}
Observe that, unlike in the Gaussian marginal setting, when $\D$ is a discrete distribution and $\sigma$ is nonsmooth, the population objectives $f_\ell(\vx)$, $\ell\in\{\Ltwo,\hub\}$, may only be locally Lipschitz rather than continuously differentiable. 
Let $\ell(\vx\cdot\vz,\vx^*\cdot\vz) = h_\ell(\sigma(\vx\cdot\vz) - \sigma(\vx^*\cdot\vz))$,
where $h_\Ltwo(z) = z^2/2$, and $h_\hub(u)= (u^2/2)\1\{|u|\leq \alpha\} + \alpha(|u| - \alpha/2)\1\{|u|>\alpha\}$. For either $L_2^2$ loss or the Huber loss, $h_\ell$ is continuously differentiable, and hence the non-differentiability of the population objective $f_\ell$ comes from the non-differentiability of the link functions. 

Note that by the Clarke chain rule, it holds $\partial^{C,R}_\vx\ell(\vx\cdot\vz,\vx^*\cdot\vz) \subseteq h'_\ell(\sigma(\vx\cdot\vz) - \sigma(\vx^*\cdot\vz))\partial^C\sigma(\vx\cdot\vz)\vz_{\perp\vx}$.
For a fixed $\vx$, let $s_{\vx}(\vz)$ be a fixed $\D$-measurable selection from the Clarke subdifferential: $s_{\vx}(\vz)\in\partial^C \sigma(u)|_{u = \vx\cdot\vz} = \partial^C \sigma(\vx\cdot\vz)$.
Define the Riemannian correlation vector field (\Cref{def:grad-field}) $\vg_\ell(\vx,\vx^*)$ via: 
\begin{align*}
    &\vg_\ell(\vx,\vx^*;s_{\vx}(\vz)) = \Evz[h_\ell'(\sigma(\vx\cdot\vz) - \sigma(\vx^*\cdot\vz))s_{\vx}(\vz)\vz_{\perp\vx}] = \Evz[\partial_1\ell(\vx\cdot\vz,\vx^*\cdot\vz)\vz_{\perp\vx}]\\
    &\begin{cases}
        \vg_\sur(\vx,\vx^*)  = \Evz[(\sigma(\vx\cdot\vz) - \sigma(\vx^*\cdot\vz))\vz_{\perp\vx}];\\
        \vg_\Ltwo(\vx,\vx^*;s_{\vx}(\vz))  = \Evz[(\sigma(\vx\cdot\vz) - \sigma(\vx^*\cdot\vz))s_{\vx}(\vz)\vz_{\perp\vx}];\\
        \vg_\hub(\vx,\vx^*;s_{\vx}(\vz))  = \Evz[h_\hub'(\sigma(\vx\cdot\vz) - \sigma(\vx^*\cdot\vz))s_{\vx}(\vz)\vz_{\perp\vx}].
    \end{cases}
\end{align*} 
Then, for $\ell = \sur,$ $\nablar f_\ell(\vx) = \vg_\sur(\vx; \vx^*)$, while for $\ell \in \{\Ltwo, \hub\}$ by \Cref{fact:Aumann-expectation} for Aumann expectation, it holds, 
\begin{align*}
    \partial^{C,R} f_\ell(\vx)&\subseteq\Evz[\partial_\vx^{C,R}\ell(\vx\cdot\vz,\vx^*\cdot\vz)] \\
    &\subseteq\{\vg_\ell(\vx,\vx^*;s_{\vx}(\vz)): \text{$s_{\vx}(\vz)\in\partial^C \sigma(\vx\cdot\vz)$ and $s_{\vx}(\vz)$ is $\D$-measurable}\}. \end{align*} 
Therefore, by the inclusion above, it follows that if the RGA property holds for all $\vg_\ell(\vx,\vx^*;s_{\vx}(\vz))$, it will also be valid for all Clarke subdifferentials of $f_\ell(\vx)$.
Thus, it suffices to prove that for any measurable 
choice of $s_{\vx}(\vz)\in\partial^C\sigma(\vx\cdot\vz),$ the vector field 
$\vg_\ell(\vx,\vx^*;s_{\vx}(\vz))$ satisfies \Cref{assum:ell-general-vector-field},
because then applying \Cref{lem:regression-loss-satisfy-sharpness-general-distr} leads to the desired RGA result for $f_\ell$.

In agreement with the notation in \Cref{assum:ell-general-vector-field}, throughout the proof, we define $\vv := (\vx^*)_{\perp\vx}/\|(\vx^*)_{\perp\vx}\|_2$, $\theta := \theta(\vx,\vx^*)$, $\vy(t) := \cos(t\theta)\vx + \sin(t\theta)\vv$ for $t\in[0,1]$. 
    Observe first that for all $\ell\in\{\sur,\Ltwo,\hub\}$, 
since $\sigma$ is Lipschitz, we have that $\vg_\ell(\vx,\vy;s_{\vx}(\vz))$ is Lipschitz in $\vy\in\R^d$ for all measurable choices of $s_{\vx}(\vz)\in\partial^C\sigma(\vx\cdot\vz)$, providing the locally Lipschitz ambient extensions required by \Cref{assum:ell-general-vector-field}. 
Therefore, the Clarke Jacobian matrices $\partial^C_{2}\vg_\ell(\vx,\vy(t);s_{\vx}(\vz))$ are well defined. Fix a measurable choice of $s_{\vx}(\vz)\in\partial^C\sigma(\vx\cdot\vz)$. Then, using \Cref{fact:Aumann-expectation}, Aumann expectation, and the Clarke chain rule again: 
    \begin{align}\label{eq:jacobi-inclusion-1}
        \partial^C_{2}\vg_\ell(\vx,\vy(t);s_{\vx}(\vz)) &\subseteq \Evz[\partial^C_2\partial_1\ell(\vx\cdot\vz,\vy(t)\cdot\vz)\vz_{\perp\vx}\vz^\top]\\
        &\subseteq\begin{cases}
            -\Evz[\partial^C\sigma(\vy(t)\cdot\vz)\vz_{\perp\vx}\vz^\top]& \sur\\
            -\Evz[\partial^C\sigma(\vy(t)\cdot\vz)s_{\vx}(\vz)\vz_{\perp\vx}\vz^\top] & \Ltwo\\
            -\Evz[\partial^C h'_\hub(\sigma(\vx\cdot\vz) - \sigma(\vy(t)\cdot\vz))\partial^C\sigma(\vy(t)\cdot\vz)s_{\vx}(\vz)\vz_{\perp\vx}\vz^\top] & \hub.
        \end{cases}
    \end{align}
Now, for any $t\in[0,1]$, let $s_{\vy(t)}(\vz) \in\partial^C\sigma(\vy(t)\cdot\vz)$, and $h''_\hub(u)\in\partial^C h'_\hub(u)$, where
\[
    \partial^C h'_\hub(u) = \begin{cases}
        \{1\}, &\text{ if } |u|<\alpha,\\
        [0, 1], &\text{ if } |u| = \alpha,\\
        \{0\}, &\text{ if } |u|>\alpha.\\
    \end{cases}
\]
Define: 
\begin{align}\label{eq:nonsmooth-joint-curvature}
    \partial_{2}\vg_\ell(\vx,\vy(t);s_{\vx}(\vz),s_{\vy(t)}(\vz),h''_\ell) = \begin{cases}
            -\Evz[s_{\vy(t)}(\vz)\vz_{\perp\vx}\vz^\top]& \sur\\
            -\Evz[s_{\vy(t)}(\vz)s_{\vx}(\vz)\vz_{\perp\vx}\vz^\top] & \Ltwo\\
            -\Evz[h''_\hub(\sigma(\vx\cdot\vz) - \sigma(\vy(t)\cdot\vz))s_{\vy(t)}(\vz)s_{\vx}(\vz)\vz_{\perp\vx}\vz^\top] & \hub,
        \end{cases}
\end{align}
then according to \Cref{eq:jacobi-inclusion-1}, the Clarke Jacobian $\partial^C_{2}\vg_\ell(\vx,\vy(t);s_{\vx}(\vz))$ becomes a subset of:
$$\partial^C_{2}\vg_\ell(\vx,\vy(t);s_{\vx}(\vz))\subseteq\{\partial_{2}\vg_\ell(\vx,\vy(t);s_{\vx}(\vz),s_{\vy(t)}(\vz),h''_\ell): s_{\vy(t)}(\vz)\in\partial^C\sigma(\vy(t)\cdot\vz), h_\ell''\in\partial^C h'_\ell\}.$$ 
Therefore, to prove that \Cref{assum:ell-general-vector-field} is satisfied for all Clarke Jacobians $\partial^C_{2}\vg_\ell(\vx,\vy(t);s_{\vx}(\vz))$, it suffices to show that $\partial_{2}\vg_\ell(\vx,\vy(t);s_{\vx}(\vz),s_{\vy(t)}(\vz),h''_\ell)$ satisfies \Cref{assum:ell-general-vector-field} for any choice of $s_{\vy(t)}(\vz)$ and $h''_\ell$.

    \textbf{(I): Surrogate Loss. } First, it is easy to see that the stationarity part of \Cref{assum:ell-general-vector-field} is satisfied for $\vg_\sur$: $\vg_{\sur}(\vx,\vx) = \Evz[(\sigma(\vx\cdot\vz) - \sigma(\vx\cdot\vz))\vz_{\perp\vx}] = \vzero$.

    It remains to check the curvature part of  \Cref{assum:ell-general-vector-field}. Note that by the definition of $\partial_{2}\vg_\sur(\vx,\vy(t);s_{\vy(t)}(\vz))$ in \Cref{eq:nonsmooth-joint-curvature},
the quadratic form is bounded  above by:
    \begin{align*}
        \vv^\top \partial_{2}\vg_\sur(\vx,\vy(t);s_{\vy(t)}(\vz))\vv\leq -\Evz[s_{\vy(t)}(\vz)(\vv\cdot\vz)^2]\leq -\E[a(\vv\cdot\vz)^2\1\{\vy(t)\cdot\vz >  0\}],
    \end{align*}
    where in the last inequality we used the fact that for an $(a,b)$-unbounded function $\sigma$ it holds $\sigma'\geq 0$, and $\sigma'(z)\geq a$ when $z > 0$.
    Recall that $\vy(t) = \cos(t\theta)\vx + \sin(t\theta)\vv$. Let $R\eqdef 2\sqrt{B_4/\lambda}$, and recall that we have restricted $\theta\leq \arctan(\gamma/R) = \arctan(\gamma\sqrt{\lambda/(4B_4)})$. Then for any $\vz$ such that $\vz\cdot\vx^*\geq\gamma$ and $\vv\cdot\vz\leq R$, and any $t\in(0,1)$, we have:
    \begin{align*}
        \vy(t)\cdot\vz = \frac{\cos(t\theta)}{\cos\theta}(\vx^* - \sin\theta\vv)\cdot\vz + \sin(t\theta)\vv\cdot\vz. \end{align*}
    When $\vv\cdot\vz<0$, since $t\in(0,1)$ and $\theta\in(0,\pi/2)$ it implies:
    \begin{align*}
        \vy(t)\cdot\vz = \frac{\cos(t\theta)}{\cos\theta}\vx^*\cdot\vz -\bigg(\frac{ \sin\theta\cos(t\theta)}{\cos\theta} - \sin(t\theta)\bigg)\vv\cdot\vz\geq \frac{\cos(t\theta)}{\cos\theta}\gamma - \frac{\sin((1 - t)\theta)}{\cos\theta}\vv\cdot\vz\geq \gamma>0.
    \end{align*}
    On the other hand, when $0\leq \vv\cdot\vz\leq R$, we have
    \begin{align*}
        \vy(t)\cdot\vz = \frac{\cos(t\theta)}{\cos\theta}(\vx^* - \sin\theta\vv)\cdot\vz + \sin(t\theta)\vv\cdot\vz> \vx^*\cdot\vz - \tan\theta \vv\cdot\vz\geq \gamma - \gamma\ = 0.
    \end{align*}
    Therefore we conclude that $\1\{\vy(t)\cdot\vz > 0,\vv\cdot\vz\leq R\}\geq \1\{\vx^*\cdot\vz\geq \gamma, \vv\cdot\vz\leq R\}$.
    Since $(\vv\cdot\vz)^2\geq 0$, it holds
    \begin{align}\label{eq:joint-curvature-bound-1}
        \Evz[(\vv\cdot\vz)^2(\1\{\vy(t)\cdot\vz >  0\} - \1\{\vx^*\cdot\vz\geq \gamma\})] &= \Evz[(\vv\cdot\vz)^2(\1\{\vy(t)\cdot\vz >  0\} - \1\{\vx^*\cdot\vz\geq \gamma\})\1\{\vv\cdot\vz\geq R\}] \nonumber\\
        &\quad + \Evz[(\vv\cdot\vz)^2(\1\{\vy(t)\cdot\vz >  0\} - \1\{\vx^*\cdot\vz\geq \gamma\})\1\{\vv\cdot\vz\leq R\}] \nonumber\\
        &\geq -\Evz[(\vv\cdot\vz)^2|\1\{\vy(t)\cdot\vz >  0\} - \1\{\vx^*\cdot\vz\geq \gamma\}|\1\{\vv\cdot\vz\geq R\}] \nonumber\\
        &\geq -\sqrt{\Evz[(\vv\cdot\vz)^4]\pr[\vv\cdot\vz\geq R]},
    \end{align}
    where the last inequality is due to Cauchy-Schwarz. 
    By Markov inequality, we have $\pr[|\vv\cdot\vz|\geq R]\leq \Evz[|\vv\cdot\vz|^4]/R^4$. Recall that $R = 2\sqrt{B_4/\lambda}$, we then obtain:
    \begin{align*}
        \Evz[(\vv\cdot\vz)^2(\1\{\vy(t)\cdot\vz >  0\} - \1\{\vx^*\cdot\vz\geq \gamma\})]\geq -\frac{B_4}{R^2} = -\lambda/4.
    \end{align*}
    Since by the margin assumption on $\D$ it holds $\E[\vz\vz^\top\1\{\vx^*\cdot\vz\geq \gamma\}]\succeq \lambda\mI$, we conclude that $\E[(\vv\cdot\vz)^2\1\{\vy(t)\cdot\vz\geq 0\}]\geq \lambda/2$, and hence $\vv^\top\partial_{2}\vg_\sur(\vx,\vy(t);s_{\vy(t)}(\vz))\vv\leq -a\lambda/2$.

    For the second part of the curvature condition in \Cref{assum:ell-general-vector-field}, by definition of the operator norm we have:
    \begin{align*}
        \|\partial_{2}\vg_\sur(\vx,\vy(t);s_{\vy(t)}(\vz))\|_{2} &= \sup_{\|\vu\|_2 = \|\vr\|_2 = 1}\Evz[s_{\vy(t)}(\vz)(\vu\cdot\vz_{\perp\vx})(\vr\cdot\vz)]\\
        &\leq \sup_{\|\vu\|_2 = \|\vr\|_2 = 1}b\sqrt{\Evz[(\vu\cdot\vz_{\perp\vx})^2(\vr\cdot\vz)^2]}
        \leq b\sqrt{B_4},
    \end{align*}
    where in the last inequality we used the fact that $\sigma$ is $b$-Lipschitz and Cauchy-Schwarz inequality.
    Thus, to conclude, \Cref{assum:ell-general-vector-field} holds for all Clarke joint curvature matrices $\partial^C_{2}\vg_\sur(\vx,\vy(t);s_{\vx}(\vz))$ with $\mu= a\lambda/2$, $B = b\sqrt{B_4}$.

Therefore, using \Cref{lem:regression-loss-satisfy-sharpness-general-distr} implies that when $\theta\leq \min\Big\{\arctan\big(\frac{a\lambda}{4b\sqrt{B_4}}\big), \arctan\Big(\gamma\sqrt{\frac{\lambda}{4B_4}}\Big)\Big\}$, the vector field $\vg_\sur(\vx,\vx^*;s_{\vx}(\vz))$ is $a\lambda/(8b\sqrt{B_4})$-RGA. Recall that $\partial^{C,R}f_\sur(\vx)\subseteq\{\vg_\sur(\vx,\vx^*;s_{\vx}(\vz)): s_{\vx}(\vz)\in\partial^C\sigma(\vx\cdot\vz)\}$, hence any $\nablar f_\sur(\vx)\in\partial^{C,R}f_\sur(\vx)$ is $a\lambda/(8b\sqrt{B_4})$-RGA.

    \textbf{(II): $L_2^2$ Loss. } For $L_2^2$ loss, it is easy to see that the stationarity requirement in  \Cref{assum:ell-general-vector-field}(i) is satisfied. Therefore, we focus on $\partial_{2}\vg_\Ltwo(\vx,\vy(t);s_{\vx}(\vz),s_{\vy(t)}(\vz))$. Fix any $s_{\vy(t)}(\vz)\in\partial^C \sigma(\vy(t)\cdot\vz)$. For the curvature condition, the quadratic form equals (by \Cref{eq:nonsmooth-joint-curvature})
    \begin{align*}
        \vv^\top\partial_{2}\vg_\Ltwo(\vx,\vy(t);s_{\vx}(\vz),s_{\vy(t)}(\vz))\vv &=-\Evz[s_{\vx}(\vz) s_{\vy(t)}(\vz) (\vv\cdot\vz)^2] \\
        &\leq -a^2\Evz[(\vv\cdot\vz)^2\1\{\vy(t)\cdot\vz >  0, \vx\cdot\vz >  0\}].
    \end{align*}
    Similar to the proof for the surrogate loss, let $R = 2\sqrt{B_4/\lambda}$ and restrict $\theta\leq \arctan(\gamma/R)$. Then for any $\vz$ such that $\vx^*\cdot\vz\geq \gamma$ and $\vv\cdot\vz\leq R$, it holds that for all $t\in(0,1)$:
    \begin{align}\label{eq:indicator-inequality}
        \vx\cdot\vz = \frac{1}{\cos\theta}(\vx^* - \sin\theta\vv)\cdot\vz\geq \frac{\gamma}{\cos\theta} - R\tan\theta >  0,\;\vy(t)\cdot\vz > 0.
    \end{align}
    Therefore, $\1\{\vy(t)\cdot\vz >  0,\vx\cdot\vz >  0\}\1\{\vv\cdot\vz\leq R\}\geq\1\{\vx^*\cdot\vz\geq \gamma\}\1\{\vv\cdot\vz\leq R\}$, and hence, similar to \Cref{eq:joint-curvature-bound-1} we have:
    \begin{align*}
        &\quad \Evz[(\vv\cdot\vz)^2(\1\{\vy(t)\cdot\vz > 0,\vx\cdot\vz > 0\} - \1\{\vx^*\cdot\vz\geq \gamma\})]\\
        &=\Evz[(\vv\cdot\vz)^2(\1\{\vy(t)\cdot\vz > 0,\vx\cdot\vz > 0\} - \1\{\vx^*\cdot\vz\geq \gamma\})\1\{\vv\cdot\vz\leq R\}]\\
        &\quad + \Evz[(\vv\cdot\vz)^2(\1\{\vy(t)\cdot\vz > 0,\vx\cdot\vz > 0\} - \1\{\vx^*\cdot\vz\geq \gamma\})\1\{\vv\cdot\vz\geq R\}]\\
        &\geq -\sqrt{\Evz[(\vv\cdot\vz)^4]\pr[\vv\cdot\vz\geq R]}\geq -\lambda/2.
    \end{align*}
    Combining with the fact that $\Evz[\vz\vz^\top\1\{\vz\cdot\vx^*\geq \gamma\}]\succeq\lambda\mI$, we have
    \begin{align*}
        \vv^\top\partial_{2}\vg_\Ltwo(\vx,\vy(t);s_{\vx}(\vz),s_{\vy(t)}(\vz))\vv\leq -a^2(\Evz[(\vv\cdot\vz)^2\1\{\vx^*\cdot\vz\geq \gamma\}]- \lambda/2)\leq -a^2\lambda/2,
    \end{align*}
    indicating that the first part of the curvature condition in \Cref{assum:ell-general-vector-field} is satisfied with $\mu = a^2\lambda/2$. For the second part of the curvature condition, similar to the proof of the surrogate loss, we have 
    \begin{align*}
        \|\partial_{2}\vg_\Ltwo(\vx,\vy(t);s_{\vx}(\vz),s_{\vy(t)}(\vz))\|_2 = \sup_{\|\vu\|_2= \|\vr\|_2 = 1}\Evz[s_{\vy(t)}(\vz) s_{\vx}(\vz)(\vu\cdot\vz)(\vr\cdot\vz_{\perp\vx})]\leq b^2\sqrt{B_4}.
    \end{align*}

    Thus, we conclude that the $L_2^2$ loss induced vector field $\vg_\Ltwo(\vx,\vx^*;s_{\vx}(\vz),s_{\vy(t)}(\vz))$ satisfies \Cref{assum:ell-general-vector-field} with $\mu = a^2\lambda/2$ and $B = b^2\sqrt{B_4}$. 
Using \Cref{lem:regression-loss-satisfy-sharpness-general-distr}, when $\theta\leq \min\Big\{\arctan\big(\frac{a^2\lambda}{4b^2\sqrt{B_4}}\big), \arctan\Big(\gamma\sqrt{\frac{\lambda}{4B_4}}\Big)\Big\}$, the vector field $\vg_\Ltwo(\vx,\vx^*;s_{\vx}(\vz))$ is $a^2\lambda/(8b^2\sqrt{B_4})$-RGA for all choices of $s_{\vx}(\vz)$. Recall that $\partial^{C,R}f_\Ltwo(\vx)\subseteq\{\vg_\Ltwo(\vx,\vx^*;s_{\vx}(\vz)): s_{\vx}(\vz)\in\partial^C\sigma(\vx\cdot\vz)\}$, hence any $\nablar f_\Ltwo(\vx)\in\partial^{C,R}f_\Ltwo(\vx)$ is $a^2\lambda/(8b^2\sqrt{B_4})$-RGA.

    \textbf{(III): $\alpha$-Huber Loss. } First, we note that the stationarity condition of \Cref{assum:ell-general-vector-field} is satisfied. Let us denote $\Delta_\sigma(\vx,\vy(t),\vz) = \sigma(\vx\cdot\vz) - \sigma(\vy(t)\cdot\vz)$ and let $h_\hub(z) = z^2/2\1\{|z| \leq \alpha\} + \alpha(|z| - \alpha/2)\1\{|z| > \alpha\}$. Then using \Cref{eq:nonsmooth-joint-curvature}
and noting that for any $h''_\hub\in\partial^C h'_\hub$, it holds $h_\hub''(u)\geq 0$, $h_\hub''(u)\geq \1\{|u|<\alpha\}$,
    and $\sigma'\geq 0$, the quadratic form is upper bounded by:
    \begin{align*}
    \vv^\top\partial_{2}\vg_\hub(\vx,\vy(t);s_{\vx}(\vz),s_{\vy(t)}(\vz),h''_\hub)\vv
       =& -\Evz[h''_\hub(\Delta_\sigma(\vx,\vy(t),\vz))s_\vx(\vz)s_{\vy(t)}(\vz) (\vv\cdot\vz)^2]\\\leq & -\E[\1\{|\Delta_\sigma(\vx,\vy(t),\vz)|<\alpha\}s_\vx(\vz)s_{\vy(t)}(\vz) (\vv\cdot\vz)^2] \\
        \leq & -a^2\Evz[\1\{|\sigma(\vx\cdot\vz) - \sigma(\vy(t)\cdot\vz)| < \alpha, \vx\cdot\vz >  0, \vy(t)\cdot\vz >  0\}(\vv\cdot\vz)^2].
    \end{align*}
    Since $\sigma$ is $b$-Lipschitz, we have $\1\{|\sigma(\vx\cdot\vz) - \sigma(\vy(t)\cdot\vz)| < \alpha\}\geq \1\{|\vx\cdot\vz - \vy(t)\cdot\vz| < \alpha/b\}$. 
    Next we need to compare:
    $\1\{\vy(t)\cdot\vz > 0,\vx\cdot\vz > 0\}\1\{|(\vx - \vy(t))\cdot\vz| <  \alpha/b\}$ and $\1\{\vx^*\cdot\vz\geq \gamma\}$. Let $R = 2\sqrt{B_4/\lambda}$. Recall that we have proved in \Cref{eq:indicator-inequality} that when $\theta\leq \arctan(\gamma/R)$, for all $t\in(0,1)$ it holds 
    $$\1\{\vy(t)\cdot\vz > 0,\vx\cdot\vz > 0\}\1\{|\vv\cdot\vz|\leq R\}\geq \1\{\vx^*\cdot\vz\geq \gamma\}\1\{|\vv\cdot\vz|\leq R\}. $$
    Furthermore, by the definition of $\vy(t)$, we know that when $\sin\theta\leq \alpha/(4bR)$, $|\vv\cdot\vz|\leq R$ and $|\vx\cdot\vz|\leq R$, it holds:
    \begin{align*}
        |\vy(t)\cdot\vz - \vx\cdot\vz| &= |(1 - \cos(t\theta))\vx\cdot\vz + \sin(t\theta)\vv\cdot\vz|\leq 2\sin^2((t\theta)/2)|\vx\cdot\vz| + \sin(t\theta)|\vv\cdot\vz|\\
        &\leq 2\sin^2\theta R + \sin\theta R< \alpha/b.
    \end{align*}
    This implies that when $\theta\leq \arcsin(\alpha/(4bR))$, it holds $\1\{|\vy(t)\cdot\vz - \vx\cdot\vz| < \alpha/b\}\1\{|\vx\cdot\vz|\leq R, |\vv\cdot\vz|\leq R\} = \1\{|\vx\cdot\vz|\leq R, |\vv\cdot\vz|\leq R\}$.
    Therefore, we can further bound the expectation form via:
    \begin{align*}
        &\quad \Evz[(\vv\cdot\vz)^2 (\1\{\vx\cdot\vz >  0, \vy(t)\cdot\vz >  0, |\sigma(\vx\cdot\vz) - \sigma(\vy(t)\cdot\vz)| <  \alpha\} - \1\{\vx^*\cdot\vz\geq \gamma\})]\\
        &=\Evz[(\vv\cdot\vz)^2 (\1\{\vx\cdot\vz >  0, \vy(t)\cdot\vz >  0, |\sigma(\vx\cdot\vz) - \sigma(\vy(t)\cdot\vz)| <  \alpha\} - \1\{\vx^*\cdot\vz\geq \gamma\})\1\{|\vv\cdot\vz|\leq R\}]\\
        &\quad +\Evz[(\vv\cdot\vz)^2 (\1\{\vx\cdot\vz >  0, \vy(t)\cdot\vz >  0, |\sigma(\vx\cdot\vz) - \sigma(\vy(t)\cdot\vz)| <  \alpha\} - \1\{\vx^*\cdot\vz\geq \gamma\})\1\{|\vv\cdot\vz| > R\}]\\
        &\geq \Evz[(\vv\cdot\vz)^2 (\1\{\vx\cdot\vz >  0, \vy(t)\cdot\vz >  0, |\sigma(\vx\cdot\vz) - \sigma(\vy(t)\cdot\vz)| <  \alpha\} - \1\{\vx^*\cdot\vz\geq \gamma\})\1\{|\vv\cdot\vz|\leq R, |\vx\cdot\vz|\leq R\}]\\
        &\quad - \Evz[(\vv\cdot\vz)^2 |\1\{\vx\cdot\vz >  0, \vy(t)\cdot\vz >  0, |\sigma(\vx\cdot\vz) - \sigma(\vy(t)\cdot\vz)| < \alpha\} - \1\{\vx^*\cdot\vz\geq \gamma\}| \1\{|\vv\cdot\vz|\leq R, |\vx\cdot\vz|> R\}]\\
        &\quad - \Evz[(\vv\cdot\vz)^2 |\1\{\vx\cdot\vz >  0, \vy(t)\cdot\vz >  0, |\sigma(\vx\cdot\vz) - \sigma(\vy(t)\cdot\vz)| < \alpha\} - \1\{\vx^*\cdot\vz\geq \gamma\}|\1\{|\vv\cdot\vz| > R\}]\\
        &\geq -\Evz[(\vv\cdot\vz)^2(\1\{|\vx\cdot\vz|>R\} + \1\{|\vv\cdot\vz|>R\})]\\
        &\overset{(i)}{\geq} -\sqrt{\Evz[(\vv\cdot\vz)^4]\pr[|\vv\cdot\vz|>R]}-\sqrt{\Evz[(\vv\cdot\vz)^4]\pr[|\vx\cdot\vz|>R]}\overset{(ii)}{\geq} -\lambda/2,
    \end{align*}
    where in $(i)$ we applied Cauchy-Schwarz and in $(ii)$ we applied Markov inequality.

    Thus, as $\Evz[\vz\vz^\top\1\{\vx^*\cdot\vz\geq \gamma\}]\succeq\lambda\mI$,  when $R = \sqrt{\frac{4 B_4}{\lambda}}$ and $\theta\leq\min\Big\{\arcsin\big(\frac{\alpha}{4bR}\big), \arctan\big(\frac{\gamma}{R}\big)\Big\}$, the quadratic form is bounded by:
    \begin{align*}
        \vv^\top\partial_{2}\vg_\hub(\vx,\vy(t);s_{\vx}(\vz),s_{\vy(t)}(\vz),h''_\hub)\vv\leq -a^2(\Evz[(\vv\cdot\vz)^2\1\{\vx^*\cdot\vz\geq\gamma\}] - \lambda/2)\leq -a^2\lambda/2.
    \end{align*}

    For the second part in the curvature condition in \Cref{assum:ell-general-vector-field}, recalling that $h_\hub''(z)\in[0,1]$, by definition we have:
    \begin{align*}
        &\|\partial_{2}\vg_{\hub}(\vx,\vy(t);s_{\vx}(\vz),s_{\vy(t)}(\vz),h''_\hub)\|_2 \\
        =\;& \sup_{\|\vu\|_2= \|\vr\|_2 = 1}\Evz[h''_\hub(\Delta_\sigma(\vx,\vy(t),\vz))s_{\vy(t)}(\vz) s_{\vx}(\vz)(\vu\cdot\vz)(\vr\cdot\vz_{\perp\vx})]\\
        \leq\; & \sup_{\|\vu\|_2= \|\vr\|_2 = 1}\Evz[s_{\vy(t)}(\vz) s_{\vx}(\vz)(\vu\cdot\vz)(\vr\cdot\vz_{\perp\vx})] \leq b^2\sqrt{B_4}.
    \end{align*}
    Thus, we conclude that the $\alpha$-Huber loss induced vector field $\vg_{\hub}(\vx,\vy(t);s_{\vx}(\vz))$ satisfies \Cref{assum:ell-general-vector-field} with $\mu = a^2\lambda/2$ and $B = b^2\sqrt{B_4}$. 
Applying \Cref{lem:regression-loss-satisfy-sharpness-general-distr}, when \[\theta\leq \min\Big\{\arctan\big(\frac{a^2\lambda}{4b^2\sqrt{B_4}}\Big),\, \arctan\Big(\frac{\gamma}{2}\sqrt{\frac{\lambda}{B_4}}\Big),\, \arcsin\Big(\min\{1, \frac{\alpha}{8b}\sqrt{\frac{\lambda}{B_4}}\}\Big)\Big\},\] we have that the vector field $\vg_{\hub}(\vx,\vx^*;s_{\vx}(\vz))$ is $a^2\lambda/(8b^2\sqrt{B_4})$-RGA. Recalling that $\partial^{C,R}f_\hub(\vx)\subseteq\{\vg_\hub(\vx,\vx^*;s_{\vx}(\vz)): s_{\vx}(\vz)\in\partial^C\sigma(\vx\cdot\vz)\}$, we conclude that $\partial^{C,R}f_\hub(\vx)$ is $a^2\lambda/(8b^2\sqrt{B_4})$-RGA.
\end{proof}

\paragraph{RGA is weaker than geodesic strong convexity} 

The joint curvature in \Cref{assum:ell-general-vector-field}  is a much weaker condition compared to geodesic (strong) convexity. In particular, RGA is implied by geodesic strong convexity (\Cref{claim:str-cvx-implies-RGA}), while, as we argue below, RGA may hold on regions where the function is not even geodesically convex. 
Consider a simple GLM regression problem, where $\D$ is the uniform distribution on the lattice $\{0,\pm 1\}^d$, and $\sigma$ is the ReLU link function. 
In \Cref{claim:lattice-relu-glm-not-geo-convex} below, we show that for this regression problem, the $L_2^2$ loss $f_\Ltwo(\vx)$ is not geodesically convex on a spherical cap centered at $\vx^*$ with radius as small as $O(1/\sqrt{d})$, while in sharp contrast, \Cref{thm:general-distribution-ab-link-GLM-are-RGA}, which uses the joint curvature condition (\Cref{assum:ell-general-vector-field}), implies that  $f_\Ltwo$ is $\Omega(1)$-RGA on a large spherical cap of radius $\Omega(1)$.

\begin{lemma}\label{claim:lattice-relu-glm-not-geo-convex}
    Let $\D = \mathrm{Unif}\{-1,0,1\}^d$ for $d \geq 2$ and let $\vx^* = \ve_1$. Let $\sigma(z) = \mathrm{ReLU}(z) = \max\{z,0\}$. Denote the spherical cap centered at $\vx^*$ with radius $r$ by $\cS_{\vx^*}(r)\eqdef\{\vx\in\spd:\theta(\vx,\vx^*)\leq r\}$. Then, $f_\Ltwo(\vx)$ cannot be geodesically convex on $\cS_{\vx^*}(2/\sqrt{d})$, whereas it is $\sqrt{3}/72$-RGA on $\cS_{\vx^*}(\arctan(\sqrt{3}/36))$.
\end{lemma}
\begin{proof}
    We first show that $\D\in\mathfrak{D_{\mathrm{nearIso}}}(4/3,1,2/9)$. Let $z_i \sim\mathrm{Unif}\{-1,0,1\}$, independent for $i=1,\dots,d$. First, for the 4th-order moment, for any $\vr\in\spd$, we have:
    \begin{align*}
        \Evz[(\vr\cdot\vz)^4] = \sum_{i=1}^d r_i^4 \E[z_i^4] + 6\sum_{i\leq j}r_i^2r_j^2\E[z_i^2]\E[\z_j^2] = \frac{2}{3}\sum_{i=1}^d r_i^4 + \frac{8}{3}\sum_{i< j}r_i^2r_j^2.
    \end{align*}
    Since $1 = (\sum_i r_i^2)^2$, we have $2\sum_{i< j}r_i^2r_j^2 = 1 - \sum_i r_i^4$, therefore, we have $\Evz[(\vr\cdot\vz)^4] = 4/3 - (2/3)\sum_{i=1}^d r_i^4\leq 4/3$. Therefore, we have $B_4\leq 4/3$. Furthermore, note that since each coordinate of $\vz$ is independent, it holds $\Evz[z_i z_j\1\{z_1\geq 1\}] = 0$ for any $i\neq j$. In addition, direct calculation yields $\E[z_1^2\1\{z_1\geq 1\}]= 1/3$ and $\E[z_i^2\1\{z_1\geq 1\}] = \E[z_i^2]\pr[z_1\geq 1] = 2/9$ when $i\neq 1$. Therefore, we conclude that $\Evz[\vz\vz^\top\1\{\vx^*\cdot\vz\geq 1\}]\succeq (2/9)\mI$. In summary, $\D\in\mathfrak{D}_{\mathrm{nearIso}}(4/3,1,2/9)$. Since $\sigma$ is $(1,1)$-unbounded, applying \Cref{thm:general-distribution-ab-link-GLM-are-RGA} we immediately obtain that $f_\Ltwo(\vx)$ is $\sqrt{3}/72$-RGA in $\cS_{\vx^*}(\arctan(\sqrt{3}/36))$.

    On the other hand, consider $\vx(\theta) = \cos\theta\vx^* + \sin\theta\vv$, where $\vv = \sum_{i\geq 2}\ve_i/\sqrt{d-1}$, $\vv\perp\vx^*$, and $\theta\in(0, \theta_1)$, $\theta_1\eqdef 2/\sqrt{d}$. 
    In other words, $\vx(\theta)$ is the geodesic from $\vx^*$ to $\vx(\theta_1)$: $\cG_{\vx^*,\vx(\theta_1)}(t), t = \theta/\theta_1\in[0,1]$. Let $\theta_0 \eqdef \arcsin(1/\sqrt{d})$. 
    For atoms $\vz\in\{-1,0,1\}^d$, we have
    \begin{align*}
        q(\theta_0) = f_\Ltwo(\vx(\theta_0)) &= \frac{3^{-d}}{2}\sum_{\vz\in\{-1,0,1\}^d} (\sigma(\vx(\theta_0)\cdot\vz) - [z_1]_+)^2 \\
        &= \frac{3^{-d}}{2}\sum_{\vz\in\{-1,0,1\}^d}\bigg(\bigg[\sqrt{\frac{d-1}{d}}z_1 + \frac{1}{\sqrt{(d-1)d}}\sum_{i\geq 2}z_i\bigg]_+ - [z_1]_+\bigg)^2. 
    \end{align*}
    If $f_\Ltwo$ were geodesically convex on $\cS_{\vx^*}(2/\sqrt{d})$, then by definition $f_\Ltwo(\vx(\theta))$ would need to be convex on the geodesic from $\vx^*$ to $\vx(\theta_1)$, therefore, it must hold that $q'(\theta_0^+)\geq q'(\theta_0^-)$, due to convexity.
    Let $\vz$ be any atom, and denote $p_\vz(\theta)\eqdef ([\vx(\theta)\cdot\vz]_+ - [z_1]_+)^2$. 
    Consider first $z_1\leq 0$. Then since $p_\vz(\theta) = [\vx(\theta)\cdot\vz]_+^2$ is a smooth function, there is no jump in $p_\vz'(\theta)$ at $\theta_0$.
    Then consider $z_1>0$. It is easy to see that the only atom such that $\vx(\theta_0)\cdot\vz = 0$ is $\vz_0 = (1,-1,\dots, -1)$. Furthermore at $\theta_0^+$, we have $\vx(\theta_0^+)\cdot\vz<0$ and $\vx(\theta_0^-)\cdot\vz>0$. This implies that $p_\vz(\theta_0^-) = (\vx(\theta_0^-)\cdot\vz - 1)^2$ and $p_\vz(\theta_0^+) = 1$, hence
    \begin{align*}
    p_{\vz_0}'(\theta_0^+) \to 0, \;     p_{\vz_0}'(\theta_0^-) \to 2(0 - 1)(-\sin\theta_0 - \sqrt{d-1}\cos\theta_0) = 2\sqrt{d}.
    \end{align*}
    For any other atoms such that $z_1>0$, we have $p_\vz(\theta) = (\vx(\theta)\cdot\vz - z_1)^2$ in the neighborhood of $\theta_0$, which is also a smooth function around $\theta_0$ and $p'_{\vz}(\theta_0^+) = p'_{\vz}(\theta_0^-)$.
    This implies that 
    $$q'(\theta_0^+) - q'(\theta_0^-) = \frac{3^{-d}}{2}(p_{\vz_0}'(\theta_0^+) - p_{\vz_0}'(\theta_0^-)) + \frac{3^{-d}}{2}\sum_{\vz\neq \vz_0}(p_{\vz}'(\theta_0^+) - p_{\vz}'(\theta_0^-)) = -\sqrt{d}/3^d<0.$$
    This implies $f_\Ltwo(\vx)$ cannot be geodesically convex on $\cS_{\vx^*}(2/\sqrt{d})$.
\end{proof}

\section{Discussion and Future Directions}

In this work, we introduced Riemannian Gradient Alignment (RGA) as a
directional local error bound for optimization over the Euclidean sphere.
RGA directly controls the alignment between a tangent vector field and the
geodesic direction toward a target vector. Unlike standard conditions based
on convexity or smoothness, RGA applies to nonsmooth objectives and, more
generally, to tangent vector fields that need not be gradients of any
objective. We showed that a diminishing-step Riemannian Gradient Descent method converges linearly when RGA holds with order $r=1$, and at rate 
\(
    O\left(k^{-1/(2r-2)}\right)
\)
when $r>1$. 

Our main technical contribution is a collection of derivative-based conditions that make the RGA property verifiable and characterize the landscape of objective functions that can be efficiently minimized on the sphere.  
The examples in \Cref{sec:examples} illustrate the scope of this approach, which unifies several classes of learning and estimation problems under the same RGA-based principle. These examples include objectives that are nonsmooth, non-geodesically convex, or defined under discrete distributions. They also demonstrate the usefulness of analyzing vector fields other than the gradient of the
associated objective. Finally, the provided lasso-regularized spiked sparse-PCA example satisfies RGA with order $r=2$, while no order $r<2$ can hold uniformly near the target. This shows that the higher-order RGA regime is intrinsic to natural problems and not merely an artifact of the definition.

We now discuss some natural directions for future research that we hope our work will stimulate. 

\paragraph{Adaptive and stochastic algorithms.}
The stepsize schedules analyzed in this work use the RGA parameters $\mu$ and $r$. It would be desirable to obtain comparable convergence rates using a parameter-free method, for example through a line search, restarts, or an adaptive estimate of the relevant angular scale.
Moreover, our analysis is primarily formulated in terms of population
or oracle vector fields. An important next step is to determine when
empirical or stochastic estimates inherit RGA uniformly over a spherical
cap, and to develop finite-sample and stochastic-oracle convergence
guarantees. This requires understanding the stability of both the RGA
inequality and the proposed curvature certificates under sampling error.

\paragraph{Characterizing joint curvature and higher-order RGA.}
Our conditions identify negative joint curvature as a broadly applicable
sufficient mechanism for RGA, but they are not intended as necessary
conditions. It would be useful to develop partial converses or alternative
certificates that more sharply characterize when RGA holds. A related
question is whether the RGA order $r$ can be inferred from the rate at
which the relevant curvature or vector-field norm degenerates near the
target. The sparse-PCA example establishes the necessity of $r=2$ in one
natural setting. Finding other intrinsic higher-order examples and
determining which values of $r$ arise in statistical and learning
problems may lead to a more refined classification of benign
nonconvex landscapes.

\paragraph{Extensions beyond the sphere.}
The unit sphere provides a particularly transparent geometry: relative
to a target $\vx^*$, there is a distinguished tangent direction
$(\vx^*)_{\perp \vx}$, and the relevant geodesics have an explicit
two-dimensional representation. A natural question is whether analogous 
alignment conditions and curvature certificates can be developed on 
other manifolds, particularly the Stiefel and Grassmann manifolds. 
Extending the present theory would require an appropriate matrix-valued 
notion of alignment, together with a way to transport and compare vector 
fields along geodesics. It would also be interesting to determine whether 
the higher-order convergence theory for $r>1$ extends to these settings.

\paragraph{Designing vector fields and enlarging the basin of attraction.} 
Because RGA is a property of a vector field rather than necessarily of  the gradient of a prescribed objective, the framework raises the possibility of algorithmically designing a ``useful'' field. One may ask  how to select a vector field for a given problem so as to maximize the RGA parameter or enlarge the region on which RGA holds. This question is especially relevant for semi-parametric problems, where different best-fitting-link objectives  or correlation fields may have very different optimization geometry. For instance, the vector fields chosen for Gaussian GLMs and Gaussian SIMs in \cite{zarifis2025robustly,wang2026robustly} were obtained by  maximizing their correlation with target $\vx^*$ over a specific class of vector fields capturing the Riemannian gradient fields of the $L_2^2$ and the surrogate loss as special cases.  More generally, combining the local RGA guarantees established here  with initialization procedures  would give  end-to-end algorithms that reach and then remain within the certified region.

\section*{Acknowledgements}

This work was supported in part by the Air Force Office of Scientific Research under award number FA9550-24-1-0076, NSF CAREER Award CCF-2440563, and NSF MFAI Award DMS-2502282. Any opinions, findings, conclusions, or recommendations expressed in this material are those of the authors and do not necessarily reflect the views of the U.S.\ Department of Defense.

Part of this work was completed while Nikos Zarifis was a graduate student at UW-Madison.

\newpage

\bibliographystyle{alpha}
\bibliography{mydb}

\appendix
\section*{Appendix}
\section{Omitted Proofs of Auxiliary Facts}\label{appx:omitted-facts}
\subsection{Omitted Proof from \Cref{subsec:glm-sim-Gaussian}}
We restate and prove \Cref{fct:Gaussian-smoothing}.
\factgaussiansmoothing*
{
\Cref{fct:Gaussian-smoothing} follows from properties of Hermite polynomials and Ornstein-Uhlenbeck operators (see, e.g., section 11 in~\cite{Don14}). Here we provide a self-contained proof for completeness. 
}
\begin{proof}
    By the definition of $\Tr_\rho\sigma(u)$, after changing integration variables, we have:
    \begin{align*}
        \Tr_\rho\sigma(u) &= \int_\R\sigma(\rho u + \sqrt{1-\rho^2}z)p_\calN(z)\diff{z} =\int_\R\sigma(y)\frac{1}{\sqrt{1-\rho^2}}p_\calN\bigg(\frac{y-\rho u}{\sqrt{1-\rho^2}}\bigg)\diff{y}\\
        &=\int_\R\sigma(y)c(\rho,u,y)p_\calN(y)\diff{y}, \;\text{where }\;c(\rho,u,y) = \frac{1}{\sqrt{(1-\rho^2)}}\exp\bigg(-\frac{(y-\rho u)^2}{2(1-\rho^2)} + \frac{y^2}{2}\bigg).
    \end{align*}
    Note that $c(\rho,u,y)$ is an infinitely smooth kernel for any $u$ and any $\rho\in(0,1)$. 
    Therefore, $(\rho,u)\to\Tr_\rho\sigma(u)$ is infinitely continuously differentiable for any $u\in\R$ and $\rho\in(0,1)$, by the fact of smooth convolution.
    Note that by the differentiation property of smooth convolution, it holds $\partial_\rho^a\partial_u^b \Tr_\rho\sigma(u) = \int_\R \sigma(y)\partial_\rho^a\partial_u^b c(\rho,u,y)p_\calN(y)\diff{y}.$
    Furthermore, for any compact set $K\subset(0,1)$ and any $\rho\in K$, it holds $\partial_\rho^a\partial_u^b c(\rho,u,y) = p(\rho,u,y)c(\rho,u,y)$ where $p(\rho,u,y)$ is a polynomial of $\rho,u,y$ with bounded coefficients. Therefore, there exists a constant $C_{K,a,b}$ such that 
    \begin{align*}
        \sup_{\rho\in K}\iint_\R (\partial_\rho^a\partial_u^b c(\rho,u,y))^2 p_\calN(y) p_\calN(u)\diff{y}\diff{u}\leq C_{K,a,b}.
    \end{align*}
    Hence, by Cauchy-Schwarz, we have:
    \begin{align*}
        \int_\R (\partial_\rho^a\partial_u^b \Tr_\rho\sigma(u))^2 p_\calN(u)\diff{u}&\leq \int_\R \bigg(\int_\R\sigma(y)^2 p_\calN(y)\diff{y}\int_\R(\partial_\rho^a\partial_u^b c(\rho,u,y))^2 p_\calN(y)\bigg)p_\calN(u)\diff{u}\\
       &\leq \|\sigma\|_{L_2(\D)}^2\int_\R \int_\R(\partial_\rho^a\partial_u^b c(\rho,u,y))^2 p_\calN(y)p_\calN(u)\diff{y}\diff{u}\leq \|\sigma\|_{L_2(\D)}^2 C_{\rho,a,b}.
    \end{align*}
    Therefore, $\partial_\rho^a\partial_u^b \Tr_\rho\sigma(u)\in L_2(\D)$. 
    
    In addition, fixing any $\rho_0\in (0,1)$ and a compact set $K\subset(0,1)$ such that $\rho_0$ is in the interior of $K$, then for any small $h\in(0,1)$ such that $\rho_0 + h\in K$, since $\partial_\rho^a\partial_u^b c(\rho,u,y)$ is a smooth function in $\rho$, there exists $0\leq h''\leq h'\leq h$ such that: 
    \begin{align*}
        &\quad \iint_\R \bigg(\frac{\partial_\rho^a\partial_u^b c(\rho_0 + h,u,y) - \partial_\rho^a\partial_u^b c(\rho_0,u,y)}{h} - \partial_\rho^{a+1}\partial_u^b c(\rho_0,u,y)\bigg)^2p_{\calN}(y) p_\calN(u)\diff{y}\diff{u}\\
        &=\iint_\R  \bigg(\partial_\rho^{a+1}\partial_u^b c(\rho_0 + h',u,y) - \partial_\rho^{a+1}\partial_u^b c(\rho_0,u,y)\bigg)^2p_{\calN}(y) p_\calN(u)\diff{y}\diff{u}\\
        &=(h'')^2 \iint_\R  \bigg(\partial_\rho^{a+2}\partial_u^b c(\rho_0 + h'',u,y)\bigg)^2 p_{\calN}(y) p_\calN(u)\diff{y}\diff{u}\leq (h'')^2 C_{K,a+2,b} \to 0,\;\text{when}\; h\to 0.
    \end{align*}
    This implies
    \begin{align*}
        &\quad \int_\R \bigg(\frac{\partial_\rho^a\partial_u^b \Tr_{\rho_0 + h}\sigma(u) -\partial_\rho^a\partial_u^b \Tr_{\rho_0}\sigma(u)}{h} - \partial_\rho^{a+1}\partial_u^b \Tr_{\rho_0}\sigma(u)\bigg)^2 p_\calN(u)\diff{u}\\
        &=\int_\R \bigg(\int_\R\sigma(y)\bigg(\frac{\partial_\rho^a\partial_u^b c(\rho_0+h,u,y) -\partial_\rho^a\partial_u^b c(\rho_0,u,y)}{h} - \partial_\rho^{a+1}\partial_u^b c(\rho_0,u,y)\bigg)p_\calN(y)\diff{y}\bigg)^2 p_\calN(u)\diff{u}\\
        &\leq \|\sigma\|_{L_2(\D)}^2  \iint_\R \bigg(\frac{\partial_\rho^a\partial_u^b c(\rho_0 + h,u,y) - \partial_\rho^a\partial_u^b c(\rho_0,u,y)}{h} - \partial_\rho^{a+1}\partial_u^b c(\rho_0,u,y)\bigg)^2p_{\calN}(y) p_\calN(u)\diff{y}\diff{u}\to 0.
    \end{align*}
    Similarly we can show that ${(\partial_\rho^a\partial_u^b \Tr_{\rho_0 + h}\sigma(u_0 + r) -\partial_\rho^a\partial_u^b \Tr_{\rho_0}\sigma(u_0))}/{hr}\to \partial_\rho^{a+1}\partial_u^{b+1} \Tr_{\rho_0}\sigma(u_0)$ is in $L_2$.
    Now let $F(\vx)\eqdef(\Tr_{\vx\cdot\vx^*}\sigma(\vx\cdot\vz) - \sigma(\vx^*\cdot\vz))^2$ and let $\vx_t = \cos(t\beta)\vx + \sin(t\beta)\vr$ be any geodesic in a small neighborhood of $\vx$ for some $\vr\perp\vx$, $\|\vr\|_2 = 1$ and some small $\beta\in(0,\pi/2)$, we need to prove $(F(\vx_t)- F(\vx))/t\to\nablar F(\vx)[\vr]$ in $L_1$. Since both $(\Tr_\rho\sigma)'(u)$ and $({\diff{}}/{\diff{\rho}})\Tr_\rho\sigma(u)$ are smooth functions in both $\rho$ and $u$, we know that $(\mathrm{T}_{\vx_t\cdot\vx^*}\sigma)'(\vx_t\cdot\vz)$ and $({\diff{}}/{\diff{\rho}})\mathrm{T}_{\vx_t\cdot\vx^*}\sigma(\vx_t\cdot\vz)$ are smooth functions in $t$, and hence by the mean value theorem, there exists some $\tau\in(0,t)$ such that
    \begin{align*}
        (F(\vx_t)- F(\vx))/t = 2(\Tr_{\vx_\tau\cdot\vx^*}\sigma(\vx_\tau\cdot\vz) - \sigma(\vx^*\cdot\vz))\bigg(\partial_u\mathrm{T}_{\vx_\tau\cdot\vx^*}\sigma(\vx_\tau\cdot\vz)\vz_{\perp\vx_\tau} + \partial_\rho\mathrm{T}_{\vx_\tau\cdot\vx^*}\sigma(\vx_\tau\cdot\vz)\vx^*_{\perp\vx_\tau}\bigg)\cdot\vr.
    \end{align*}
    Then, since every term converges in $L_2$, we can then conclude that $(F(\vx_t)- F(\vx))/t\to\nablar F(\vx)[\vr]$ in $L_1$, implying that $\nablar f(\vx) = \Evz[\nablar F(\vx)]$.
\end{proof}

\subsection{Omitted Proof from \Cref{sec:spiked-sparse-pca}}
We restate and prove \Cref{lem:spiked-sparse-pca-cannot-have-lower-order}.
\spikePCAMustBeOrderTwo*
\begin{proof}
    Our goal is to construct a path $\vx(t)\to\vx^*$ such that when $t$ is small, $-\vg(\vx(t))\cdot\vx^*\asymp t^2$ while $\|\vg(\vx(t))\|_2 = \Theta(1)$. Then when $t\to 0$, this would imply that $-\vg(\vx(t))\cdot\vx^*/(\|\vg(\vx(t))\|_2\sin^r\theta(\vx(t),\vx^*))\to 0$. 
    The idea is to consider a point $\vx$ that has a non-zero entry outside of the support of $\vx^*$, i.e., there exists an index $j\notin \cI$ such that $x_j\neq 0$, then consider the geodesic $\cG_{\vx,\vx^*}(t), t\in[0,1]$.
    On this geodesic, when $\cG_{\vx,\vx^*}(t)$ approaches $\vx^*$ while $t\to 0$, the signal from $\vx^*$ is of second order, while the gradient vector is always large due to the lasso penalization on the outlier entry $j$.  
   
   Recall that under the setting of \Cref{prob:spiked-sparse-pca}, $x^*_i\neq 0$ for $i\in\cI$, $x^*_i = 0$ for $i\notin \cI$, and $|\cI| = s$. 
   Let $\vy\in\spd$ be a vector supported on $\cI$ (in other words, $y_j = 0$ for all $j\notin \cI$) such that $\vy\perp\vx^*$. Such a vector exists because $s = |\cI|\geq 2$ by our assumption, indicating that the subspace supported on $\cI$ has dimension at least 2. Then pick any $j\notin \cI$ (such an index exists since we assumed $s<d$). Let $\vu(t) = \sqrt{1 - t^2}\vy + t \ve_j$, and construct $\vx(t) = \cos(t)\vx^* + \sin(t)\vu(t)$, $t\in[0,1]$.
   Note that since both $\vy$ and $\ve_j$ are orthogonal to $\vx^*$, the vector $\vu(t)$ is orthogonal to $\vx^*$. In addition, since $\vy\perp\ve_j$, $\vu(t)$ has unit norm.

   Pick any $\vg(\vx(t))\in\partial^{C,R}f_\lambda(\vx(t))$. In \Cref{lem:spiked-pca-order-2-rga} we showed that by the definition of the Clarke subdifferential and the definition of $f_\lambda$, $\vg(\vx(t))$ equals
   \begin{align*}
       \vg(\vx(t)) = -2\beta(\vx(t)\cdot\vx^*)\vx^*_{\perp\vx(t)} + \lambda \vz(t)_{\perp\vx(t)},
   \end{align*}
   where $z_i(t) = \sign(x_i(t))$ when $x_i(t)\neq 0$ and $z_i(t)\in[-1,1]$ otherwise. First, observe that any choice of $\vg(\vx(t))$ has a large component along the $\ve_j$ direction. Noting that $\ve_j\perp\vx^*$, $\vz(t)\cdot\vx(t) = \|\vx(t)\|_1$, $\theta(\vx(t),\vx^*) = t$ and $z_j(t) = 1$ since $x_j(t) = \sin(t)t>0$, we have
   \begin{align*}
       \vg(\vx(t))\cdot\ve_j &= -2\beta \cos(\theta(\vx(t),\vx^*))\vx^{*\top}(\mI - \vx(t)\vx(t)^\top)\ve_j + \lambda \vz(t)^\top(\mI - \vx(t)\vx(t)^\top)\ve_j\\
       &=2\beta\cos^2(\theta(\vx(t),\vx^*)) \sin(t)t + \lambda(\vz(t)\cdot\e_j - (\vz(t)\cdot\vx(t))(\vx(t)\cdot\ve_j))\\
       & = 2\beta\cos^2(t) \sin(t)t + \lambda(1 - \|\vx(t)\|_1\sin(t)t)\to\lambda \text{ when } t\to 0.
   \end{align*}
   Therefore, we conclude that $\|\vg(\vx(t))\|_2\geq \lambda/2$ when $t\to 0$.

   On the other hand, the inner product between $\vg(\vx(t))$ and $\vx^*$ is:
   \begin{align*}
       \vg(\vx(t))\cdot\vx^* &= -2\beta \cos(t)\sin^2(t) + \lambda\vz(t)^\top(\mI - \vx(t)\vx(t)^\top)\vx^*\\
       &= -2\beta \cos(t)\sin^2(t) + \lambda((\vz(t)\cdot\vx^*) - \|\vx(t)\|_1\cos(t)).
   \end{align*}
   Now when $t$ is small, $\vx(t)$ is nearby $\vx^*$, therefore for $i\in\cI$ it will hold $\sign(x_i(t)) = \sign(x^*_i)$. This implies that $\vz(t)\cdot\vx^* = \|\vx^*\|_1 = \sqrt{s}$. Furthermore, recall that in \Cref{lem:spiked-pca-order-2-rga} we proved that $\|\vx(t)\|_1\geq \sqrt{s}\cos(\theta(\vx(t),\vx^*)) = \sqrt{s}\cos(t)$. 
   In addition, by the definition of $\vx(t)$, and $\vu(t)$, when $t$ is small we have $\|\vx(t)\|_1 = \|\cos(t)\vx^* + \sin(t)\vu(t)\|_1\leq \cos(t)\sqrt{s} + t\sin(t)$.
   Therefore, when $t$ is sufficiently small, it holds that
   \begin{align*}
       |\vg(\vx(t))\cdot\vx^*|\leq 2\beta \cos(t)\sin^2(t) + \lambda\sqrt{s}\sin^2(t).
   \end{align*}
   Combining with the lower bound on $\|\vg(\vx(t))\|_2$ yields
   \begin{align*}
        \frac{\vg(\vx(t))\cdot\vx^*}{\|\vg(\vx(t))\|_2\sin^r\theta(\vx(t),\vx^*)}\to 0 \text{ when } t\to 0.
    \end{align*}
    This completes the proof.
\end{proof}

\section{Geodesic Strong Convexity Implies RGA}\label{sec:tr-cvx-implies-RGA}
\begin{lemma}\label{claim:str-cvx-implies-RGA}
    Let $\cS\subset\spd$ be a geodesically convex set, and let $f:\cS\to\R$ be differentiable and $\nu$-strongly geodesically convex on $\cS$. Suppose $\vx^*$ is the minimizer of $f$ over $\cS$, and define $\cS' = \cS\cap\{\vx\in\spd:\theta(\vx,\vx^*)\in(0,\pi/2)\}$. If $\nablar f$ satisfies \Cref{assum:smoothness-of-regression-loss}(i), then $f$ is $(\mu,2)$-RGA on $\cS'$ with $\mu = \nu/(2L_1)$; if $\nablar f$ satisfies \Cref{assum:smoothness-of-regression-loss}(ii), then $f$ is $(\mu, 1)$-RGA on $\cS'$ with $\mu = \nu/(4L_2)$. 
\end{lemma}
\begin{proof}
    Because $f(\vx)$ is $\nu$-strongly geodesically convex on a geodesically convex subset on $\spd$, by definition~\cite{boumal2023introduction}, for any $t\in(0,1]$ and $\vv'\in{\rm T}_\vx\spd$ such that ${\rm Exp}_\vx(t\vv')\in\cS$ , it holds
    \begin{align}\label{eq:geo-str-cvx}
        f({\rm Exp}_\vx(t\vv'))\geq f(\vx) + t\la \nablar f(\vx), \vv'\ra + \frac{t^2\nu}{2}\|\vv'\|_2^2.
    \end{align}
    In the equality above, ${\rm Exp}_\vx(t\vv')$ denotes the exponential map from $\vx$ at time $t$ with velocity $\vv'$. On the sphere, the exponential map is simply ${\rm Exp}_\vx(t\vv') = \cos(t\|\vv'\|_2)\vx + \sin(t\|\vv'\|_2)\vv'/\|\vv'\|_2$. 
Fix any $\vx\in\cS' = \cS\cap\{\theta(\vx,\vx^*)\in(0, \pi/2)\}$ and denote $\theta = \theta(\vx,\vx^*)$. Choosing $\vv = \vx^*_{\perp\vx}/\|\vx^*_{\perp\vx}\|_2$ and $\vv' = \theta\vv$, we have ${\rm Exp}_\vx(t\vv') = \cos(t\theta)\vx + \sin(t\theta)\vv$ and ${\rm Exp}_\vx(\vv') = \vx^*$. Since $f(\vx^*)\leq f(\vx)$ for any $\vx\in\cS'$, strong geodesic convexity \eqref{eq:geo-str-cvx} then implies (choosing $t = 1$):
    \begin{align*}
        f(\vx^*)\geq f(\vx) + \theta\nablar f(\vx) \cdot\vv + \frac{\nu^2}{2}\theta^2\;\Rightarrow \nablar f(\vx) \cdot \vv\leq -\frac{\nu}{2}\theta.
    \end{align*}
    Since when $\theta\in(0,\pi/2)$ we have $\sin\theta\leq \theta$, it holds that $\nablar f(\vx) \cdot\vx^* = \sin\theta\nablar f(\vx)\cdot\vv\leq -(\nu/2)\sin^2\theta$. Therefore, when $\nablar f(\vx)$ satisfies \Cref{assum:smoothness-of-regression-loss}(i), we have that $f$ is $(\mu,2)$-RGA on $\cS'$ with $\mu = \nu/(2L_1)$; and when $\nablar f(\vx)$ satisfies \Cref{assum:smoothness-of-regression-loss}(ii), we have that $f$ is $(\mu, 1)$-RGA on $\cS'$ with $\mu = \nu/(4L_2)$. 
\end{proof}

 \end{document}